\documentclass{paper}
\usepackage{fullpage}

\usepackage{amsmath,amsfonts}
\usepackage{algorithmic}
\usepackage{algorithm}
\usepackage{array}
\usepackage[caption=false,font=normalsize,labelfont=sf,textfont=sf]{subfig}
\usepackage{textcomp}
\usepackage{stfloats}
\usepackage{url}
\usepackage{verbatim}
\usepackage{graphicx}

\usepackage[numbers,sort&compress]{natbib}

\usepackage{amssymb}
\usepackage{mathrsfs}
\usepackage{dsfont}
\usepackage{bbm}

\usepackage{multirow}
\usepackage{mathtools}
\mathtoolsset{showonlyrefs}

\usepackage{hyperref}

\usepackage{amsthm}
\makeatletter
\renewenvironment{proof}[1][\proofname]{%
    \par
    \pushQED{\qed}%
    \normalfont
    \topsep6\p@\@plus6\p@\relax
    \trivlist
    \item[\hskip\labelsep\bfseries\sffamily #1\@addpunct{.}]%
    }{%
        \popQED\endtrivlist\@endpefalse
    }
    \makeatother

    \newtheoremstyle
    {mystyle}
    {6pt}
    {6pt}
    {\itshape}
    {}
    {\sffamily\bfseries}
    {.}
    {0.5em}
    {}

    \theoremstyle{mystyle}

    \newtheorem{theorem}{Theorem}
    \newtheorem{lemma}[theorem]{Lemma}
    \newtheorem{proposition}[theorem]{Proposition}
    \newtheorem{corollary}[theorem]{Corollary}

    \theoremstyle{definition}

    \usepackage{xcolor}

    \DeclareMathOperator{\unif}{Unif}

    \newcommand{\ind}{\mathbbm{1}}
    \newcommand\simiid{\overset{iid}{\sim}}
    \newcommand\nb{\mathbf{NB}}
    \newcommand{\dist}{\mathsf{d}}
    \newcommand{\supp}{\operatorname{supp}}

    \def\sA{\mathcal{A}}

    \def\sF{\mathcal{F}}

    \def\sM{\mathcal{M}}
    
    \def\sP{\mathcal{P}}

    \def\sS{\mathcal{S}}

    \def\sX{\mathcal{X}}
    \def\sY{\mathcal{Y}}

    \def\sMexp{\mathcal{M}_\mathrm{exp}}

    \def\tO{\widetilde{O}}

    \newcommand{\R}{\mathbb{R}}
    \newcommand{\C}{\mathbb{C}}
    \def\N{\mathbb{N}}

    \def\cL{\mathscr{L}}
    \def\law{\mathcal{L}}
    \def\schwartz{\mathfrak{S}}
    \def\fourier{\mathscr{F}}

    \def\spn{\operatorname{span}}
    \def\projtimes{\widehat{\otimes}_\pi}
    \def\injtimes{\widehat{\otimes}_\varepsilon}
    \def\ball{\mathbb{B}}
    \def\sphere{\mathbb{S}}

    \newcommand{\towstar}{\xrightarrow{\mathrm{w}^*}}

    \def\sil{V}
    \def\silset{\mathcal{S}}
    \def\rankmeas{\mathcal{R}}
    
    \newcommand\numberthis{\addtocounter{equation}{1}\tag{\theequation}}

\begin{document}

    \title{The Silhouette Operator: Identifiability of Low-Rank Measures from One-Dimensional Projections}

    \author{Robert A. Vandermeulen}

    \maketitle
    \begin{abstract}
        Structured recovery phenomena, such as restricted isometry properties in compressed sensing, have shown that high-dimensional objects can often be reconstructed from remarkably low-dimensional linear measurements.
        This work develops an analogous recovery framework for low-rank signed measures on \(\R^2\), defined here as measures that can be expressed as finite sums of product measures with one-dimensional factors.
        The framework is based on linear operators, termed \emph{silhouette operators}, that map a measure to a fixed finite collection of one-dimensional linear pushforwards.
        The main results show that a suitably chosen collection of \(2k\) projected marginals suffices to identify every compactly supported rank-\(\le k\) signed measure, that this number is optimal, and that the projection directions cannot be chosen arbitrarily.
        The framework is also extended to higher-dimensional sums of product measures by establishing sufficient conditions under which collections of pairwise marginals identify the full model.
        Building on this framework, a computationally efficient estimator, termed \emph{silhouette mixture estimation} (SME), is introduced for constructing a low-rank empirical measure from data by matching its one-dimensional projected marginals to the corresponding empirical marginals in Wasserstein distance.
        When combined with one-dimensional density estimators, SME yields an efficient nonparametric density estimator that performs strongly relative to a range of parametric, nonparametric, and deep-learning baselines in settings of moderate dimension and sample size.
    \end{abstract}
    \begingroup
    \renewcommand{\thefootnote}{}
    \begin{NoHyper}
    \footnotetext{\texttt{rvdm@umich.edu}}
    \end{NoHyper}
    \endgroup
    \section{Introduction}

    High-dimensional problems have become a central focus of modern statistics and machine learning. A perennial challenge in this setting is the \emph{curse of dimensionality}: statistical and computational difficulty often grows rapidly with dimension. Numerous techniques have been proposed to mitigate these difficulties. Among the most widely used and theoretically grounded techniques are those that exploit sparsity in vectors and low-rank structure in matrices or tensors. Within these frameworks, the \emph{restricted isometry property} (RIP) \cite{candes05,recht10} and the closely related notion of restricted strong convexity \cite{negahban11} play central roles. For matrices, an RIP operator \(T : \R^{m\times n} \to \R^\ell\) (with \(\ell\) typically on the order of \(k(m+n)\)) is a linear operator that approximately preserves the geometry of the space of rank-\(\le k\) matrices. Injectivity of \(T\) on rank-\(\le k\) matrices is equivalent, after appropriately rescaling \(T\), to matrix-RIP of order \(2k\); see Appendix \ref{appx:rip-rsc-inj-equiv}. Although the existence of such an embedding is consistent with the intrinsic dimensionality \(\Theta(k(m+n))\) of the model class, it is noteworthy that the embedding can be realized by a linear map.

    This work establishes an analogous phenomenon for a structured class of measures on Euclidean space. In this work, a finite signed measure \(\mu\) on \(\R^2\) is said to have \emph{rank at most \(k\)} if it admits a representation of the form
    \begin{equation}\label{eqn:basic-model-measures-d=2} 
        \mu = \sum_{i=1}^k \mu_{i,1} \times \mu_{i,2},
    \end{equation} 
    where each \(\mu_{i,j}\) is a finite signed measure on \(\R\), and \(\mu_{i,1}\times\mu_{i,2}\) denotes their product measure on \(\R^2\). This notion is directly analogous to matrix rank. In particular, a measure has rank exactly \(k\) if and only if \(k\) is the smallest integer for which it admits a representation of the form \eqref{eqn:basic-model-measures-d=2}. Equivalently, it has rank exactly \(k\) if and only if it admits such a representation for which \(\mu_{1,1},\ldots,\mu_{k,1}\) and \(\mu_{1,2},\ldots,\mu_{k,2}\) are each linearly independent (Proposition~\ref{prop:rank-characterization}).

    The core structural contribution of this work, established in Theorems~\ref{thm:main} and \ref{thm:weakstar-homeo}, is that a broad class of rank-\(\le k\) measures, including all compactly supported and, more generally, all sub-Gaussian measures, can be uniquely identified from linear pushforwards to \(\R\) along a suitably chosen collection of \(2k\) vectors \(v_1,\ldots,v_{2k}\in\R^2\). Given such a collection, the map sending a measure to the corresponding tuple of projected marginals will be referred to as a \emph{silhouette operator}, and the resulting tuple as the \emph{silhouette} of the measure. For probability measures, the result says that the joint distribution \(\mu\) of \([X,Y]^\top\) is uniquely determined by the distributions of \(2k\) suitably chosen linear combinations of \(X\) and \(Y\). Writing \(\law\) for the law of a random variable, these projected distributions are
    \[
        \law(\langle v_1,[X,Y]^\top\rangle),
        \ldots,
        \law(\langle v_{2k},[X,Y]^\top\rangle).
    \]
    Figure~\ref{fig:silhouette-illustration} illustrates this construction for \(k=2\).

    This number of projections is optimal: Theorem~\ref{thm:2k-necessary} shows that no collection of \(2k-1\) projection directions identifies all rank-\(\le k\) signed measures. Even with \(2k\) projections, the directions cannot be chosen arbitrarily: Proposition~\ref{prop:2k-geometry-necessary} shows that certain symmetric configurations fail to provide identifiability. Moreover, the silhouette operator is a weak\(^*\) homeomorphic embedding on the set of rank-\(\le k\) probability measures supported in a fixed compact set, and more generally on any uniformly total-variation-bounded set of rank-\(\le k\) signed measures supported in that compact set. 

    \begin{figure}[!t]
        \centering
        \includegraphics[width=0.9\textwidth]{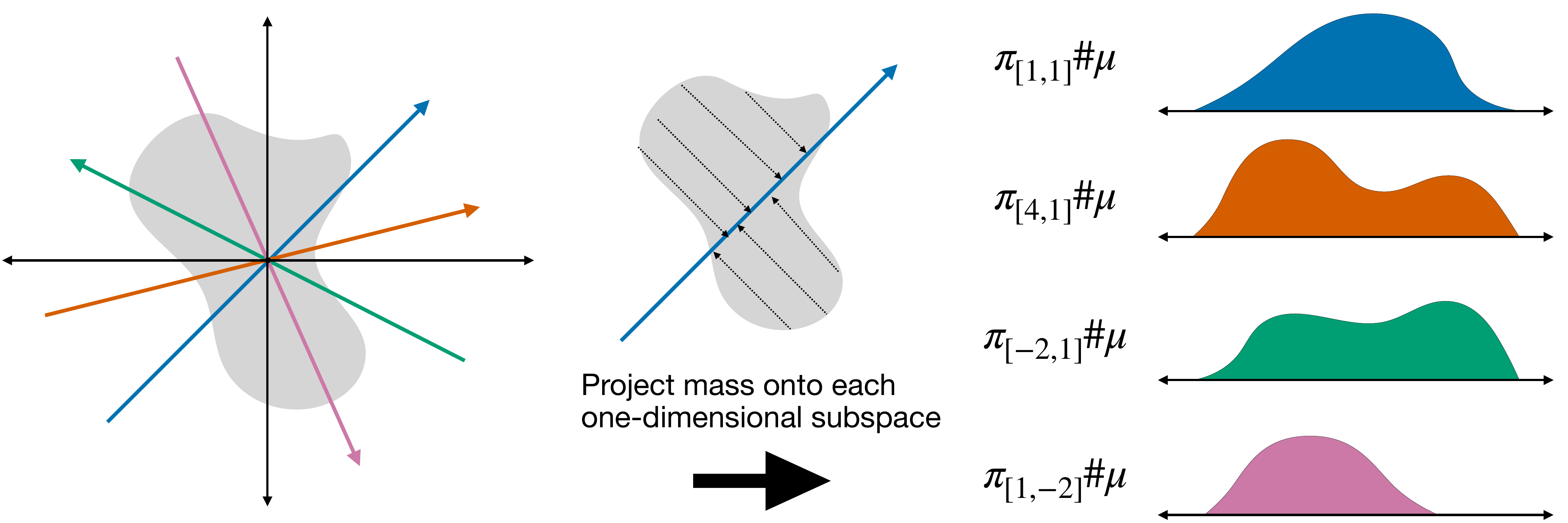}
        \caption{Illustration of a silhouette operator for \(k=2\), corresponding to \(a=-2\) and \(b_1=b_2=1\) in Theorem~\ref{thm:main}. A measure \(\mu\) on \(\R^2\) is mapped to its one-dimensional pushforwards along the resulting four projection directions. When \(\mu\) has rank at most \(2\), these four projected measures uniquely determine \(\mu\).}
        \label{fig:silhouette-illustration}
    \end{figure}

    These results provide a mathematically principled basis for exploiting the statistical and algorithmic advantages of one-dimensional methods in multi-dimensional settings.
    For example, convergence in Wasserstein distance of the projected measures implies convergence in Wasserstein distance on $\R^2$. 
    Projection-based approaches of this type underlie the sliced Wasserstein distance and related methods \cite{bonnotte13,bonneel15,nadjahi20,nguyen21distributional}; this work shows that a finite collection of such projections preserves the full weak$^*$ topology on a rich structured class of measures.

    The silhouette framework is extended to higher dimensions by identifying a measure from a collection of its two-dimensional marginals, each of which can be determined through its silhouette operator. For example, one natural approach to inferring the joint distribution of \((X_1,X_2,X_3)\) is to couple the silhouettes of \((X_1,X_2)\) and \((X_2,X_3)\) through their shared coordinate. To formalize this approach, Section~\ref{sec:extending-2d} establishes measure-theoretic identifiability results for models of the form
    \begin{equation}\label{eqn:basic-model-measures-d=general}
        \mu = \sum_{i=1}^k w_i \prod_{j=1}^d \mu_{i,j},
    \end{equation}
    where each \(\mu_{i,j}\) is a probability measure and the weights \(w_1,\ldots,w_k\) are nonnegative and sum to one. Probability measures of this form will be referred to as \emph{mixture-of-products models} (also known in various settings as multiview models, nonparametric mixture models, or naive Bayes mixtures). Here \(\prod_{j=1}^d \mu_j\) denotes the usual product measure on the product \(\sigma\)-algebra. It is shown that these models can be identified from collections of two-dimensional marginals under assumptions adapted from standard identifiability conditions for matrix and tensor models (including linear independence and joint irreducibility \cite{scott13,blanchard14}, with the latter serving as a measure-theoretic analogue of the anchor-word assumption \cite{arora12}).

    In nonparametric density estimation on \(\R^d\), mixture-of-products models \eqref{eqn:basic-model-measures-d=general} are known to exhibit favorable statistical behavior under suitable regularity assumptions, particularly in high-dimensional regimes \cite{hall03,hall05,vandermeulen21,kwon21,amiridi22I,chhor25,lyu25}.
However, practical estimation of these models remains challenging in the continuous-density setting. Existing estimation methods include approaches based on matrix or tensor factorization. A natural example is low-rank factorization of a joint histogram, where the bin resolution must be selected and finer resolutions lead to larger factorization problems. Other approaches repeatedly estimate component densities during model fitting, which can also entail substantial computational cost \cite{benaglia09,vandermeulen25pmode}. These approaches are discussed further in Section~\ref{ssec:related-mixture-products}.

Building on the silhouette-based identifiability results above, this work develops a concrete algorithm, termed \emph{silhouette mixture estimation} (SME), for estimating low-rank empirical measures of the form \eqref{eqn:basic-model-measures-d=general}. SME separates fitting the empirical mixture structure from estimating component densities, so choices such as density-estimation bandwidths need not enter the mixture-fitting objective. Several downstream uses of the resulting representations are discussed, with density estimation being a particularly natural application.

SME is computationally efficient because its objective is built from univariate Wasserstein distances that can be evaluated rapidly and in parallel. Combined with standard univariate density estimators, the approach appears particularly effective in moderate-dimensional settings with limited sample sizes. Empirically, the resulting estimator performs strongly on benchmark datasets, outperforming a range of parametric, nonparametric, and deep-learning baselines.

    Taken together, the theoretical and algorithmic developments presented above establish a formal connection between sliced methods and nonparametric mixture models of the form \eqref{eqn:basic-model-measures-d=general}. The silhouette operator provides a tractable linear transformation from multivariate measures to collections of one-dimensional measures, enabling computation and estimation to be carried out in the projected space without sacrificing identifiability of the underlying low-rank measure. Although these results bear a close structural resemblance to phenomena in compressed sensing, the proofs are wholly independent of compressed-sensing arguments and use original techniques introduced in this work. 

    Section~\ref{sec:sharpness} establishes the sharpness of the projection count for identifiability and shows that, even with \(2k\) projections, the directions cannot be chosen arbitrarily. It further shows that the tail and boundedness assumptions cannot in general be omitted from the identifiability and homeomorphism results, respectively. Open questions and directions for further investigation are also identified.

    \section{Main Theorems}\label{sec:main-theorems}
    This section establishes the core injectivity and topological results for the silhouette operator on classes of low-rank measures in $\R^2$. These results form the foundation for the subsequent theoretical and algorithmic developments. Necessary notation and definitions are introduced first.

    \subsection{Notation and Definitions}\label{ssec:notation-definitions}
    For a natural number \(n\), let \([n] \triangleq \{1,\ldots,n\}\). For a measurable space \((\Omega,\mathcal{F})\), write \(\mathcal{M}(\Omega)\) for the Banach space of finite signed measures on \((\Omega,\mathcal{F})\), equipped with the total variation norm \(\|\mu\|\triangleq|\mu|(\Omega)\). When $\Omega$ is a topological space (e.g., $\R^d$), measures are understood to be finite Borel measures.
    Throughout, a measure \(\mu\) on \(\R^d\) is said to have a density \(f_\mu\) if
    \[
        \mu(A)=\int_A f_\mu(x)\,dx
    \]
    for every measurable set \(A\).
    If \(\mu\) is a measure on \(\Omega\) and \(f : \Omega \to \Omega'\) is a measurable function into another measurable space \((\Omega',\mathcal{F}')\), the pushforward of \(\mu\) under \(f\), denoted \(f\#\mu\), is defined by
    \[
        f\#\mu(A) \triangleq \mu\bigl(f^{-1}(A)\bigr), \qquad A \in \mathcal{F}'.
    \]
    For fixed \(f\), the map \(\mu\mapsto f\#\mu\) is linear.

    Denote by \(\sMexp(\R^d) \subset \mathcal{M}(\R^d)\) the class of measures whose Laplace transform exists and is finite on a neighborhood of the origin. Measures with compact support or sub-Gaussian tails lie in \(\sM_{\exp}(\mathbb{R}^d)\), so this class includes many standard nonparametric settings. A superscript \(+\), as in \(\sM^{+}(\R)\), denotes the cone of finite nonnegative (unsigned) measures. For a vector \(a = [a_1,\ldots,a_d]^{\top} \in \R^d\), define \(\pi_a(x) \triangleq \langle a,x\rangle\). When convenient, the same map will also be written as \(\pi_{[a_1,\ldots,a_d]}(x)\), suppressing the transpose in the subscript.

    \subsection{Injectivity of the Silhouette Operator}
    Theorem~\ref{thm:main} below gives an explicit family of silhouette operators that are injective on low-rank measures admitting representations whose component marginals belong to \(\sMexp(\R)\). The theorem is stated below in measure-theoretic form. A probabilistic interpretation in terms of the distributions of linear combinations of random variables is given immediately after Corollary~\ref{cor:main-multiview}.

    \begin{samepage}
        \begin{theorem}\label{thm:main}
            Let $k\in\mathbb{N}$ and let $a,b_1,b_2\in\R$ be nonzero with $|a|\neq 1$.
            Let $\mu,\mu'\in\sM(\R^2)$ admit representations
            \begin{equation}\label{eqn:main-target}
                \mu=\sum_{i=1}^k \mu_{i,1}\times\mu_{i,2},
                \qquad
                \mu'=\sum_{i=1}^k \mu'_{i,1}\times\mu'_{i,2},
            \end{equation}
            where $\mu_{i,j},\mu'_{i,j}\in\sMexp(\R)$ for all $i\in[k]$ and $j\in[2]$.
            If
            \begin{align}
                \pi_{[a^m b_1,b_2]}\#\mu &= \pi_{[a^m b_1,b_2]}\#\mu'
                                         && \text{for all } m\in\{1,\ldots,k\}, \label{eqn:silhoulette-vector1}\\
                \pi_{[b_1,a^m b_2]}\#\mu &= \pi_{[b_1,a^m b_2]}\#\mu'
                                         && \text{for all } m\in\{0,\ldots,k-1\}, \label{eqn:silhoulette-vector2}
            \end{align}
            then $\mu=\mu'$.
        \end{theorem}\end{samepage}
        Proofs for all results in this subsection can be found in Section~\ref{sec:proofs-injectivity}. The component marginals (e.g., \(\mu_{i,j}\)) may be the zero measure, so the theorem applies to measures of rank at most \(k\), not only those of rank exactly \(k\).
        When each component marginal \(\mu_{i,j}\) has a density \(f_{\mu_{i,j}}\), the density of \(\mu\) is
        \[
            f_\mu(x,y)=\sum_{i=1}^k
            f_{\mu_{i,1}}(x)f_{\mu_{i,2}}(y).
        \]

        When applying Theorem~\ref{thm:main} in the context of estimation, if \(\mu\) is a target measure with \(\mu_{i,j}\in\sMexp(\R)\) and \(\mu'\) is an estimator, the component marginals \(\mu'_{i,j}\) must also belong to \(\sMexp(\R)\) for equality of their silhouettes to imply \(\mu=\mu'\). The tail condition in the signed setting cannot be substantially weakened: Lemma~\ref{lem:proj-null} shows that the nullspace of any silhouette operator formed from finitely many projection directions contains a nonzero rank-one signed measure whose component marginals have smooth densities that decay faster than any polynomial. The required tail condition can nevertheless be ensured simply by working on a fixed compact interval, such as \([-1,1]\). In the nonnegative setting, however, no corresponding tail condition on the estimator is required: provided \(\mu_{i,j}\in\sMexp^+(\R)\), Corollary~\ref{cor:main-multiview} allows \(\mu'_{i,j}\) to be arbitrary elements of \(\sM^+(\R)\).
        \begin{corollary}\label{cor:main-multiview}
            The conclusion of Theorem~\ref{thm:main} remains valid under the assumptions
            \[
                \mu_{i,j}\in\sMexp^+(\R),
                \qquad
                \mu'_{i,j}\in\sM^+(\R),
            \]
            for all \(i\in[k]\) and \(j\in[2]\).
        \end{corollary}
        For a probabilistic interpretation, suppose that \(\mu\) and \(\mu'\) are probability measures satisfying the assumptions of Corollary~\ref{cor:main-multiview}, and let
        \[
            [X,Y]^\top\sim\mu,
            \qquad
            [X',Y']^\top\sim\mu'.
        \]
        Let \(v_1,\ldots,v_{2k}\) denote the \(2k\) directions in \eqref{eqn:silhoulette-vector1}--\eqref{eqn:silhoulette-vector2}. Then \(\mu=\mu'\) if and only if, for every \(i\in[2k]\), the inner products
        \[
            \langle v_i,[X,Y]^\top\rangle
            \quad\text{and}\quad
            \langle v_i,[X',Y']^\top\rangle
        \]
        have the same distribution.

        The projection count \(2k\) in Theorem~\ref{thm:main} and Corollary~\ref{cor:main-multiview} is optimal. Theorem~\ref{thm:2k-necessary} shows that, for any collection of \(2k-1\) projection directions, there exist two distinct probability measures of mixture-of-products form, each with at most \(k\) components, whose projections agree along every prescribed direction, even when all component marginals have smooth, compactly supported densities. Proposition~\ref{prop:2k-geometry-necessary} complements this projection-count lower bound by showing that the geometry of the directions also matters. For any distinct \(a_1,\ldots,a_k>0\), the \(2k\) symmetrically paired directions \([a_r,1]^\top\) and \([-a_r,1]^\top\), \(r\in[k]\), fail to identify mixture-of-products models with as few as \(\lfloor k/2\rfloor+1\) components. Identifiability also depends on the tail behavior of the component marginals when \(k\ge2\). Indeed, Corollary~\ref{cor:proj-null} shows that, for any finite collection of projection directions, there exist two distinct probability measures of mixture-of-products form, each with two components, whose projections agree along every prescribed direction, even though all component marginals have densities that decay faster than any polynomial.

        If desired, potential tail issues can be mitigated by applying a coordinate-wise transformation that maps \(\R\) into a compact set, such as \(\arctan\). More generally, in \(d\) dimensions, let \(f_1,\ldots,f_d:\R\to\R\) and define \(f:\R^d\to\R^d\) by
        \[
            f(x) = (f_1(x_1), \ldots, f_d(x_d)).
        \]
        Then
        \begin{align}
            f\#\!\left(\sum_{i=1}^k \prod_{j=1}^d \mu_{i,j}\right)
    &= \sum_{i=1}^k f\#\!\left(\prod_{j=1}^d \mu_{i,j}\right) \\
    &= \sum_{i=1}^k \prod_{j=1}^d f_j\# \mu_{i,j},
    \label{eqn:compressed-low-rank}
        \end{align}
        so coordinate-wise transformations preserve the sum-of-products structure. In the probabilistic setting, one may therefore use \(f\) to map the data into a compact set, estimate the transformed model in \eqref{eqn:compressed-low-rank}, and, when appropriate, map back via \(f^{-1}\). Notably, when \(f\) is invertible, the linear projections defining the silhouette operator in Theorem~\ref{thm:main} may be replaced by generally nonlinear maps of the form \(\pi_v\circ f\).

        \subsection{Topological and Geometric Properties of the Silhouette Operator}
        The preceding results establish injectivity of the silhouette operator on low-rank measures. Beyond injectivity, it is natural to ask whether, and under what additional conditions, convergence of silhouettes guarantees convergence of the underlying measures. This subsection examines the topological and geometric structure induced by the silhouette operator. Proofs of all results in this subsection can be found in Section~\ref{sec:proofs-core}.

        For $k\in\N$, let $\silset_{2k}$ denote the set of linear operators $\sil:\sM(\R^2)\to\sM(\R)^{2k}$ obtained from Theorem~\ref{thm:main} using the $2k$ projection directions in \eqref{eqn:silhoulette-vector1}–\eqref{eqn:silhoulette-vector2}. Thus, each choice of admissible parameters $(a,b_1,b_2)$ determines an element $\sil\in\silset_{2k}$. For measurable spaces $\Omega_1,\Omega_2$, with $\Omega_1\times\Omega_2$ equipped with the product $\sigma$-algebra, define
        \[
            \rankmeas_{\le k}(\Omega_1,\Omega_2)
            \subseteq
            \sM(\Omega_1\times\Omega_2)
        \]
        to be the set of finite signed measures that can be written in the form
        \[
            \rankmeas_{\le k}(\Omega_1,\Omega_2)
            =
            \left\{
                \sum_{i=1}^{k} \mu_{i,1}\times \mu_{i,2}
                :\;
                \mu_{i,j}\in\sM(\Omega_j)
            \right\}.
        \]    
        This is the class of measures of \emph{rank at most \(k\)}. A measure is said to have \emph{rank exactly \(k\)} if \(k\) is the smallest integer for which it belongs to \(\rankmeas_{\le k}(\Omega_1,\Omega_2)\). By Proposition~\ref{prop:rank-characterization}, this is equivalent to admitting a representation with \(k\) terms in which \(\mu_{1,1},\ldots,\mu_{k,1}\) and \(\mu_{1,2},\ldots,\mu_{k,2}\) are each linearly independent.

        To study continuity properties of the silhouette operator, it is necessary to specify the ambient topologies. For a measurable space \(\Omega\), the product space \(\sM(\Omega)^\ell\) is equipped with the standard product topology. When \(\sM(\Omega)\) carries its weak\(^*\) topology, the induced product topology on \(\sM(\Omega)^\ell\) will also be referred to as the weak\(^*\) topology. This convention is natural and consistent with the identification of duals of finite direct sums; see \cite[Proposition~5.5]{fabian01}. When each component is instead equipped with the total variation norm, the induced product topology will be referred to as the norm topology.

        The next result shows that, for compact \(K_1,K_2\subset\R\), silhouette operators act as weak\(^*\) homeomorphisms on total-variation-bounded subsets of \(\rankmeas_{\le k}(K_1,K_2)\) and, in the language of compressed sensing, can be viewed as ``restricted weak\(^*\) homeomorphisms.'' For a normed space \(X\), let \(\ball_X\) denote its closed unit ball.
        \begin{theorem}\label{thm:weakstar-homeo}
            Let \(K_1,K_2\subset\R\) be compact, let \(k\in\N\), let
            \(\sil\in\silset_{2k}\), and fix \(r>0\). Then the restriction of \(\sil\) to
            \[
                \rankmeas_{\le k}(K_1,K_2)
                \cap r\ball_{\sM(K_1\times K_2)}
            \]
            is a weak\(^*\)--weak\(^*\) homeomorphism onto its image.
        \end{theorem}
        The boundedness assumption cannot in general be omitted. Theorem~\ref{thm:boundedness-necessary} shows that, for any finite collection of projection directions, the associated map is not a weak\(^*\)--weak\(^*\) homeomorphic embedding of \(\rankmeas_{\le1}([0,1],[0,1])\), regardless of how many directions are included. This remains true even when the domain is restricted to measures with \(C^\infty\) densities.

        Theorem~\ref{thm:weakstar-homeo} has an important practical consequence for probability measures. Every probability measure has total variation norm one, and weak\(^*\) convergence of probability measures coincides with convergence in distribution. Consequently, let \(\mu\) and \((\hat{\mu}_i)_{i=1}^\infty\) be probability measures in \(\rankmeas_{\le k}(K_1,K_2)\). If each component of \(\sil(\hat{\mu}_i)\) converges in distribution to the corresponding component of \(\sil(\mu)\), then \(\hat{\mu}_i\) converges in distribution to \(\mu\).

        At the population level, injectivity also admits a familiar optimization formulation. If \(\dist\) is any discrepancy on \(\sM(\R)^{2k}\) that vanishes exactly when its arguments agree, then \(\mu\) is the unique rank-\(\le k\) solution of
        \[
            \min_{\nu\in\rankmeas_{\le k}(K_1,K_2)}
            \dist\bigl(\sil(\nu),\sil(\mu)\bigr).
        \]
        A natural estimation strategy is therefore to replace the unknown silhouette \(\sil(\mu)\) by the silhouette of the empirical measure and minimize the resulting loss over low-rank measures. The key computational advantages of the silhouette framework stem from its reduction of high-dimensional estimation problems to collections of one-dimensional projections, the central computational principle underlying sliced methods. Many one-dimensional statistical procedures, including Wasserstein distance computations, admit efficient \(\tO(n)\) algorithms. Moreover, computations associated with different projection directions can be performed independently, making the resulting procedures naturally parallelizable. Section~\ref{sec:application-algorithm} develops a practical estimation framework based on this strategy.

        The proof of Theorem~\ref{thm:weakstar-homeo} relies on a structural result for finite-rank tensor products in Banach spaces, established later as Theorem~\ref{thm:rank-closed}, which gives the weak\(^*\) closedness of rank-\(\le k\) measures needed here. This result extends related work in \cite{petitjean25} and may be of independent interest.

        Silhouette operators also satisfy a restricted isometry-type property on finite-dimensional subspaces. In particular, after rescaling, they preserve distances up to a uniform distortion factor on the intersection of any finite-dimensional subspace with the class of rank-\(\le k\) measures. This is captured in the following proposition. In what follows, \(\sM(\R)^{2k}\) may be equipped with any product norm induced by the total variation norm.
        \begin{proposition}\label{prop:rip-span}
            Let \(K_1,K_2 \subset \R\) be compact, let $k\in \N$, let $S$ be a finite-dimensional linear subspace of $\sM(K_1\times K_2)$, and let \(\sil \in \silset_{2k}\). Then there exist $\gamma >0 $ and $\delta \in [0,1)$ such that
            \[
                (1-\delta)\,\|\mu-\mu'\|
                \;\le\;
                \|\gamma \sil(\mu-\mu')\|
                \;\le\;
                (1+\delta)\,\|\mu-\mu'\|,
            \]
            for all $\mu,\mu' \in S\cap \rankmeas_{\le k}(K_1, K_2)$.
        \end{proposition}
        An explicit parallel with matrix RIP can be drawn when the component marginals are restricted to finite-dimensional subspaces. Let \(S_j\subset\sM(K_j)\) have dimension \(d_j\), and let \(S\) be the span of the product measures \(\eta_1\times\eta_2\) with \(\eta_j\in S_j\). After choosing bases for \(S_1\) and \(S_2\), the space \(S\) can be identified with \(\R^{d_1\times d_2}\), with \(S\cap\rankmeas_{\le k}(K_1,K_2)\) corresponding to matrices of rank at most \(k\). In these coordinates, the inequality in Proposition~\ref{prop:rip-span} takes the same form as the classical RIP for low-rank matrices.

        \section{Higher-Dimensional Extensions of the Silhouette Framework} \label{sec:extending-2d}
        The previous section established that, in two dimensions, a low-rank measure is uniquely determined by its silhouette. The present section extends this framework to higher-dimensional mixture-of-products models of the form \eqref{eqn:basic-model-measures-d=general}, where \(w_i>0\), \(\sum_{i=1}^k w_i=1\), and each \(\mu_{i,j}\) is a probability measure. The section identifies conditions under which such measures can be recovered from collections of pairwise marginals of the form
        \[
            \sum_{i=1}^k w_i\, \mu_{i,j_1} \times \mu_{i,j_2}.
        \]
        The statements are formulated for general measurable spaces \(\Omega\), though the primary case of interest remains \(\Omega = \R\). Proofs of all results in this section can be found in Section~\ref{sec:proofs-d>2}. 

        To express component identifiability as an exact equality without imposing an arbitrary ordering on the components, finite mixtures of product measures are represented as probability measures on the space of \(d\)-tuples of component marginals. For a measurable space \((\Omega,\mathcal{F})\), let \(\sP(\Omega)\) denote the space of probability measures on \(\Omega\). The space
        \(
        \sP(\Omega)^d
        \)
        is equipped with the discrete \(\sigma\)-algebra, so that all singleton sets are measurable. An element \((\mu_1,\ldots,\mu_d)\in\sP(\Omega)^d\) corresponds to the collection of marginals defining a product component \(\prod_{j=1}^d\mu_j\). 

        For finite linear combinations of point masses on \(\sP(\Omega)^d\), define the linear map \(R\) by
        \[
            R\left(
                \sum_{i=1}^k
                a_i\delta_{(\mu_{i,1},\ldots,\mu_{i,d})}
            \right)
            \triangleq
            \sum_{i=1}^k
            a_i\prod_{j=1}^d\mu_{i,j}.
        \]
        For \(J\subseteq[d]\), write its elements in increasing order as
        \(j_1<\cdots<j_\ell\) and define the linear map \(R_J\) by
        \[
            R_J\left(
                \sum_{i=1}^k
                a_i\delta_{(\mu_{i,1},\ldots,\mu_{i,d})}
            \right)
            \triangleq
            \sum_{i=1}^k
            a_i\prod_{r=1}^{\ell}\mu_{i,j_r}.
        \]

        The identifiability results below use linear independence together with two additional notions for collections of probability measures: irreducibility and singularity. Let \(\mu_1,\ldots,\mu_k\) be probability measures on a common measurable space. A measure \(\mu_1\) is said to be \emph{irreducible} with respect to \(\mu_2,\ldots,\mu_k\) if the following holds: whenever \(\sum_{i=1}^k w_i \mu_i\) is a probability measure for some real coefficients \(w_i\), it must be the case that \(w_1 \ge 0\). Equivalently, \(\mu_1\) is irreducible with respect to the others if, for every \(w>0\), the signed measure \(\sum_{i=2}^k \mu_i - w \mu_1\) has a nonzero negative part. If each \(\mu_i\) is irreducible with respect to all the others, the collection is called \emph{jointly irreducible} \cite{scott13,blanchard14}. Finally, \(\mu_1\) will be called \emph{singular} with respect to \(\mu_2,\ldots,\mu_k\) if there exists a measurable set \(A\) such that \(\mu_1(A)=1\) and \(\mu_i(A)=0\) for every \(i\in\{2,\ldots,k\}\).

        Irreducibility is a measure-theoretic analogue of the ``support condition'' or ``anchor-word'' assumption in nonnegative matrix factorization \cite{donoho04,arora12b}. For example, \(\mu_1\) is irreducible if there exists a measurable set \(A\) such that \(\mu_1(A)>0\) while \(\mu_i(A)=0\) for every \(i\in\{2,\ldots,k\}\). On a finite discrete space, where measures can be identified with probability vectors, this condition is also necessary. Joint irreducibility is stronger than linear independence, while singularity is stronger than irreducibility.

        The results below give structural conditions under which the component marginals and weights in the mixture-of-products model \eqref{eqn:basic-model-measures-d=general} can be recovered from a collection of pairwise marginals. Such a collection can be encoded by a graph \(G=([d],E)\), where \(\{j_1,j_2\}\in E\) indicates that the corresponding pairwise marginal is included in the collection. Within the silhouette framework, each of these marginals can itself be identified from its silhouette. The weaker problem of recovering the mixture \(\mu\) itself from such pairwise marginals, without identifying its component marginals and weights, remains open.

        Outside of trivial cases, the marginal graph must at least be connected to guarantee recovery of \(\mu\); if the graph is disconnected, at least two groups of coordinates remain uncoupled and the full model cannot in general be identified. The complete graph yields the strongest constraints but requires estimating \(\Theta(d^2)\) pairwise marginals, whereas a connected graph may have as few as \(d-1\) edges. Thus, depending on the choice of marginal graph, the number of pairwise marginals that must be estimated ranges from linear to quadratic in \(d\), with a corresponding impact on computational cost in practical implementations.

        Theorem~\ref{thm:li-ident} shows that, when the true component marginals are linearly independent along each coordinate, identifying them along a single coordinate suffices to identify the full model from any connected graph of pairwise marginals. The linear independence assumption only needs to hold for the true model and does not need to be imposed as a constraint during estimation. Related propagation phenomena under linear independence assumptions also appear in finite-alphabet tensor models for recovering high-dimensional probability distributions from low-order marginals \cite{kargas18}.

        \begin{samepage}\begin{theorem}\label{thm:li-ident}
            Let $(\Omega,\sF)$ be a measurable space, let $k,d\in\N$ and $d\ge 2$. 
            For $i\in[k]$ and $j\in[d]$, let $\mu_{i,j},\mu'_{i,j}\in\sP(\Omega)$, and let 
            $w_i>0$, $w'_i\ge0$ satisfy $\sum_{i=1}^k w_i=\sum_{i=1}^k w'_i=1$.

            Assume that for every $j\in[d]$, the measures 
            $\mu_{1,j},\ldots,\mu_{k,j}$ are linearly independent.

            Let $G=([d],E)$ be a connected graph. Suppose that
            \[
                R_{e}\!\left(\sum_{i=1}^k w_i\,\delta_{(\mu_{i,1},\ldots,\mu_{i,d})}\right)
                \;=\;
                R_{e}\!\left(\sum_{i=1}^k w'_i\,\delta_{(\mu'_{i,1},\ldots,\mu'_{i,d})}\right)
            \]
            for every $e\in E$.

            Suppose further that for some $d'\in[d]$,
            \[
                \{\mu_{1,d'},\ldots,\mu_{k,d'}\}
                =
                \{\mu'_{1,d'},\ldots,\mu'_{k,d'}\}.
            \]
            Then
            \begin{equation}\label{eqn:li-ident-consequence}
                \sum_{i=1}^k w_i\,\delta_{(\mu_{i,1},\ldots,\mu_{i,d})}
                \;=\;
                \sum_{i=1}^k w'_i\,\delta_{(\mu'_{i,1},\ldots,\mu'_{i,d})}.
            \end{equation}
        \end{theorem}\end{samepage}

        Each summand \(w_i\,\delta_{(\mu_{i,1},\ldots,\mu_{i,d})}\) in \eqref{eqn:li-ident-consequence} corresponds to the weighted product component \(w_i\prod_{j=1}^d\mu_{i,j}\) of the associated mixture-of-products model. Thus, \eqref{eqn:li-ident-consequence} states not only that the resulting mixtures \(\mu\) and \(\mu'\) are equal (by applying \(R\) to both sides), but that their underlying components and weights agree up to a permutation of indices, i.e., the model is identifiable. The assumption that \(\mu_{1,j},\ldots,\mu_{k,j}\) are linearly independent for each \(j\in[d]\) is a standard and relatively mild structural condition \cite{allman09}; for example, collections of one-dimensional Gaussian measures with distinct parameter pairs are linearly independent \cite{yakowitz68}.

        Theorem~\ref{thm:li-ident} shows that once the components along one coordinate have been recovered, the remaining components can be identified by propagating through the marginal graph. It does not, however, show how to obtain that initial set of components from the observed pairwise marginals; they must already be known or have been recovered by some other means. For example, if the component marginals along that coordinate are Gaussian, this initial set could be estimated using standard univariate mixture-model methods. In a fully nonparametric setting, it is desirable for uniqueness to be intrinsic---namely, that any factorization consistent with an observed pairwise marginal must coincide with the true one, without imposing structural conditions on the candidate factorization. The next result shows that joint irreducibility of the true component families provides this property for a pair of coordinates.

        \begin{proposition}\label{prop:ji-ident}
            Let $(\Omega,\sF)$ be a measurable space and let $k\in\N$. 
            For $i\in[k]$ and $j\in[2]$, let $\mu_{i,j},\mu'_{i,j}\in\sP(\Omega)$, and let
            $w_i>0$, $w'_i\ge 0$ satisfy $\sum_{i=1}^k w_i=\sum_{i=1}^k w'_i=1$.

            Assume that the families
            \(\mu_{1,1},\ldots,\mu_{k,1}\) and
            \(\mu_{1,2},\ldots,\mu_{k,2}\)
            are jointly irreducible.

            If
            \[
                R\!\left(\sum_{i=1}^k w_i\,\delta_{(\mu_{i,1},\mu_{i,2})}\right)
                \;=\;
                R\!\left(\sum_{i=1}^k w'_i\,\delta_{(\mu'_{i,1},\mu'_{i,2})}\right),
            \]
            then
            \[
                \sum_{i=1}^k w_i\,\delta_{(\mu_{i,1},\mu_{i,2})}
                \;=\;
                \sum_{i=1}^k w'_i\,\delta_{(\mu'_{i,1},\mu'_{i,2})}.
            \]
        \end{proposition}

        The previous proposition identifies a setting in which the true mixture components are uniquely determined from a single pairwise marginal. In conjunction with Theorem~\ref{thm:li-ident}, identifiability on one pair of coordinates can then propagate through the connected marginal graph to recover the full model. 

        A limitation of Proposition~\ref{prop:ji-ident} is that joint irreducibility must hold across every component for the same pair of coordinates. The next theorem shows that, under a stronger singularity assumption, a single component can instead be identified when it is separated from all the others along a single pair of coordinates, and its identification can then be propagated through the remaining pairwise marginals.

        \begin{theorem}\label{thm:singular-ident}
            Let $(\Omega,\sF)$ be a measurable space and let $k,d\in\mathbb{N}$ with $d\ge 2$. For $i\in[k]$ and $j\in[d]$, let $\mu_{i,j},\mu'_{i,j}\in\sP(\Omega)$, and let $w_i>0$, $w'_i\ge 0$ satisfy $\sum_{i=1}^k w_i=\sum_{i=1}^k w'_i=1$.

            Assume that for every $j\in[d]$, the measures $\mu_{1,j},\ldots,\mu_{k,j}$ are linearly independent, and that for $j\in [2]$ the measure $\mu_{1,j}$ is singular with respect to $\mu_{2,j},\ldots,\mu_{k,j}$.

            Let $G=([d],E)$ be a graph such that $\{1,2\}\in E$ and every $j>2$ is adjacent to at least one of $1$ or $2$.

            If
            \[
                R_{e}\!\left(\sum_{i=1}^k w_i\,\delta_{(\mu_{i,1},\ldots,\mu_{i,d})}\right)
                =
                R_{e}\!\left(\sum_{i=1}^k w'_i\,\delta_{(\mu'_{i,1},\ldots,\mu'_{i,d})}\right)
            \]
            for every $e\in E$, then
            \[
                \sum_{i=1}^k
                w'_i\,
                \delta_{(\mu'_{i,1},\ldots,\mu'_{i,d})}
                \bigl(\{(\mu_{1,1},\ldots,\mu_{1,d})\}\bigr)
                = w_1.
            \]
        \end{theorem}
        The conclusion of Theorem~\ref{thm:singular-ident} implies that any representation of the form
        \(
        \sum_{i=1}^k w'_i \prod_{j=1}^d \mu'_{i,j}
        \)
        consistent with the observed pairwise marginals must assign total weight \(w_1\) to the component
        \(
        \prod_{j=1}^d \mu_{1,j}.
        \)
        Theorem~\ref{thm:singular-ident} is most relevant when a rich collection of pairwise marginals, such as the complete graph, is available and, for each component, there exist at least two coordinates along which the singularity condition holds. This condition need not be imposed on the candidate factorization; rather, the theorem shows that any such factorization consistent with the observed pairwise marginals must contain the identified component.

        The conditions above provide identifiability guarantees in specific regimes, but are unlikely to be exhaustive. In particular, identifiability may emerge from interactions among several structural properties of the components and the connectivity pattern of the marginal graph that are not captured by any single explicit condition, such as irreducibility or singularity.
        \section{Related Work}
        This work lies at the intersection of several areas. The two primary connections are to mixture-of-products models and restricted isometry-type conditions arising in compressed sensing. A final subsection discusses several additional connections, including projection-based methods, optimal transport, and inverse problems.

        \subsection{Mixture-of-Products Models}\label{ssec:related-mixture-products}
        Mixture-of-products models are statistical models of the form \eqref{eqn:basic-model-measures-d=general}. They admit a natural interpretation in which the \(d\) covariates become conditionally independent after conditioning on an unobserved discrete random variable with distribution \((w_1,\ldots,w_k)\). Such models are also naturally viewed as nonparametric mixture models. When the component measures are supported on finite sets, \eqref{eqn:basic-model-measures-d=general} reduces to a low-rank nonnegative tensor model, connecting this setting to the literatures on tensor factorization, latent variable models, and topic modeling \cite{allman09,anandkumar14,kargas18,tang26}. The discussion first focuses on models whose component marginals \(\mu_{i,j}\) have densities.

        Early work on mixture-of-products models in the continuous nonparametric setting includes \cite{hall03,hall05}, where consistent estimators for the mixture components were introduced. These works also observed that the product structure can yield improved statistical rates. Since then, a number of estimation methods for this model class have been proposed \cite{benaglia09,song13,song14,bonhomme16,kargas19,vandermeulen21,kwon21,amiridi22I,chhor25,lyu25}, many of which are based on matrix or tensor factorization techniques.

        Much of the literature treats \eqref{eqn:basic-model-measures-d=general} as a nonparametric mixture model, with a corresponding focus on identifiability. A large body of work studies identifiability of such models under structural conditions on the collections \(\mu_{1,j},\ldots,\mu_{k,j}\), typically imposing conditions based on linear independence or on stronger forms of separation \cite{allman09,tahmasebi18,kargas18,kargas19,lyu25}. At the same time, the product structure can be exploited to improve statistical rates. In \cite{vandermeulen21}, it was shown that there exists an estimator achieving rate $\tO(n^{-1/3})$ in $L^1$ for densities of this form, assuming the component marginals $\mu_{i,j}$ admit Lipschitz continuous densities. That work also considers related Tucker-style factorizations for multivariate densities together with associated sample complexity bounds. The \(\tO(n^{-1/3})\) rate matches, up to logarithmic factors, the minimax rate for one-dimensional Lipschitz continuous density estimation, whereas general \(d\)-dimensional Lipschitz continuous densities have minimax rate \(O(n^{-1/(2+d)})\); more generally, \(\beta\)-H\"older continuous densities have minimax rate \(O(n^{-\beta/(2\beta+d)})\) \cite{stone80,stone82,tsybakov08}. This indicates that the product structure can mitigate the curse of dimensionality. 

        More recently, \cite{chhor25} proposed a practical estimator in the two-dimensional setting based on a multiscale decomposition of a histogram estimator combined with nuclear-norm regularization via singular value decomposition. That estimator achieves rate $\tO(n^{-\beta/(2\beta+1)})$ for $\beta$-H\"older densities of the above form with $0<\beta\le1$, which, up to logarithmic factors, matches the minimax rate for one-dimensional density estimation. In \cite{lyu25}, mixture-of-products density estimation was studied in the general $d$-dimensional setting, where the same $\tO(n^{-\beta/(2\beta+1)})$ rate is achieved for densities of this form. Under linear independence assumptions similar to those used in Section~\ref{sec:extending-2d}, that work also establishes recovery guarantees for the mixture components and analyzes how the resulting sample complexity depends on incoherence properties of the component marginals. 

        Every a.e.-continuous density admits a countable mixture-of-products representation with smooth Lipschitz component marginals \cite{vandermeulen23}. The resulting non-negative Lipschitz spectrum quantifies density complexity through the decay of the mixture weights and the growth of the marginal Lipschitz constants. Because such a spectrum exists throughout this broad class, the resulting theory yields finite-sample bounds and, under suitable decay and growth conditions, dimension-independent convergence rates without requiring an exact finite mixture-of-products representation. This countable representation also places Theorem~\ref{thm:main} in a natural relationship with countable-projection versions of the Cramér--Wold theorem \cite{renyi52,gilbert55,cuesta-albertos07}. For each finite \(k\), Theorem~\ref{thm:main} shows that \(2k\) projections suffice to identify a rank-\(\le k\) measure. As \(k\) grows, both the number of mixture components and the number of required projections become countable, paralleling the fact that, under suitable regularity conditions, countably many projections suffice to identify a general probability distribution.

        A closely related setting arises in grouped-observation models of the form
        \begin{equation}
            \sum_{i=1}^k w_i\,\mu_i^{\times d},
            \label{eqn:grouped-model}
        \end{equation}
        which may be viewed as a symmetric specialization of \eqref{eqn:basic-model-measures-d=general}, where the component measures \(\mu_i\) may themselves be arbitrary probability measures on a general measurable space. Identifiability and estimation for such models have been studied in \cite{teicher63,anandkumar14,rabani14,vandermeulen19,ritchie21,wei22,vandermeulen24grouped} and are closely connected to topic models such as latent Dirichlet allocation (LDA) \cite{LDA}. The identifiability results for models of the form \eqref{eqn:grouped-model} are directly related to corresponding results for LDA in \cite{do25}.

        \subsection{Restricted Isometry and Low-Rank Matrix Recovery}
        Restricted isometry and related properties play a central role in compressed sensing and low-rank matrix recovery, where they provide geometric guarantees for stable reconstruction from linear measurements. For a vector $x \in \R^n$, let $\|x\|_0$ denote its sparsity, i.e., the number of nonzero entries of $x$. 

        The original formulation of the restricted isometry property (RIP) \cite{candes05} concerns sparse vectors. More specifically, a matrix $A \in \R^{m\times n}$ is said to satisfy the \emph{vector-RIP} of order $k$ if there exists $\delta \in [0,1)$ such that for all $x \in \R^n$ with $\|x\|_0 \le k$,
        \begin{equation}\label{eqn:rip-vector}
            (1-\delta)\|x\|_2^2 \;\le\; \|Ax\|_2^2 \;\le\; (1+\delta)\|x\|_2^2.
        \end{equation}

        Subsequent work extended this framework to low-rank matrix recovery. In this setting, a linear operator $T : \R^{m \times n} \to \R^\ell$ is said to satisfy the \emph{matrix-RIP} \cite{recht10} of order $k$ if there exists $\delta \in [0,1)$ such that for all $X \in \R^{m\times n}$ with rank at most $k$,
        \begin{equation}\label{eqn:rip-matrix}
            (1-\delta)\|X\|_F^2 \;\le\; \|T(X)\|_2^2 
            \;\le\; (1+\delta)\|X\|_F^2.
        \end{equation}
        Vector-RIP can hold with \(m\) much smaller than \(n\), and matrix-RIP with \(\ell\) much smaller than \(mn\). In both cases, the measurement map substantially reduces dimension while approximately preserving norms on the restricted class. The framework has also been extended to higher-order objects through notions such as tensor-RIP \cite{rauhut17}.

        Restricted strong convexity is a closely related notion, which appears in several forms in the literature. Following the operator formulation in \cite{negahban11}, a linear operator \(T:\R^{m\times n}\to\R^\ell\) is said to satisfy restricted strong convexity (RSC) on a given class of matrices if there exists \(\kappa>0\) such that, for every \(X\) in that class,
        \[
            \|T(X)\|_2^2 \;\ge\; \kappa\|X\|_F^2.
        \]
        The normalization in \cite{negahban11} is absorbed into \(\kappa\).
        Injectivity of \(T\) on rank-\(\le k\) matrices is equivalent to restricted strong convexity on rank-\(\le 2k\) matrices and, after rescaling \(T\), to matrix-RIP of order \(2k\); see Appendix~\ref{appx:rip-rsc-inj-equiv} for a proof.

        RIP and closely related conditions have found numerous applications and are central to the field of \emph{compressed sensing} \cite{donoho06,candes08}, which emerged as a major development in signal processing in the 2000s. A fundamental question in this area is the construction of operators satisfying these properties. A canonical approach to establishing vector-RIP \eqref{eqn:rip-vector} is to use random matrices with independent, appropriately scaled, zero-mean Gaussian entries, which satisfy RIP with high probability for sufficiently large \(m\), depending on \(n\), \(k\), and \(\delta\) \cite{candes05,baraniuk08}. To prove this guarantee, one shows, via concentration inequalities and covering numbers, that such a matrix is an approximate isometry on any fixed \(k\)-dimensional subspace with failure probability exponentially small in \(m\). Although there are combinatorially many coordinate subspaces corresponding to \(k\)-sparse vectors, a union bound combined with exponential concentration allows one to control all such subspaces simultaneously.

        Gaussian entries are not essential for vector-RIP: appropriately normalized random matrices with independent, mean-zero sub-Gaussian entries also satisfy RIP with high probability \cite{baraniuk08}. Sparse random constructions are also possible, generally at the cost of additional measurements or weaker isometry and recovery guarantees \cite{gilbert10}. Analogous probabilistic constructions extend to matrix-RIP \eqref{eqn:rip-matrix}, with measurement counts near the degrees of freedom of low-rank matrices \cite{recht10}.

        Obtaining comparably strong classical RIP guarantees from explicit deterministic constructions remains challenging. A related difficulty is that deciding whether a given matrix satisfies vector-RIP is NP-hard \cite{banderia13}. The silhouette operator considered here is explicit and deterministic, and its restricted-isometry-type property on finite-dimensional subspaces of measures is established through a fundamentally different approach, without randomization, concentration inequalities, or covering-number arguments.

        \subsection{Other Related Topics}

        Silhouette operators can be viewed as a kind of analogue of RIP operators for low-rank measures rather than low-rank matrices. The framework also connects to several other areas, including projection-based methods, optimal transport, and inverse problems.

        Graphical models provide another way to exploit lower-dimensional structure in density estimation. In a Markov random field, a graph encodes conditional independence relations among the observed coordinates. For tree- and forest-structured models, the joint density is determined by the univariate marginals and the pairwise marginals along the graph edges, leading to density estimators built from one- and two-dimensional marginal estimates \cite{liu11,gyorfi23}. More generally, Markov assumptions can reduce the effective dimension of nonparametric density estimation \cite{vandermeulen24mrf,vandermeulen25mrf,fan25,kwon26}. In the present work, the marginal graph instead specifies the pairwise marginals used to identify the latent mixture-of-products structure, rather than conditional independence relations among the observed coordinates.

        A central computational motivation for the silhouette framework is to reduce the cost of Wasserstein distance computations. By measuring the cost of moving probability mass between locations, Wasserstein distance incorporates the geometry of the underlying space into its notion of distance \cite{villani09}. In contrast, total variation, \(L^2\) distance, and KL divergence do not account for the distance between locations, so even a small spatial shift can produce a large discrepancy.
        However, for empirical measures in dimension \(d\ge2\), exact computation of Wasserstein distance involves a discrete optimal transport problem with \(O(n^2)\) transport variables for datasets of size \(n\), leading to quadratic memory requirements and typically superquadratic running time. In one dimension, by contrast, Wasserstein distance can be computed in \(O(n\log n)\) time via sorting.

        One approach to reducing this computational cost is the \emph{sliced Wasserstein distance} \cite{bonnotte13,bonneel15}. It is constructed by projecting probability measures onto one-dimensional subspaces, computing Wasserstein distances between the resulting projected probability measures, and aggregating these distances over a collection of projections. By the Cram\'er--Wold theorem, sufficiently rich collections of projections can distinguish probability measures; related topological and statistical properties of sliced probability metrics have been studied in \cite{nadjahi20,xi22,weed22}.

        Many works have developed sliced Wasserstein methods for statistical and machine learning tasks \cite{deshpande18,wang21twosample,nguyen21distributional,meunier22,nguyen23selfattention,bing24estimationinferencewassersteindistance}. For mixture modelling, a sliced-Wasserstein approach to fitting Gaussian mixture models was proposed in \cite{kolouri18}. It has been shown that a $d$-dimensional Gaussian mixture model with $m$ components can be identified from a finite number of projections, with $(2m-1)(d^2+d-2)/2+1$ projections sufficing \cite{fraiman25}. Related results establish identifiability of other classes, such as elliptical distributions, from $(d^2+d)/2$ projections \cite{fraiman23}.

        These finite-projection results are conceptually related to the present work because they likewise identify an underlying distribution from a projection-based representation.   When an elliptical distribution has a density, the density has the form \(f(x)\propto g\bigl((x-\mu)^\top\Sigma^{-1}(x-\mu)\bigr)\). When second moments exist, \(\mu\) is its mean, and \(\Sigma\) is taken to be its covariance matrix, with the corresponding scale absorbed into \(g\). Thus, although \(g\) is arbitrary, the finite-projection identification problem principally concerns recovering the finite-dimensional quantities \(\mu\) and \(\Sigma\). Once these are known, a single projection determines \(g\). The count in \cite{fraiman23} reflects this: \((d^2+d)/2\) is the dimension of the space of symmetric \(d\times d\) matrices, while \(\mu\) can be recovered from a linearly independent subcollection of \(d\) projections.

        In contrast, the nonparametric freedom in the present work resides in its rank-one factors, rather than in a single shared one-dimensional profile combined with finite-dimensional location and shape parameters. For the silhouette operator, \(2k\) projections suffice to identify a rank-\(\le k\) measure, with each projection yielding a finite signed measure. The number of projected measures, \(2k\), matches the number of component marginal measures \(\mu_{i,j}\), \(i\in[k]\), \(j\in[2]\), appearing in the underlying representation (cf.\ \eqref{eqn:basic-model-measures-d=2}). Thus, rather than recovering a finite collection of scalar parameters, each projection can be viewed as recovering a measure-valued degree of freedom.

        A further connection arises with Radon-type transforms \cite{helgason10}. Classically, Radon transforms associate to a function its integrals over affine subspaces, with corresponding extensions to measures. These transforms are closely related to the projection-based constructions discussed above. In particular, projecting a measure onto one-dimensional subspaces can be viewed as a Radon-type operation, a connection that has been observed in the sliced Wasserstein literature \cite{bonneel15}. Radon transforms play a central role in inverse problems, especially in computed tomography, where one seeks to reconstruct a high-dimensional object from a collection of projections \cite{kak01}. In such settings, it is typically infeasible to collect a complete set of projections, and recovery relies on structural assumptions or regularization. The results here suggest a possible connection to this setting, in that low-rank structure provides a mechanism for recovery from a finite set of projections.

        Silhouette projections of mixture-of-products models are also closely related to blind deconvolution and bilinear inverse problems \cite{campisi07,ahmed14}. For \(c\in\R\), let \(T_c:\R\to\R\) denote the scaling map \(T_c(x)=cx\). If \(a=[a_1,\ldots,a_d]^\top\in\R^d\) and \(\mu\) is of the form \eqref{eqn:basic-model-measures-d=general}, then
        \[
            \pi_a\#\mu
            =
            \sum_{i=1}^k w_i
            T_{a_1}\#\mu_{i,1}\ast \cdots \ast T_{a_d}\#\mu_{i,d}.
        \]
        Recovering the component marginals from projected measures therefore combines mixture identification with a deconvolution problem. Deconvolution is statistically difficult: optimal rates depend strongly on the smoothness of the error distribution. Even in the classical setting where the convolution kernel is known and Gaussian, optimal rates are logarithmic \cite{fan91}. This provides one motivation for sparse projection directions, especially pairwise silhouettes: restricting \(a\) to two nonzero coordinates limits each component to a two-factor convolution, rather than entangling all \(d\) coordinate marginals in every projected measure.

        \section{Algorithm and Applications}\label{sec:application-algorithm}
        The results of Sections~\ref{sec:main-theorems} and \ref{sec:extending-2d} show that, under suitable structural conditions, mixture-of-products models can be identified from a finite collection of one-dimensional linear pushforwards. This suggests a practical estimation strategy: fit a mixture-of-products model by matching its one-dimensional projected distributions to the empirical distributions obtained by projecting the observed data along the same directions.

        This section describes a concrete implementation of this approach, termed \emph{silhouette mixture estimation} (SME). To construct a candidate model, the data are partitioned into blocks intended to correspond to the individual mixture components \(\prod_{j=1}^d \mu_{i,j}\) in \eqref{eqn:basic-model-measures-d=general}. Within each block, an empirical marginal is formed for each coordinate, and the product of these marginals yields a naive Bayes-type empirical estimate of the corresponding component. These block-level product measures are weighted by their relative block sizes and summed to form the final estimator. 

        \subsection{Form of the Estimator} \label{ssec:form-of-the-estimator}
        The construction begins with the empirical measure associated with the observed data. Given samples \(X_1,\ldots,X_n\simiid\mu\), the empirical measure is
        \[
            \hat{\mu}_n
            \triangleq
            \frac{1}{n}\sum_{i=1}^n \delta_{X_i}.
        \]
        To motivate the form of SME, first suppose that \(\mu\) admits a naive Bayes factorization
        \begin{equation}\label{eqn:naive-bayes-example}
            \mu=\prod_{j=1}^d\mu_j,
        \end{equation}
        where each \(\mu_j\) is a probability measure on \(\R\). In this case, a natural factorizable analogue of \(\hat{\mu}_n\) is obtained by taking the product of its one-dimensional empirical marginals.
        Writing \(X_i=[X_{i,1},\ldots,X_{i,d}]^\top\), the resulting factorizable empirical measure is
        \begin{equation}\label{eqn:naive-bayes-empirical}
            \prod_{j=1}^d
            \left(
                \frac{1}{n}\sum_{i=1}^n \delta_{X_{i,j}}
            \right)
            =
            \frac{1}{n^d}
            \sum_{I \in [n]^d}
            \delta_{\left(X_{I_1,1},\ldots,X_{I_d,d}\right)}.
        \end{equation}
        The expansion on the right shows that this measure assigns equal weight to every coordinate-wise recombination of the observed samples.

        For a nonempty multiset \(\{x_1,\ldots,x_m\}\) in \(\R^d\), where
        \(
        x_i = [x_{i,1},\ldots,x_{i,d}]^\top,
        \)
        define the naive Bayes map $\nb$ by
        \[
            \nb(\{x_1,\ldots,x_m\})
            \triangleq
            \prod_{j=1}^d
            \left(
                \frac{1}{m}\sum_{i=1}^m \delta_{x_{i,j}}
            \right).
        \]
        Given a partition \(\sA\) of \([n]\) into \(k\) blocks, SME constructs a factorizable empirical measure from each block and mixes these measures according to their relative block sizes:
        \begin{equation}\label{eqn:basic-nb-estimator}
            \hat{\mu}_{\mathrm{SME}}
            \triangleq
            \sum_{A \in \sA} \frac{|A|}{n}\, \nb\!\left(\{X_i : i \in A\}\right).
        \end{equation}
        Given the training data, the estimator is therefore completely specified by \(\sA\). The subsequent subsections describe how this partition is selected.

        An important property of \(\hat{\mu}_{\mathrm{SME}}\) is that, by construction, it matches every coordinate marginal of the empirical distribution exactly, regardless of the partition. Writing \(\nu^{(j)}\) for the \(j\)-th coordinate marginal of a measure \(\nu\), one has, for each \(j\in[d]\),        
        \begin{align}
            \hat{\mu}_{\mathrm{SME}}^{(j)}
            &=\sum_{A \in \sA} \frac{|A|}{n}\, \nb\!\left(\{X_i : i \in A\}\right)^{(j)} \\
            &= \sum_{A \in \sA} \frac{|A|}{n} \left( \prod_{r=1}^d \left( \frac{1}{|A|} \sum_{\ell \in A} \delta_{X_{\ell,r}} \right) \right)^{(j)} \\
            &= \sum_{A \in \sA} \frac{|A|}{n} \left( \frac{1}{|A|} \sum_{\ell \in A} \delta_{X_{\ell,j}} \right) \\
            &= \frac{1}{n}\sum_{i=1}^n \delta_{X_{i,j}}
            = \hat{\mu}_n^{(j)}.
        \end{align}
        Thus, the coordinate marginals of the estimator are independent of the partition and coincide exactly with those of the empirical measure.

        \subsection{Objective Function}
        Let \(\nu\) denote a candidate estimator of the form \eqref{eqn:basic-nb-estimator}. Estimation proceeds by matching one-dimensional projections of pairwise marginals of \(\nu\) to the corresponding empirical projections derived from \(\hat{\mu}_n\). Motivated by the identifiability results of Section~\ref{sec:extending-2d}, the fitting objective aggregates projection discrepancies across selected pairs of coordinates. Let
        \(
        G = ([d],E)
        \)
        be a graph specifying which coordinate pairs are used for fitting. Write each edge \(e\in E\) as \(\{p,q\}\), with \(p<q\). Extending the marginal notation above, let \(\nu^{(p,q)}\) denote the two-dimensional marginal of \(\nu\) on coordinates \(p\) and \(q\). That is, if
        \(
        Z=[Z_1,\ldots,Z_d]^\top\sim\nu,
        \)
        then
        \(
        [Z_p,Z_q]^\top\sim\nu^{(p,q)}.
        \)

        As in Section~\ref{sec:main-theorems}, for $v \in \R^2$ define $\pi_v(x)=\langle v,x\rangle$. Let \(\dist\) denote a metric on probability measures on \(\R\). The Wasserstein metric is a natural choice because it captures the underlying geometry of the distributions and is efficiently computable in one dimension, although the framework is not restricted to this choice. In all experiments, \(\dist\) is the \(1\)-Wasserstein metric. For each edge \(\{p,q\}\in E\), let \(S_{(p,q)}\subset\R^2\) denote the finite set of projection directions used for that coordinate pair. The directions in \eqref{eqn:silhoulette-vector1}--\eqref{eqn:silhoulette-vector2} provide a theoretically justified choice. For $\{p,q\}\in E$ and $v\in S_{(p,q)}$, consider the discrepancy
        \begin{equation}\label{eqn:basic-distance}
            \dist\!\left(
                \pi_v \# \nu^{(p,q)},
                \pi_v \# \hat{\mu}_n^{(p,q)}
            \right).
        \end{equation}
        Writing \(v=[v_1,v_2]^\top\), the empirical term is computed from the projected samples
        \[
            v_1X_{i,p}+v_2X_{i,q},
            \qquad i\in[n],
        \]
        thereby reducing the comparison to a one-dimensional problem.

        SME minimizes an aggregate of terms of the form \eqref{eqn:basic-distance} over the selected graph edges and projection directions. In the experiments of Section~\ref{sec:experiments}, using substantially fewer than the theoretically suggested \(2k\) projection directions for each coordinate pair worked well in practice. The experiments on benchmark datasets also used non-dense but well-connected graphs. The terms were aggregated using their mean, although alternative aggregation schemes are possible.
        \subsection{Computation of Projected Marginals}
        This subsection describes how to compute the pushforward appearing in \eqref{eqn:basic-distance} for a candidate estimator of the form \eqref{eqn:basic-nb-estimator}. Fix a partition \(\sA\) of the data indices and a coordinate pair \((p,q)\). Consider
        \begin{equation}
            \pi_{[v_1,v_2]}\#
            \left(
                \sum_{A \in \sA}
                \frac{|A|}{n}
                \nb\!\left(\{X_i : i \in A\}\right)
            \right)^{(p,q)}.
        \end{equation}
        This is the pushforward obtained by first taking the \((p,q)\)-marginal of the SME estimator and then projecting linearly along \([v_1,v_2]^\top\). Pushforwards are linear operators on the space of finite signed measures, and taking the \((p,q)\)-marginal simply removes the remaining coordinates. Hence
        \begin{align}
            \pi_{[v_1,v_2]}\#
            \left(
                \sum_{A \in \sA}
                \frac{|A|}{n}
                \nb\!\left(\{X_i : i \in A\}\right)
            \right)^{(p,q)} 
    &=
    \pi_{[v_1,v_2]}\#
    \left(
        \sum_{A \in \sA}
        \frac{|A|}{n}
        \nb\!\left(\{[X_{i,p},X_{i,q}]^\top : i \in A\}\right)
    \right) \\
    &=
    \sum_{A \in \sA}
    \frac{|A|}{n}
    \pi_{[v_1,v_2]}\#
    \nb\!\left(\{[X_{i,p},X_{i,q}]^\top : i \in A\}\right).
        \end{align}
        Thus, the projected SME marginal can be computed blockwise. That is, for each \(A\in\sA\), compute
        \[
            \pi_{[v_1,v_2]}\#
            \nb\!\left(\{[X_{i,p},X_{i,q}]^\top : i \in A\}\right),
        \]
        and then mix these measures using the weights \(|A|/n\).

        Fix a block \(A\in\sA\). Then
        \begin{align}
            &\pi_{[v_1,v_2]}\#
            \nb\!\left(\{[X_{i,p},X_{i,q}]^\top:i\in A\}\right) \\
            &=
            \pi_{[v_1,v_2]}\#
            \left(
                \left(\frac{1}{|A|}\sum_{i\in A}\delta_{X_{i,p}}\right)
                \times
                \left(\frac{1}{|A|}\sum_{j\in A}\delta_{X_{j,q}}\right)
            \right) \\
            &=
            \left(\frac{1}{|A|}\sum_{i\in A}\delta_{v_1X_{i,p}}\right)
            *
            \left(\frac{1}{|A|}\sum_{j\in A}\delta_{v_2X_{j,q}}\right) \\
            &=
            \frac{1}{|A|^2}
            \sum_{i\in A}\sum_{j\in A}
            \delta_{v_1X_{i,p}+v_2X_{j,q}},
        \end{align}
        where \( * \) denotes convolution of measures.
        Substituting this expression back into the full mixture and summing over the blocks yields
        \begin{equation}\label{eqn:full-nb-mix-projcetion}
            \pi_{[v_1,v_2]}\#
            \left(
                \sum_{A \in \sA}
                \frac{|A|}{n}
                \nb\!\left(\{X_i : i \in A\}\right)
            \right)^{(p,q)}
            =
            \sum_{A \in \sA}
            \frac{1}{|A|\,n}
            \sum_{i \in A}
            \sum_{j \in A}
            \delta_{v_1 X_{i,p} + v_2 X_{j,q}}.
        \end{equation}
        For comparison, the projected empirical marginal of $\hat{\mu}_n$ is
        \begin{equation} \label{eqn:empirical-projection}
            \pi_{[v_1,v_2]}\#
            \hat{\mu}_n^{(p,q)}
            =
            \frac{1}{n}
            \sum_{i=1}^n
            \delta_{v_1 X_{i,p} + v_2 X_{i,q}}.
        \end{equation}
        For each \(\{p,q\}\in E\) and \(v=[v_1,v_2]^\top\in S_{(p,q)}\), the measures in \eqref{eqn:full-nb-mix-projcetion} and \eqref{eqn:empirical-projection} are compared using \(\dist\). These discrepancies are then aggregated to form the SME fitting objective.
        \subsection{Optimization}
        With the objective defined, estimation reduces to selecting a partition \(\sA\) that minimizes the aggregate discrepancy over graph edges and projection directions. When partitioning the data into \(k\) blocks, it is convenient to represent a partition by an assignment vector \(z\in[k]^n\), where \(z_i=r\) indicates that sample \(X_i\) is assigned to block \(r\). The induced partition is specified by
        \[
            A_r \triangleq \{i\in[n]:z_i=r\},
            \qquad r\in[k].
        \]
        An empty assignment class has zero mixture weight and can simply be omitted from \(\sA\). The optimization problem is therefore to minimize \(L(z)\) over assignment vectors \(z\in[k]^n\), where \(L(z)\) is obtained by substituting the corresponding SME estimator into \eqref{eqn:basic-distance} and averaging over the selected graph edges and projection directions. The unconstrained assignment space has cardinality \(k^n\), rendering exhaustive optimization infeasible except in trivial regimes. Nevertheless, the objective possesses structure that permits efficient evaluation and practical optimization strategies.

        First, the loss decomposes across graph edges and projection directions. All terms corresponding to distinct graph edges or projection directions may therefore be evaluated independently and aggregated afterward, making the computation highly parallelizable.

        For a fixed edge \(\{p,q\}\) and projection vector \(v\), evaluating the corresponding term in \(L(z)\) requires constructing the discrete measure \eqref{eqn:full-nb-mix-projcetion} and comparing it to the empirical projection \eqref{eqn:empirical-projection}. Each block \(A_r\) contributes a weighted collection of point masses of the form
        \[
            \frac{1}{|A_r|n}
            \sum_{i\in A_r}\sum_{j\in A_r}
            \delta_{v_1X_{i,p}+v_2X_{j,q}}.
        \]
        Because these weights vary across blocks, it is natural to represent each projected block as a list of support points and associated weights.

        When \(\dist\) is the Wasserstein metric, sorting the support points substantially reduces evaluation time. The empirical projection \eqref{eqn:empirical-projection} can be computed once and cached for every combination of graph edge and projection direction. Block-level contributions can also be cached: when a local update changes a single assignment \(z_i\), only the two affected blocks and their corresponding contributions to \eqref{eqn:full-nb-mix-projcetion} must be recomputed.

        The convolution within each block incurs \(O(|A_r|^2)\) cost in a naive implementation. This cost can be reduced by discretizing the values \(X_{i,j}\) during the optimization phase, for example by rounding to a fixed grid, thereby replacing the \(|A_r|\) sample values with a smaller set of weighted representatives. The projected support can be reduced further by choosing projection coefficients that encourage different coordinate combinations to yield the same projected value. For example, if the coordinate values are rounded to \(c\mathbb Z\), with \(c>0\), and the projection families in \eqref{eqn:silhoulette-vector1}--\eqref{eqn:silhoulette-vector2} use \(a=\tfrac12\) and \(b_1=b_2=1\), all scaled coordinate values lie on \((c/2^k)\mathbb Z\). Coincident projected point masses can be combined into a single point mass at their shared location by summing their weights. This aggregation can substantially reduce the effective support size of each component, particularly when the coordinate values contain many repetitions or are discretized onto a coarse grid. Once the partition has been determined, the original full-precision values may be reinstated for downstream use. In the experiments of Section~\ref{sec:experiments}, this approximation yielded substantial computational speedups without materially degrading performance.

        Additional computational savings may be obtained by using fewer projection directions and sparser graphs, thereby reducing the number of Wasserstein distances that must be evaluated. In the experiments of Section~\ref{sec:experiments}, substantially fewer than the theoretically suggested \(2k\) directions per coordinate pair and relatively sparse but well-connected graphs both worked well in practice.

        \subsection{Utilizing the SME Estimator}\label{sec:using-estimator}
        After fitting, SME produces a probability measure of the form
        \begin{equation}\label{eqn:sme-full-form}
            \hat{\mu}_{\mathrm{SME}}
            =
            \sum_{A \in \sA}
            \frac{|A|}{n}
            \prod_{j=1}^d
            \left(
                \frac{1}{|A|}
                \sum_{\ell \in A}
                \delta_{X_{\ell,j}}
            \right).
        \end{equation}
        When expanded, this representation contains
        \(
        \sum_{A\in\sA}|A|^d
        \)
        point-mass terms and may therefore have combinatorially large support. Furthermore, unless all blocks contain equal numbers of samples, the terms contributed by different blocks have different weights. Consequently, \(\hat{\mu}_{\mathrm{SME}}\) cannot simply be treated as a plug-in replacement for an ordinary empirical measure, and downstream procedures must account for its mixture-of-products structure. This subsection describes several practical ways to work with estimators of the form \eqref{eqn:sme-full-form}.

        \paragraph{Sampling}
        Despite the combinatorial expansion suggested by \eqref{eqn:sme-full-form}, sampling from \(\hat{\mu}_{\mathrm{SME}}\) is efficient due to its mixture-of-products structure. To generate a sample, first select a block \(A\in\sA\) with probability \(|A|/n\). Then, independently for each coordinate \(j\in[d]\), draw an index \(\ell_j\) uniformly from \(A\) and set
        \(
        Z_j=X_{\ell_j,j}.
        \)
        The resulting vector \(Z=[Z_1,\ldots,Z_d]^\top\) is distributed according to \(\hat{\mu}_{\mathrm{SME}}\). Sampling requires only a single draw from the mixture weights together with \(d\) independent draws from the selected block and does not require enumerating the \(|A|^d\) point-mass terms contributed by that block.

        \paragraph{Density Estimation}
        The representation \eqref{eqn:sme-full-form} naturally gives rise to a density estimator. For each block \(A\in\sA\) and coordinate \(j\in[d]\), use the one-dimensional sample
        \(
        (X_{\ell,j})_{\ell\in A}
        \)
        to construct a density estimate \(\hat{p}_{A,j}\) using any suitable one-dimensional method, such as a fitted Gaussian density, a kernel density estimator, or even a one-dimensional normalizing flow. Combining these estimates yields
        \begin{equation}\label{eqn:lr-kde}
            \hat{p}(x_1,\ldots,x_d)
            =
            \sum_{A\in\sA}
            \frac{|A|}{n}
            \prod_{j=1}^d
            \hat{p}_{A,j}(x_j).
        \end{equation}
        The estimator \eqref{eqn:lr-kde} has the mixture-of-products form \eqref{eqn:basic-model-measures-d=general}, with each block \(A\) corresponding to a component, \(|A|/n\) serving as its mixture weight, and \(\hat{p}_{A,j}\) estimating the density of the corresponding component marginal probability measure \(\mu_{i,j}\). The one-dimensional estimators can be fitted independently across blocks and coordinates, allowing straightforward parallelization.

        A further useful property of mixture-of-products densities is that they permit efficient conditional density evaluation and sampling.
        Consider a density
        \begin{equation}\label{eqn:mixture-of-products-density}
            p(x)=\sum_{i=1}^k w_i\prod_{j=1}^d p_{i,j}(x_j),
        \end{equation}
        where \(w_i>0\), \(\sum_{i=1}^k w_i=1\), and each \(p_{i,j}\) is a probability density function. Let \(X\sim p\). For \(S\subset[d]\), let \(X_S\) denote the tuple of coordinates indexed by \(S\), listed in increasing index order, and write \(S^c=[d]\setminus S\). For conditioning values \(y_S\) with positive marginal density,
        \begin{align*}
            p(x_{S^c}\mid X_S=y_S)
            &\propto
            \sum_{i=1}^k
            w_i
            \left(
                \prod_{j\in S}p_{i,j}(y_j)
            \right)
            \left(
                \prod_{j\in S^c}p_{i,j}(x_j)
            \right) \\
            &=
            \sum_{i=1}^k
            \tilde w_i(y_S)
            \prod_{j\in S^c}p_{i,j}(x_j),
        \end{align*}
        where \(\tilde w_i(y_S)\triangleq w_i\prod_{j\in S}p_{i,j}(y_j)\) is the unnormalized conditional weight of component \(i\).
        Therefore, the conditional density retains the mixture-of-products form
        \[
            p(x_{S^c}\mid X_S=y_S)
            =
            \sum_{i=1}^k
            \frac{\tilde w_i(y_S)}
            {\sum_{\ell=1}^k\tilde w_\ell(y_S)}
            \prod_{j\in S^c}p_{i,j}(x_j).
        \]
        Thus, conditioning changes only the mixture weights, leaving each component's marginals on the remaining coordinates unchanged. Notably, a single fitted density of the form \eqref{eqn:mixture-of-products-density} can be used for density evaluation and sampling, both unconditionally and conditionally, without fitting separate models or resorting to approximate inference. Conditioning on rectangular events is similarly tractable, since the required component probabilities can be computed from coordinate-wise CDFs.

        \paragraph{Distance-Based Methods}
        Many procedures in statistics, such as \(k\)-nearest-neighbor density estimation and anomaly detection, are based on order statistics of distances \cite{biau15}. For a target distribution \(\mu\) and a test point \(x\), such methods implicitly estimate the smallest radius \(r\) such that
        \begin{equation} \label{eqn:knn}
            \mathbb{P}_{X\sim\mu}\!\left(\|X-x\|\le r\right)
            \approx
            \frac{k}{n},
        \end{equation}
        that is, the ball centered at \(x\) contains approximately a fraction \(k/n\) of the total mass.

        In standard \(k\)-nearest-neighbor methods, such radii are estimated directly from the empirical measure \(\hat{\mu}_n\) through distances to the observed samples. For \(\hat{\mu}_{\mathrm{SME}}\), directly computing such radii, or even evaluating
        \(
        \mathbb{P}_{Z\sim\hat{\mu}_{\mathrm{SME}}}
        \!\left(\|Z-x\|\le r\right)
        \),
        may be costly because the estimator can have combinatorially many support points. However, under the \(\ell_\infty\) norm, the product structure yields a useful simplification. The event \(\|Z-x\|_\infty\le r\) is equivalent to the coordinate-wise constraints \(|Z_j-x_j|\le r\) for all \(j\in[d]\). Writing \(\ind(\cdot)\) for the indicator function, \eqref{eqn:sme-full-form} implies
        \begin{equation}
            \mathbb{P}_{Z\sim\hat{\mu}_{\mathrm{SME}}}
            \!\left(\|Z-x\|_\infty\le r\right)
            =
            \sum_{A\in\sA}
            \frac{|A|}{n}
            \prod_{j=1}^d
            \left(
                \frac{1}{|A|}
                \sum_{\ell\in A}
                \ind\left(|X_{\ell,j}-x_j|\le r\right)
            \right).
        \end{equation}
        Thus, evaluating \(\ell_\infty\)-ball probabilities reduces to computing, for each block, the one-dimensional empirical interval probabilities, multiplying them across coordinates, and taking the convex combination of the resulting products with weights \(|A|/n\). This avoids explicitly enumerating the \(|A|^d\) coordinate recombinations within each block. The radius \(r\) in \eqref{eqn:knn} can then be found efficiently by bisection.

        For other norms, such as \(\ell_2\), the ball event does not decouple across coordinates, but the corresponding probabilities can instead be approximated by sampling \(Z\sim\hat{\mu}_{\mathrm{SME}}\) and forming the empirical CDF of \(\|Z-x\|\). The Dvoretzky--Kiefer--Wolfowitz inequality provides uniform control of this approximation as a function of \(r\).

        \section{Experiments} \label{sec:experiments}

        This section presents evaluations of SME on synthetic and real-world datasets.  The synthetic experiment tests whether it recovers the ground-truth rank-one components of an identifiable mixture-of-products model, while the real-world experiments assess its usefulness for downstream density estimation. The SME fitting objective uses the \(1\)-Wasserstein distance in all experiments.

        \subsection{Mixture Recovery on a Synthetic Dataset}
        To assess whether SME can recover an identifiable mixture-of-products structure from projected marginals, consider the following synthetic example. For an interval \(I\subset\R\), let \(\unif(I)\) denote the uniform probability measure on \(I\), and set
        \begin{align*}
            \nu_1 &= \unif\!\left(\left[-\tfrac{5}{6},\tfrac{5}{6}\right]\right),\\
            \nu_2 &= \tfrac{1}{2}\,\unif\!\left(\left[-\tfrac{3}{2},-\tfrac{1}{2}\right]\right) + \tfrac{1}{2}\,\unif\!\left(\left[\tfrac{1}{2},\tfrac{3}{2}\right]\right),\\
            \nu_3 &= \tfrac{1}{2}\,\unif\!\left(\left[-\tfrac{13}{6},-\tfrac{7}{6}\right]\right) + \tfrac{1}{2}\,\unif\!\left(\left[\tfrac{7}{6},\tfrac{13}{6}\right]\right).
        \end{align*}
        Define the three mixture components by
        \[
            \mu_1=\nu_1\times\nu_1,\qquad
            \mu_2=\nu_2\times\nu_2,\qquad
            \mu_3=\nu_3\times\nu_3,
        \]
        and let
        \begin{equation}
            \label{eqn:synth-exp-mixture}
            \mu=\frac{1}{3}\sum_{i=1}^3\mu_i.
        \end{equation}
        The mixture components $\mu_1,\mu_2,\mu_3$ have overlapping supports, but the component marginals \(\nu_1,\nu_2,\nu_3\) satisfy the joint irreducibility assumption, so by Proposition~\ref{prop:ji-ident}, the representation of $\mu$ in \eqref{eqn:synth-exp-mixture} is identifiable. Each component is symmetric about both coordinate axes. The ``Ground truth'' panel of Figure~\ref{fig:synthetic} shows 900 samples, with 300 drawn from each component. Point color and marker indicate the true component membership. The components exhibit nontrivial overlap while retaining regions of support unique to each component.

        The remaining panels in Figure~\ref{fig:synthetic} show partitions obtained by initializing component labels at random and running the SME optimization until a local minimum was reached. The parameters \(a=-2\) and \(b_1=b_2=1\) were used in \eqref{eqn:silhoulette-vector1} and \eqref{eqn:silhoulette-vector2}, with \(k=1,2,3\) in the projection formulas for the two-, four-, and six-projection cases, respectively. The objective was the mean Wasserstein distance over the projections. The negative sign gives projection directions with both positive and negative slopes. In these synthetic experiments, projection directions were not normalized to unit length, so the larger vector norms associated with higher powers of \(a\) effectively gave those directions greater weight in the objective. This was intended to place greater emphasis on newly introduced higher-order projections as additional projection directions were added. The one-projection case used only the projection onto \([1,1]^\top\). As noted in Section~\ref{ssec:form-of-the-estimator}, the coordinate-axis marginals are matched exactly by construction, so these experiments implicitly incorporate the axis projections in addition to the explicitly listed projection directions. For visualization purposes, component labels are permuted to align with the ground-truth labeling.

        \begin{figure}[hbtp]
            \centering
            \includegraphics[width=0.32\linewidth]{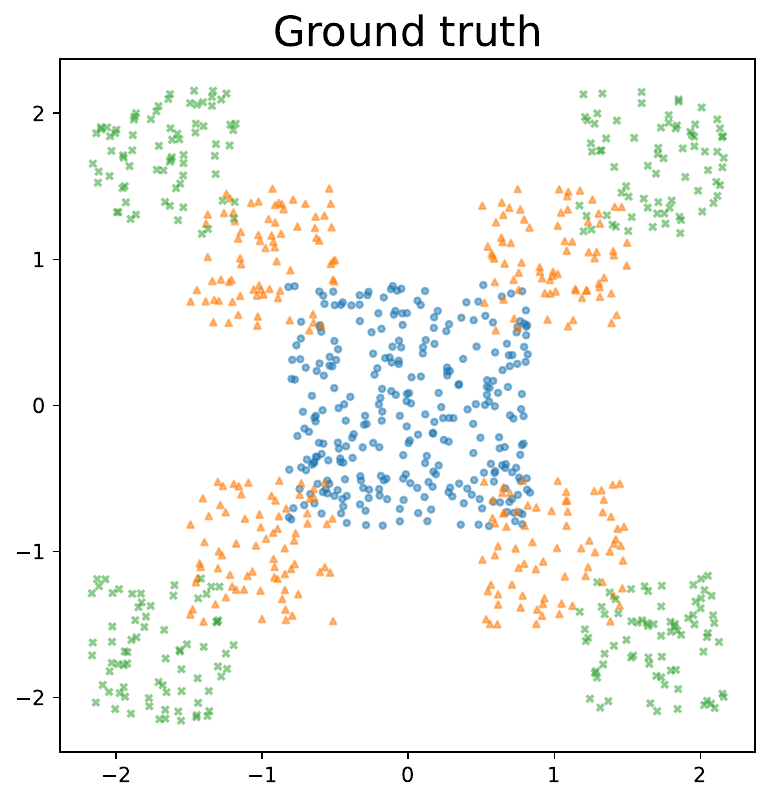}
            \includegraphics[width=0.32\linewidth]{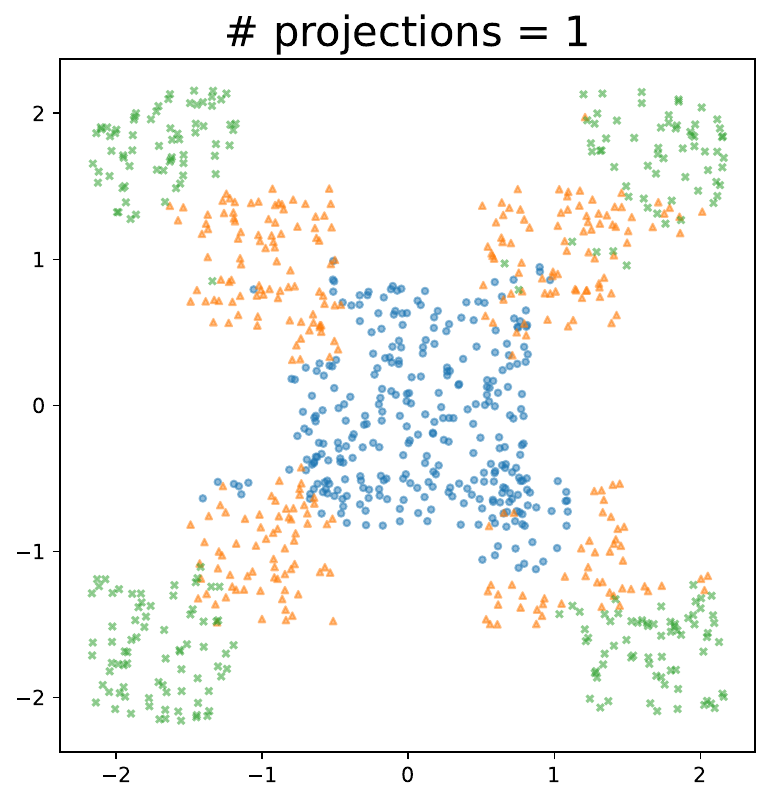}
            \includegraphics[width=0.32\linewidth]{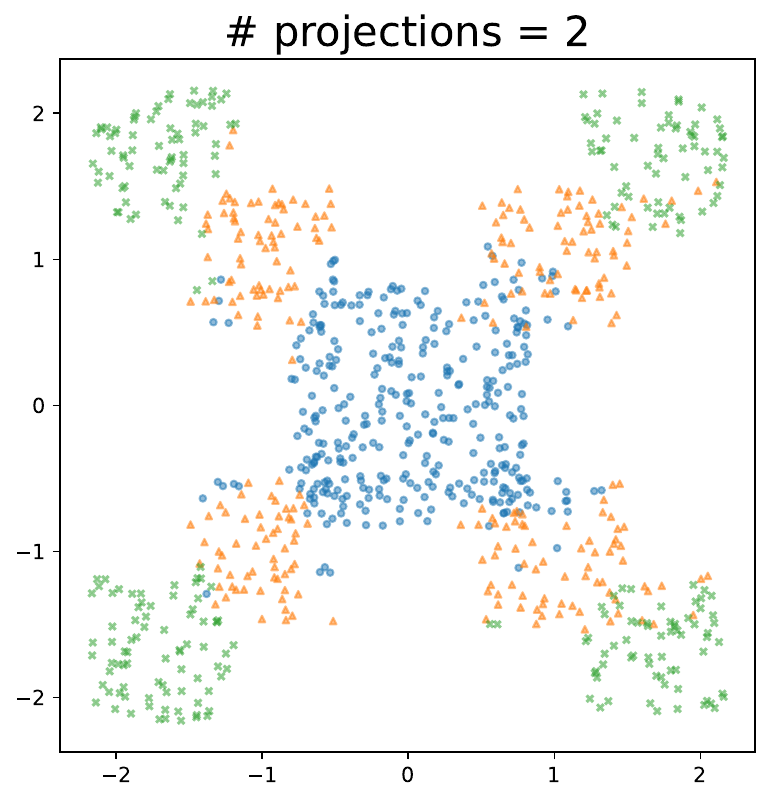}
            \vspace{0.5em}
            \includegraphics[width=0.32\linewidth]{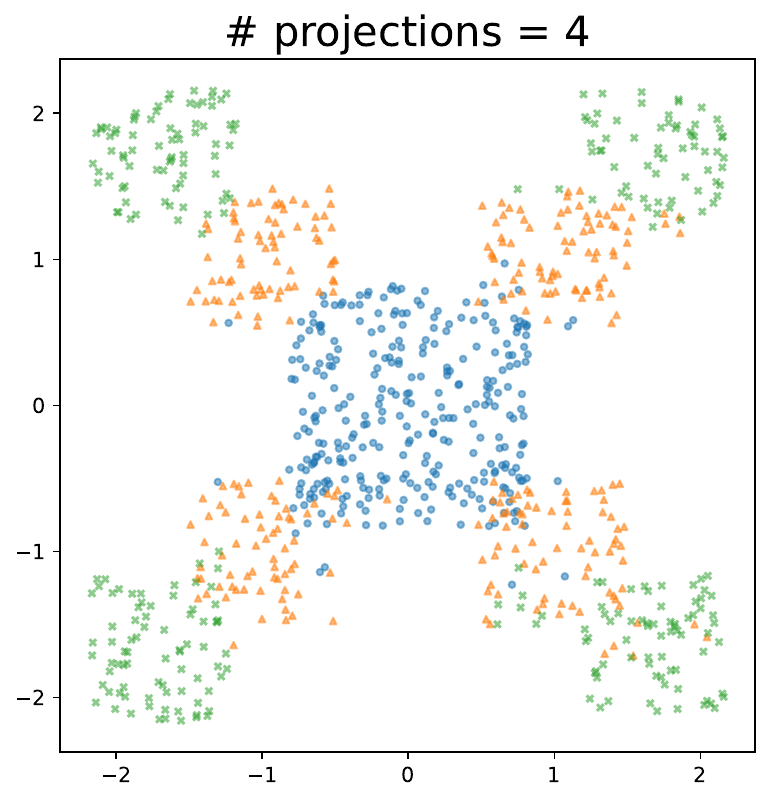}
            \includegraphics[width=0.32\linewidth]{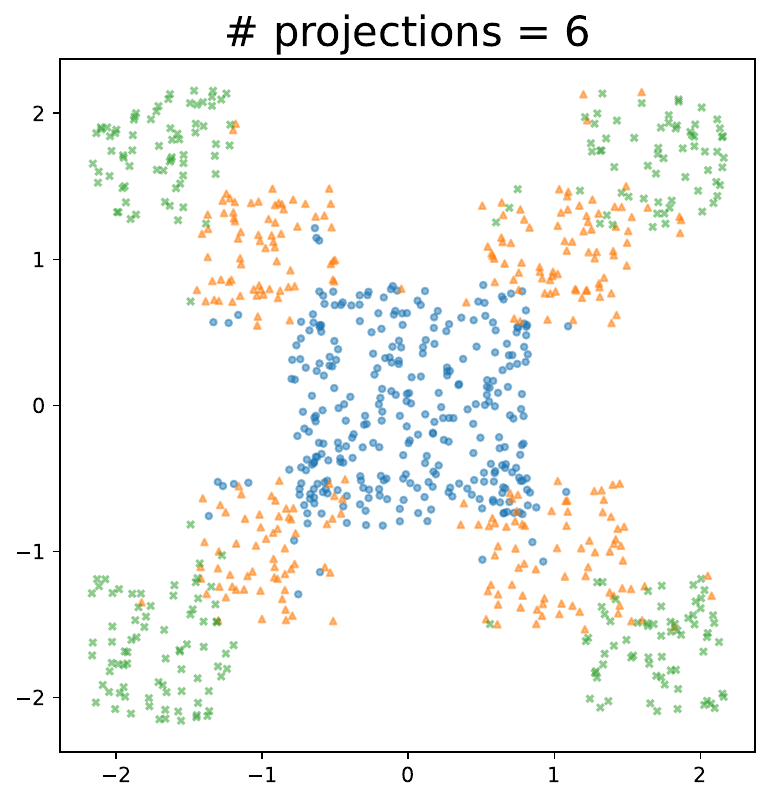}
            \caption{Ground-truth component assignments and partitions recovered by SME using \(1\), \(2\), \(4\), and \(6\) non-axis projections. Colors and markers indicate component labels. Every fit also matches the empirical coordinate marginals exactly by construction.}
            \label{fig:synthetic}
        \end{figure}

        As seen in Figure~\ref{fig:synthetic}, a single projection recovers the main shapes and locations of the underlying components, although some misassignments remain, particularly between the central component (blue circles) and the adjacent component (orange triangles). The recovered central and adjacent clusters extend into one another's regions, most visibly near the upper-left and lower-right corners of the central square. This produces directional leakage rather than errors confined to the local overlap region. Notably, the leakage occurs approximately orthogonally to the sole projection direction \([1,1]^\top\).

        Adding a second non-axis projection substantially reduces this directional leakage and improves agreement with the ground truth. Beyond two projections, the changes are more modest and do not show a uniform improvement: the four- and six-projection fits still retain some errors near overlap regions and a small number of isolated misassignments. This raises the question of whether the theoretical count of \(2k=6\) projections is necessary for this particular model or primarily reflects a worst-case requirement.

        \subsection{Density Estimation on Real-World Datasets}\label{sec:exp-real-world}

        This subsection evaluates SME and several competing methods on real-world density estimation tasks. For each split, performance is measured by the mean log-likelihood on the test set (equivalently, negative test cross-entropy). Experiments were conducted on the HEPMASS \cite{hepmass} and kin8nm \cite{kin} datasets. For HEPMASS, the preprocessed dataset from \cite{papamakarios17}, which excludes dimensions containing discrete components and standardizes each coordinate to have mean zero and unit variance, was used. For kin8nm, each coordinate was likewise standardized using the mean and variance of the full dataset before splitting. Each dataset was split into training, validation, and test sets, and all experiments were repeated across 20 random splits. Dataset dimensions and training set sizes are included in Table~\ref{tab:results} for reference. The validation split is used somewhat differently across methods depending on the corresponding training procedure, as described below in Section~\ref{sec:methods}. Additional experimental details are provided in Appendix~\ref{appx:experimental}.

        \subsubsection{Density Estimation Methods}\label{sec:methods}
        The experiments compare the SME-based density estimator described in Section~\ref{sec:using-estimator} with several standard density estimation baselines.

        \paragraph{Silhouette Mixture Estimation (SME)}
        SME used the nonparametric density estimation procedure described in Section~\ref{sec:using-estimator}. For each dataset, a random \(4\)-regular graph was constructed on the \(d\) coordinates, giving \(2d\) edges. For every edge, the same four unit-norm projection directions were used. These were obtained from \eqref{eqn:silhoulette-vector1} and \eqref{eqn:silhoulette-vector2} with \(a=-2\) and \(b_1=b_2=1\), followed by normalization. Four projections are sufficient according to the theory when \(k=2\), although the experiments also considered larger numbers of mixture components.

        The number of mixture components was selected from \(k\in\{2,3,4\}\) by maximizing validation likelihood. This range was chosen for computational convenience but did not appear to severely constrain model selection: the values selected by validation averaged \(3.1\) for HEPMASS and \(2.2\) for kin8nm. To initialize the component assignments, the training samples were assigned independently and uniformly at random to the \(k\) components. The algorithm then processed the samples repeatedly in a fixed order. For each sample, the objective was evaluated under every possible component assignment. After all assignments had been evaluated, the sample was assigned to the component yielding the lowest objective value. The procedure terminated when a complete pass through the training set produced no improvement, that is, when it reached a local minimum with respect to single-sample reassignments.

        After the partition was obtained, a one-dimensional Gaussian kernel density estimator was fitted to each component-coordinate pair, and the resulting estimators were combined using the procedure described in Section~\ref{sec:using-estimator}. For each component and coordinate, the bandwidth was selected by \(15\)-fold cross-validation over a logarithmic grid of \(80\) values ranging from \(10^{-3}\) to \(1\). The separate validation split was not used for bandwidth selection.
        \paragraph{Kernel Density Estimator (KDE)}
        A standard Gaussian KDE was used, with the bandwidth selected using the validation dataset over the same logarithmic grid used for SME.
        \paragraph{Naive Bayes KDE (NB-KDE)} This method fits a one-dimensional Gaussian kernel density estimator to each coordinate and uses the product of these estimators as the final density estimate. It is exactly equivalent to the SME density estimation procedure above with all samples assigned to a single component.
        \paragraph{Gaussian Mixture Model (GMM)}
        A Gaussian mixture model was fitted using the EM implementation in scikit-learn \cite{sklearn}, with default settings except for the number of components. The candidate models ranged from a single multivariate Gaussian density to a \(32\)-component Gaussian mixture, with the number of components selected using the validation log-likelihood.
        \paragraph{Axis-Aligned Gaussian Mixture Model (GMM-AA)}
        The axis-aligned GMM used exactly the same fitting and model-selection procedure as the GMM above, except that each component covariance matrix was constrained to be diagonal.
        \paragraph{Masked Autoregressive Flow (MAF) \cite{papamakarios17}}
        Masked autoregressive flows were implemented using the Zuko Python library \cite{zuko}. Architecture and training hyperparameters were chosen through preliminary experiments, and validation likelihood was used for early stopping. Complete implementation details are provided in Appendix~\ref{appx:experimental}.

        \subsubsection{Choice of Baselines}
        The competing methods were chosen to reflect different structural assumptions relative to SME. KDE and NB-KDE represent two extremes. If \(k=n\) and each sample is assigned to its own component, then the SME estimator reduces to the empirical distribution and the silhouette objective is zero. Smoothing this empirical distribution with Gaussian kernels yields KDE. At the opposite extreme, when \(k=1\), the model reduces to a single-component factorized estimator, corresponding to NB-KDE. GMM-AA provides a parametric analogue with axis-aligned covariance matrices, preserving the assumption that coordinates are independent within each mixture component. Together, these three methods compare SME with its two extreme cases and with a parametric model having the same within-component factorization. The full-covariance GMM, which includes a single Gaussian model as a special case, serves as a representative parametric clustering and density estimation baseline. Finally, MAF is included as a representative deep density estimator, although the relatively small training sets considered here are not particularly favorable to such models.

        \subsubsection{Results and Discussion}
        \begin{table}[b]
            \centering
            \begin{tabular}{l c c}
                \hline
                Dataset $(d,n)$ 
                & HEPMASS $(21,250)$ 
                & kin8nm $(9,1000)$ 
                \\
                \hline
                SME (proposed)
                & $\mathbf{-27.36 \pm 0.13}$  {\scriptsize $(3.1 \pm 0.8)$}
                & $\mathbf{-11.495 \pm 0.051}$  {\scriptsize $(2.2 \pm 0.4)$} 
                \\
                KDE    
                & $-30.71 \pm 0.09$ {\scriptsize $(2 \times 10^{-6})$}
                & $-12.386 \pm 0.016$ {\scriptsize $(2 \times 10^{-6})$}
                \\
                NB-KDE 
                & $-29.24 \pm 0.70$ {\scriptsize $(2 \times 10^{-6})$}
                & $-11.563 \pm 0.012$ {\scriptsize $(4 \times 10^{-5})$}
                \\
                GMM
                & $-28.19 \pm 0.25$ {\scriptsize $(2 \times 10^{-6})$}
                & $-12.176 \pm 0.059$ {\scriptsize $(2 \times 10^{-6})$}
                \\
                GMM-AA    
                & $-27.73 \pm 0.12$ {\scriptsize $(2 \times 10^{-6})$}
                &  $-12.293 \pm 0.037$ {\scriptsize $(2 \times 10^{-6})$}
                \\
                MAF    
                & $-28.57 \pm 0.55$ {\scriptsize $(2 \times 10^{-6})$}
                & $-12.216 \pm 0.107$ {\scriptsize $(2 \times 10^{-6})$} 
                \\
                \hline
            \end{tabular}
            \caption{
                Test log-likelihood (mean \(\pm\) standard deviation across \(20\) random splits). For SME, parentheses report the validation-selected number of components \(k\) (mean \(\pm\) standard deviation); for competing methods, parentheses report two-sided Wilcoxon signed-rank \(p\)-values from paired comparisons with SME. Bold indicates the best mean in each column.
            }
            \label{tab:results}
        \end{table}

        Table~\ref{tab:results} summarizes performance across the \(20\) matched random splits and reports \(p\)-values from two-sided Wilcoxon signed-rank tests comparing SME with each competing method. A small \(p\)-value provides evidence of a systematic difference in performance between the methods. For SME, the table also reports the mean and standard deviation across splits of the number of components \(k\) selected using the validation data.

        SME achieves the highest mean test log-likelihood for both datasets in the experimental settings considered here. These experiments focus on a sample-limited, moderate-dimensional regime in which flexible density estimation remains desirable but fully high-dimensional models can be difficult to estimate reliably. This differs substantially from the HEPMASS setting commonly used to evaluate normalizing flows: the evaluation in \cite{papamakarios17} uses \(315{,}123\) training observations, compared with \(250\) here, and demonstrates strong flow performance in that large-sample setting. The MAF results in Table~\ref{tab:results} should therefore be understood as describing a different regime, rather than as a general comparison with flows under conventional large-sample benchmark conditions.

        For HEPMASS, the strongest competing method is GMM-AA, while for kin8nm it is NB-KDE. This is consistent with the structural assumptions of these methods: GMM-AA combines mixture structure with within-component coordinate independence, while NB-KDE uses a single fully factorized nonparametric density. The full-covariance GMM must estimate substantially more parameters to capture dependence within each component, while multivariate KDE is well known to be particularly susceptible to the curse of dimensionality. SME occupies an intermediate position, retaining flexibility through its mixture structure while reducing estimation within each component to one-dimensional problems.

        Runtime details and further diagnostics, including the relationship between the silhouette objective and downstream density estimation performance, are provided in Appendix~\ref{appx:experimental}.

        \section{Sharpness of the Main Theorems}\label{sec:sharpness}
        This section examines the sharpness of the identifiability and homeomorphism results in Section~\ref{sec:main-theorems}. The first two results concern the number and geometry of the projection directions required for identifiability. Theorem~\ref{thm:2k-necessary} shows that the projection count \(2k\) in Theorem~\ref{thm:main} and Corollary~\ref{cor:main-multiview} is optimal, while Proposition~\ref{prop:2k-geometry-necessary} shows that, even with \(2k\) projections, the directions cannot be chosen arbitrarily. The remaining results show that the tail assumption in Theorem~\ref{thm:main} and the boundedness restriction in Theorem~\ref{thm:weakstar-homeo} cannot in general be omitted. Proofs of all results in this section are given in Section~\ref{sec:proofs-sharpness}. 
        \begin{samepage}
            \begin{theorem} \label{thm:2k-necessary}
                Let $k \in \N$ and let $v_1,\ldots,v_{2k-1} \in \R^2$. There exist probability measures \(\mu_{i,j},\mu'_{i,j}\in\sP(\R)\) for \(i\in[k]\) and \(j\in[2]\), each admitting a smooth, compactly supported density, and weights \(w_i,w'_i\geq0\) for \(i\in[k]\), satisfying                \[
                    \sum_{i=1}^k w_i = \sum_{i=1}^k w'_i = 1,
                \]
                such that the probability measures
                \[
                    \mu = \sum_{i=1}^k w_i \mu_{i,1} \times \mu_{i,2}, \qquad \mu' = \sum_{i=1}^k w'_i \mu'_{i,1} \times \mu'_{i,2}
                \]
                satisfy
                \[
                    \pi_{v_\ell}\#\mu = \pi_{v_\ell}\#\mu' \qquad \text{for all } \ell\in [2k-1],
                \]
                while $\mu \neq \mu'$.
            \end{theorem}
        \end{samepage}
        Thus, no collection of \(2k-1\) projection directions can identify all rank-\(\le k\) probability measures. Notably, the lower bound holds even for probability measures whose component marginals have smooth, compactly supported densities, so it is not an artifact of signed measures, irregularity, or heavy tails.

Although the \(2k\) projections match the number of component marginals in the mixture representation, this correspondence does not itself explain the lower bound. A vector-RIP map cannot represent all \(k\)-sparse vectors using fewer than \(k\) scalar measurements. In contrast, an element of \(\sM([0,1])^{2k}\) can be encoded as a single element of \(\sM([0,1])\) by contracting and shifting its entries onto disjoint subintervals through affine pushforwards, then summing them. These operations are linear in the measures, and the disjoint supports make their norms add, giving an isometric encoding. In fact, composing the identifying silhouette of Theorem~\ref{thm:main} with suitable affine rescalings and this concatenation gives an injective linear representation
\[
    \rankmeas_{\le k}([0,1],[0,1])
    \longrightarrow
    \sM([0,1]).
\]
Thus, the lower bound reflects the restriction to projection measurements, rather than a dimensional obstruction arising from the number of component marginals.

        The number of directions is not the only consideration. The following proposition shows that \(2k\) directions can still fail to provide identifiability when arranged in certain symmetric configurations.
        \begin{samepage}\begin{proposition}\label{prop:2k-geometry-necessary}
            Let \(k\in\N\) and let \(a_1,\ldots,a_k>0\) be distinct.
            For \(r\in[k]\), define
            \[
                v_r^+=[a_r,1]^\top,
                \qquad
                v_r^-=[-a_r,1]^\top.
            \]
            Then there exist probability measures
            \(\mu_{i,j},\mu'_{i,j}\in\sP(\R)\) for
            \(i\in[\lfloor k/2\rfloor+1]\) and \(j\in[2]\),
            each admitting a smooth, compactly supported density, and weights
            \(w_i,w_i'\geq0\) for \(i\in[\lfloor k/2\rfloor+1]\), satisfying
            \[
                \sum_{i=1}^{\lfloor k/2\rfloor+1} w_i
                =
                \sum_{i=1}^{\lfloor k/2\rfloor+1} w_i'
                =
                1,
            \]
            such that the probability measures
            \[
                \mu
                =
                \sum_{i=1}^{\lfloor k/2\rfloor+1}
                w_i\,\mu_{i,1}\times\mu_{i,2},
                \qquad
                \mu'
                =
                \sum_{i=1}^{\lfloor k/2\rfloor+1}
                w_i'\,\mu_{i,1}'\times\mu_{i,2}'
            \]
            satisfy
            \[
                \pi_{v_r^+}\#\mu
                =
                \pi_{v_r^+}\#\mu',
                \qquad
                \pi_{v_r^-}\#\mu
                =
                \pi_{v_r^-}\#\mu'
                \qquad
                \text{for all }r\in[k],
            \]
            while \(\mu\neq\mu'\).
        \end{proposition}
    \end{samepage}
    Together with Theorem~\ref{thm:2k-necessary}, this shows that \(2k\) is the optimal projection count and that the geometry of the directions is a substantive part of the construction in Theorem~\ref{thm:main}. In particular, Proposition~\ref{prop:2k-geometry-necessary} identifies reflection symmetry among the projection directions as a potential obstruction to identifiability.

    The preceding results concern the number and geometry of the projection directions. The following lemma shows that the tail assumption in Theorem~\ref{thm:main} cannot be substantially weakened.

    \begin{lemma}\label{lem:proj-null}
        Let \(\ell\in\N\) and let \(v_1,\ldots,v_\ell\in\R^2\). There exist nonzero measures \(\mu,\nu\in\sM(\R)\), each admitting a density belonging to the Schwartz class, such that
        \[
            \pi_{v_i}\#(\mu\times\nu)=0
            \qquad \text{for all } i\in[\ell].
        \]
    \end{lemma}

    Functions in the Schwartz class are smooth and decay, together with all their derivatives, faster than any polynomial \cite[p.~237]{folland99}. Decomposing the factors into positive and negative parts and normalizing preserves this tail decay and yields the following corollary for mixture-of-products models.
    \begin{corollary}\label{cor:proj-null}
        Let $\ell\in\N$ and let $v_1,\ldots,v_\ell\in\R^2$. There exist probability measures $\mu_{i,j},\mu'_{i,j}\in\sP(\R)$ for $i,j\in[2]$, each admitting a density that decays faster than any polynomial, and weights
        $w_1,w_2,w'_1,w'_2>0$ satisfying
        \[
            w_1+w_2=w'_1+w'_2=1,
        \]
        such that the mixture-of-products models
        \[
            \mu=\sum_{i=1}^2 w_i\,\mu_{i,1}\times\mu_{i,2},
            \qquad
            \mu'=\sum_{i=1}^2 w'_i\,\mu'_{i,1}\times\mu'_{i,2}
        \]
        satisfy
        \[
            \pi_{v_i}\#\mu=\pi_{v_i}\#\mu'
            \qquad \text{for all } i\in[\ell],
        \]
        while $\mu\neq\mu'$.
    \end{corollary}
    Thus, for \(k\ge2\), finite-projection identifiability can fail even when the component densities decay faster than any polynomial and have moments of all orders, whereas Theorem~\ref{thm:main} establishes identifiability under the stronger assumption that their Laplace transforms are finite on a neighborhood of the origin.

    The preceding results concern identifiability. The next result shows that the boundedness restriction in Theorem~\ref{thm:weakstar-homeo} cannot be omitted, regardless of the finite collection of projection directions used.

    \begin{theorem}\label{thm:boundedness-necessary}
        For any finite collection \(v_1,\ldots,v_\ell\in\R^2\),
        the map
        \[
            \mu\longmapsto
            \bigl(\pi_{v_1}\#\mu,\ldots,\pi_{v_\ell}\#\mu\bigr)
        \]
        is not a weak\(^*\)--weak\(^*\) homeomorphic embedding of
        \(\rankmeas_{\le1}([0,1],[0,1])\).
    \end{theorem}

    The obstruction persists even under strong regularity conditions. The proof constructs a sequence of rank-one signed measures with \(C^\infty\) densities supported on \([0,1]^2\), whose total variation norms diverge while every prescribed projection converges to zero in total variation. Thus, the failure is not caused by singular or irregular measures: compact support and smoothness do not compensate for the absence of a uniform total variation bound.

    \section{Directions for Future Research}\label{sec:future}
    Theorems~\ref{thm:main} and \ref{thm:weakstar-homeo} establish that compactly supported low-rank probability measures on \(\R^2\) are identified by their silhouettes and that this representation preserves the weak\(^*\) topology. Section~\ref{sec:extending-2d} extends this framework to higher dimensions through collections of pairwise silhouettes, identifying settings in which the underlying measure can be recovered. Sections~\ref{sec:application-algorithm} and \ref{sec:experiments} develop these ideas algorithmically and demonstrate practical performance gains over standard nonparametric baselines. These results suggest several directions for further investigation.

    One broad direction is to better understand the geometry induced by the silhouette representation beyond the injectivity and topological results established here. In finite-dimensional low-rank problems, injectivity is closely related to restricted-isometry-type bounds. This motivates the search for analogous quantitative guarantees for the silhouette operator. For measures, such guarantees are likely to depend substantially on the class of measures, the projection directions, and the metrics used on the domain and range. In particular, it would be useful to identify settings in which stability estimates hold and to determine whether additional projections improve them even after identifiability has been achieved.

    A further direction is to use the mixture-of-products assumption to recast statistical problems in silhouette space, rather than using silhouettes only to fit a density estimator. Under the assumptions of the identifiability results, a prescribed finite collection of projected measures determines the full multivariate law. Thus, for example, a two-sample test could compare just those one-dimensional projections, with equality of the projected laws equivalent to equality of the original distributions at the population level. The broader question is which other statistical tasks can be usefully recast in terms of this finite collection of one-dimensional measures, and what guarantees or efficiencies result.

    \section{Proof of Injectivity of the Silhouette Operator on Low-Rank Measures} \label{sec:proofs-injectivity}
    This section proves the central result of the paper, Theorem~\ref{thm:main}, along with its immediate corollary, Corollary~\ref{cor:main-multiview}. The proof requires a substantial amount of technical setup and preparatory results before the main proof can be presented.
    \subsection{Signed Measure Preliminaries and Rank}\label{sec:signed-measure-prelim}
    For a measurable space $(\Omega,\sF)$ let $\sM(\Omega)$ denote the Banach space of finite signed measures on $\Omega$ equipped with the total variation norm
    \[
        \|\mu\| \triangleq \mu^+(\Omega) + \mu^-(\Omega),
    \]
    where $\mu^+$ and $\mu^-$ are the components of the Jordan decomposition $\mu = \mu^+ - \mu^-$. The subset $\sP(\Omega) \subset \sM(\Omega)$ denotes probability measures.
    Recall that, for a measurable function $f:\Omega \to \R$, the integral of a function with respect to a signed measure is defined by  
    \[
        \int f\,d\mu \;\triangleq\; \int f\,d\mu^+ \;-\; \int f\,d\mu^- ,
    \]
    as in \cite{folland99}.
    Thus the integral on the left-hand side is finite if and only if the two integrals on the right-hand side are finite. Consequently, whenever the integral on the left-hand side is finite, the above decomposition is valid.

    The following proposition shows that rank admits a characterization analogous to the familiar matrix case: a decomposition is minimal precisely when the component measures are linearly independent in each coordinate.

    \begin{proposition}\label{prop:rank-characterization}
        Let $(\Omega_1,\mathcal{F}_1)$ and $(\Omega_2,\mathcal{F}_2)$ be measurable spaces and let $k\in\mathbb{N}$. Suppose
        \[
            \mu = \sum_{i=1}^{k} \mu_{i,1}\times\mu_{i,2}
            \in \mathcal{M}(\Omega_1\times\Omega_2),
        \] where $\mu_{i,j}\in\mathcal{M}(\Omega_j)$ for $i\in[k]$ and $j\in[2]$. Then \(\mu\) has rank \(k\) if and only if \(\mu_{1,1},\ldots,\mu_{k,1}\) are linearly independent and \(\mu_{1,2},\ldots,\mu_{k,2}\) are linearly independent.
    \end{proposition}
    Before proving Proposition~\ref{prop:rank-characterization}, some additional measure-theoretic machinery is required. For a measurable space $(\Omega,\sF)$ denote by $L^\infty(\Omega)$ the Banach space of bounded measurable functions $f:\Omega\to\R$, equipped with the norm $\|f\| = \sup_{x \in \Omega}|f(x)|$. Given measurable spaces $(\Omega_1,\sF_1)$ and $(\Omega_2,\sF_2)$, and a bounded measurable function $f\in L^\infty(\Omega_1)$, define the bounded linear operator
    \[
        T_f:\sM(\Omega_1\times\Omega_2)\to\sM(\Omega_2)
    \]
    by
    \begin{equation}\label{eqn:Tf-def}
        (T_f\mu)(A)
        \;\triangleq\;
        \int_{\Omega_1\times\Omega_2} f(x)\, \ind_A(y)\,d\mu(x,y),
        \qquad A\in\sF_2.
    \end{equation}
    In particular, if $\mu\in\sM(\Omega_1)$ and $\mu'\in\sM(\Omega_2)$, then
    \[
        T_f(\mu\times\mu') = \mu(f)\,\mu'.
    \]

    The following lemma provides a convenient way to apply linear algebraic arguments to finite collections of measures.

    \begin{lemma}\label{lem:measure-bidual}
        Let $(\Omega,\sF)$ be a measurable space and let $\mu_1,\ldots,\mu_k\in\sM(\Omega)$ be linearly independent. Then there exist bounded measurable functions $f_1,\ldots,f_k\in L^\infty(\Omega)$ such that
        \[
            \mu_i(f_j)=\ind(i=j)
            \qquad\text{for all } i,j\in[k].
        \]
    \end{lemma}

    Before proving Lemma~\ref{lem:measure-bidual}, recall that for a measurable space $(\Omega,\sF)$, \(ba(\Omega)\) denotes the Banach space of finitely additive signed measures on \((\Omega,\sF)\), equipped with the total variation norm. Such measures are not required to be countably additive. The usual space $\sM(\Omega)$ of finite signed (countably additive) measures is a norm-closed subspace of $ba(\Omega)$. The space $ba(\Omega)$ is used here only to justify the duality arguments appearing in the proof. The argument also relies on Goldstine's theorem \cite[Theorem~3.27]{fabian01}, which states that the closed unit ball of a Banach space is weak$^*$ dense in the closed unit ball of its bidual.

    \begin{proof}[Proof of Lemma~\ref{lem:measure-bidual}]
        It is well known that
        \[
            L^\infty(\Omega)^* = ba(\Omega),
        \]
        see \cite[Theorem~14.4]{aliprantis06}.
        Since $\mu_1,\ldots,\mu_k$ are linearly independent elements of $ba(\Omega)$, the Hahn--Banach theorem yields linear functionals $\mu_1^*,\ldots,\mu_k^*\in ba(\Omega)^*$ such that
        \[
            \mu_i^*(\mu_j)=\ind(i=j).
        \]
        By Goldstine's theorem, for every $\varepsilon>0$ there exist functions $\tilde f_1,\ldots,\tilde f_k\in L^\infty(\Omega)$ such that
        \[
            \bigl|\mu_i(\tilde f_j)-\mu_i^*(\mu_j)\bigr|\le \varepsilon
            \qquad\text{for all } i,j\in[k].
        \]
        Choosing $\varepsilon$ sufficiently small ensures that the matrix $\bigl[\mu_i(\tilde f_j)\bigr]_{i,j\in[k]}$ is invertible. Standard linear algebra then allows one to form suitable linear combinations of the functions $\tilde f_j$ to obtain functions $f_1,\ldots,f_k$ satisfying $\mu_i(f_j)=\ind(i=j)$.
    \end{proof}

    \begin{proof}[Proof of Proposition \ref{prop:rank-characterization}]
        Suppose that one of the families is linearly dependent. Without loss of generality assume that $\mu_{1,1},\ldots,\mu_{k,1}$ are linearly dependent. Then one of the measures, say $\mu_{p,1}$, can be written as a linear combination of the others. Substituting this relation into
        \[
            \mu=\sum_{i=1}^k \mu_{i,1}\times\mu_{i,2}
        \]
        and using bilinearity of the product measure eliminates one term, yielding a representation with at most $k-1$ terms. The same argument applies if $\mu_{1,2},\ldots,\mu_{k,2}$ are linearly dependent.

        Conversely, suppose that both families $\mu_{1,1},\ldots,\mu_{k,1}$ and $\mu_{1,2},\ldots,\mu_{k,2}$ are linearly independent. Assume for contradiction that
        \[
            \mu = \sum_{i=1}^{k} \mu_{i,1}\times\mu_{i,2}
            = \sum_{j=1}^{m} \nu_{j,1}\times\nu_{j,2}
        \]
        for some $m<k$ with $\nu_{j,\ell}\in\mathcal{M}(\Omega_\ell)$. By Lemma~\ref{lem:measure-bidual} applied to the family $\mu_{1,1},\ldots,\mu_{k,1}$, there exist bounded measurable functions $f_{1,1},\ldots,f_{k,1}\in L^\infty(\Omega_1)$ such that
        \[
            \mu_{i,1}(f_{j,1})=\ind(i=j)
            \qquad\text{for all } i,j\in[k].
        \]
        Applying the operator $T_{f_{p,1}}$ (defined in~\eqref{eqn:Tf-def}) to both representations of $\mu$ yields, for each $p\in[k]$,
        \[
            \mu_{p,2}
            = T_{f_{p,1}}\mu
            = \sum_{j=1}^{m} \nu_{j,1}(f_{p,1})\,\nu_{j,2}.
        \]
        Thus each $\mu_{p,2}$ lies in the linear span of $\nu_{1,2},\ldots,\nu_{m,2}$, which has dimension at most $m$. Since $m<k$, this contradicts the linear independence of $\mu_{1,2},\ldots,\mu_{k,2}$.
    \end{proof}

    \subsection{Laplace Transform Preliminaries} \label{sec:laplace-prelim}
    The proof of Theorem~\ref{thm:main} relies heavily on structural properties of Laplace transforms of measures. The required facts are collected and reviewed in this section. For a finite signed measure $\mu \in \sM(\R^d)$, the \emph{(two-sided) Laplace transform} is defined as
    \begin{equation}\label{eqn:def-laplace}
        \cL(\mu)(t) \;\triangleq\; \int_{\R^d} \exp(\langle x,t\rangle)\, d\mu(x),
        \qquad t \in \R^d,
    \end{equation}
    whenever the integral is finite.

    When $\mu$ is a probability measure, $\cL(\mu)$ coincides with the classical \emph{moment generating function} from probability theory. Throughout this work, Laplace transform terminology and notation are used exclusively; results stated in the literature in terms of moment generating functions are interpreted accordingly.

    Following standard usage in probability theory, the Laplace transform $\cL(\mu)$ is said to \emph{exist} if it is finite on some open neighborhood of the origin in $\R^d$.
    By definition of signed measure integration, this requires that the Laplace transforms of both $\mu^+$ and $\mu^-$ also exist and are finite on that set.

    The Laplace transform acts linearly on measures. In particular, if for some $t\in\R$ the quantities $\cL(\mu)(t)$ and $\cL(\mu')(t)$ both exist and are finite, then for all $c,c'\in\R$,
    \[
        \cL(c\mu + c'\mu')(t)
        =
        c\,\cL(\mu)(t) + c'\,\cL(\mu')(t).
    \]

    The following proposition shows that, on $\R$, equality of Laplace transforms on any set having an accumulation point at $0$ implies equality of the underlying measures.
    \begin{proposition}\label{prop:accumulation}
        Let \(\mu,\mu'\in\sMexp(\R)\).
        Then $\mu = \mu'$ if and only if there exists a set $A \subset \R$ with an accumulation point at $0$ such that
        \[
            \cL(\mu)(t) = \cL(\mu')(t) \quad \text{for all } t \in A.
        \]
    \end{proposition}

    The proof of Proposition~\ref{prop:accumulation} relies on the analyticity of Laplace transforms. Two supporting results are stated first.
    \begin{theorem}[{\cite[Theorem~11.8]{dasgupta2010}}]\label{thm:mgf-equal}
        Let $\mu,\mu'$ be probability measures on $\R^d$ such that at least one of their Laplace transforms exists in a neighborhood of the origin. Then $\mu=\mu'$ if and only if $\cL(\mu)$ and $\cL(\mu')$ agree on some neighborhood of $0$.
    \end{theorem}
    This can be extended to the following corollary.
    \begin{corollary}\label{cor:mgf-equal}
        Let \(\mu,\mu'\in\sM(\R^d)\), and suppose that at least one of them lies in \(\sMexp(\R^d)\).
        Then \(\mu=\mu'\) if and only if \(\cL(\mu)\) and \(\cL(\mu')\) agree on some open neighborhood of \(0\).
    \end{corollary}

    \begin{proof}[Proof of Corollary \ref{cor:mgf-equal}]
        Suppose that \(\cL(\mu)\) and \(\cL(\mu')\) agree on some open neighborhood of \(0\). Without loss of generality, assume that \(\mu\in\sMexp(\R^d)\). Let \(A\subset\R^d\) be a nonempty open set containing \(0\) such that \(\cL(\mu)\) is finite on \(A\) and
        \[
            \cL(\mu)(t)=\cL(\mu')(t)\qquad\text{for all }t\in A.
        \]
        Then $\cL(\mu')$ is also finite on $A$ (it equals $\cL(\mu)$ there). Set $\nu= \mu-\mu'$. Then
        \[
            \cL(\nu)(t)=\cL(\mu)(t)-\cL(\mu')(t)=0\qquad\text{for all }t\in A.
        \]

        Write the Jordan decomposition $\nu=\nu^+-\nu^-$ with $\nu^\pm$ finite nonnegative measures, and let $c_\pm= \nu^\pm(\R^d)$.

        Consider first the case \(c_+=0\) or \(c_-=0\).
        Then \(\nu\) is either nonnegative or nonpositive.
        For concreteness assume $\nu$ is nonnegative, the nonpositive case is analogous.
        Since $e^{\langle x,t\rangle}>0$ for all $x,t$, it then follows that the Laplace transform $\cL(\nu)$ is nonnegative on $A$.
        As $\cL(\nu)\equiv 0$ on $A$, it follows that
        \[
            \nu(\R^d)=\cL(\nu)(0)=0.
        \]
        A nonnegative measure with total mass zero is identically zero,
        so $\nu=0$.
        Hence $\nu=0$ and the result follows.

        Now consider the case \(c_+>0\) and \(c_->0\). Define the probability measures
        $P=\nu^+/c_+$ and $Q=\nu^-/c_-$.
        By the definition of integration with respect to a signed measure, the finiteness of \(\cL(\nu)\) on \(A\) implies that \(\cL(\nu^+)\) and \(\cL(\nu^-)\) are finite on \(A\). Hence \(\cL(P)\) and \(\cL(Q)\) are finite on \(A\), and by linearity,
        \[
            0=\cL(\nu)=c_+\cL(P)-c_-\cL(Q)\quad\text{on }A.
        \]
        Thus, for all \(t\in A\),
        \[
            c_+\cL(P)(t)=c_-\cL(Q)(t).
        \]
        By the definition of Laplace transform, evaluating at $t=0$ gives $c_+=c_->0$, hence
        \[
            \cL(P)(t) = \cL(Q)(t)\qquad \text{for all }t\in A.
        \]
        By Theorem~\ref{thm:mgf-equal}, this implies $P=Q$.  
        Thus $\nu^+=\nu^-$ and so $\nu=0$, i.e.\ $\mu=\mu'$.

        Conversely, suppose that \(\mu=\mu'\).
        Then \(\cL(\mu)=\cL(\mu')\) on some open neighborhood of \(0\).
    \end{proof}

    A second key ingredient in the proof of Proposition~\ref{prop:accumulation} is the analyticity of Laplace transforms. The following standard result, stated for probability measures, will be used.
    \begin{proposition}[{\cite[Exercise~8, Chapter 5]{fomp}}]\label{prop:mgf-analytic}
        For a probability measure $\mu$ on $\R$, if $\cL(\mu)$ is finite on some nonempty open interval $I\subseteq \R$ then it has an analytic extension on the strip $\left\{z \in \C: \Re(z) \in I\right\}$.
    \end{proposition}
    A corresponding analyticity statement holds for finite signed measures.
    \begin{corollary}\label{cor:mgf-analytic}
        For a finite signed measure $\mu$ on $\R$, if $\cL(\mu)$ is finite on some nonempty open interval $I\subseteq \R$ then it has an analytic extension on the strip
        $\left\{z \in \C: \Re(z) \in I\right\}$.
    \end{corollary}
    \begin{proof}[Proof of Corollary \ref{cor:mgf-analytic}]
        If $\mu=0$, then $\cL(\mu)=0$ and the claim is trivial.  

        Let $\mu=\mu^+-\mu^-$ be the Jordan decomposition, and write $c_\pm = \mu^\pm(\R)$.
        If $c_+=0$, then $\mu=-\mu^-$ and
        \[
            \cL(\mu) = -\cL(\mu^-)= -c_-\,\cL(\mu^-/c_-),
        \]
        where $\mu^-/c_-$ is a probability measure. Since $\cL(\mu)$ is finite on $I$, so is $\cL(\mu^-/c_-)$, and Proposition~\ref{prop:mgf-analytic} implies that $\cL(\mu^-/c_-)$ admits an analytic extension to the strip $\{z\in\C:\Re z\in I\}$. The same then holds for $\cL(\mu)$. The case $c_-=0$ is analogous.

        Assume now that $c_+>0$ and $c_->0$, and write
        \[
            \mu = c_+\,(\mu^+/c_+) - c_-\,(\mu^-/c_-),
        \]
        where $\mu^+/c_+$ and $\mu^-/c_-$ are probability measures.
        For $t\in I$, finiteness of $\cL(\mu)(t)$ implies that both $\cL(\mu^+)(t)$ and $\cL(\mu^-)(t)$ are finite, since by the definition of integration with respect to a signed measure the positive and negative parts must be finite individually. Hence $\cL(\mu^+/c_+)(t)$ and $\cL(\mu^-/c_-)(t)$ are also finite.

        Applying Proposition~\ref{prop:mgf-analytic} to the probability measures $\mu^+/c_+$ and $\mu^-/c_-$, it follows that their Laplace transforms admit analytic extensions to the strip $\{z\in\C:\Re z\in I\}$. By linearity of the Laplace transform,
        \[
            \cL(\mu)(z)
            =
            c_+\,\cL(\mu^+/c_+)(z) - c_-\,\cL(\mu^-/c_-)(z)
        \]
        on this strip, and hence $\cL(\mu)$ itself extends analytically to $\{z\in\C:\Re z\in I\}$.
    \end{proof}

    It is now possible to prove Proposition~\ref{prop:accumulation}.
    \begin{proof}[Proof of Proposition \ref{prop:accumulation}]
        Suppose that \(\mu=\mu'\).
        Then \(\cL(\mu)=\cL(\mu')\), and hence they agree on any set having an accumulation point at \(0\).

        Conversely, suppose that there exists a set \(A\subset\R\) having an accumulation point at \(0\) such that
        \[
            \cL(\mu)(t)=\cL(\mu')(t)
            \qquad\text{for all } t\in A.
        \]

        By Corollary~\ref{cor:mgf-analytic}, both \(\cL(\mu)\) and \(\cL(\mu')\) admit analytic extensions to a complex strip containing \(0\).
        Since they agree on a set having an accumulation point at \(0\), the Identity Theorem implies that they agree on some open neighborhood of \(0\).
        By Corollary~\ref{cor:mgf-equal}, it follows that \(\mu=\mu'\).
    \end{proof}
    The preceding section establishes the basic properties of the Laplace transform needed to prove Theorem~\ref{thm:main}. The remaining ingredient is the fact that, for sufficiently regular classes of signed measures, the full measure can be recovered when a countably infinite collection of linear pushforward measures agree. This may be viewed as an analogue of the Cramér--Wold theorem, which states that probability measures on $\R^d$ are determined by their one-dimensional linear pushforwards.

    The required regularity assumption is the \emph{Carleman condition}. A $d$-dimensional probability measure $\mu$ is said to satisfy the Carleman condition if its moments
    \[
        m_n = \int \|x\|_2^n\,d\mu(x)
    \]
    satisfy
    \[
        \sum_{n\in\mathbb N} m_n^{-1/n} = \infty.
    \]
    The Carleman condition guarantees that $\mu$ is uniquely determined by its moments. Moreover, any probability measure whose Laplace transform exists satisfies the Carleman condition \cite[bottom of p.~203]{cuesta-albertos07}, so this assumption is automatically satisfied in the present setting.

    An early version of the theorem used below was established in \cite{renyi52} for compactly supported distributions, extended to densities satisfying the Carleman condition in \cite{gilbert55}, and generalized to $d$-dimensional measures in \cite{cuesta-albertos07}. Let $S^1$ denote the unit circle in $\R^2$.
    \begin{proposition}[\cite{gilbert55}]\label{prop:infinite-projections}
        Let $\mu$ and $\mu'$ be probability measures on $\R^2$, with $\mu$ satisfying the Carleman condition (in particular, this holds whenever $\mu$ has an existing Laplace transform).  
        Then \(\mu=\mu'\) if and only if, for \(X\sim\mu\) and \(X'\sim\mu'\),
        \[
            \law(\langle X,c\rangle)
            =
            \law(\langle X',c\rangle),
        \]
        for countably infinitely many \(c\in S^1\).
    \end{proposition}
    The next corollary extends the previous result from probability measures to finite signed measures.
    \begin{corollary}[Signed–measure version of \cite{gilbert55}]\label{cor:infinite-projections-signed}
        Let $\mu,\mu' \in \sMexp(\R^2)$. Then $\mu=\mu'$ if and only if
        \[
            \pi_c\#\mu \;=\; \pi_c\#\mu',
        \]
        for countably infinitely many \(c\in S^1\).
    \end{corollary}

    \begin{proof}[Proof of Corollary~\ref{cor:infinite-projections-signed}]
        Suppose that \(\mu=\mu'\).
        Then \(\pi_c\#\mu=\pi_c\#\mu'\) for all \(c\in S^1\), and in particular for every \(c\) in any countably infinite subset of \(S^1\).

        Conversely, suppose that there exists a countably infinite set \(A\subseteq S^1\) such that
        \[
            \pi_c\#\mu=\pi_c\#\mu'
        \]
        for all \(c\in A\).
        Write the Jordan decompositions
        \[
            \mu=\mu^+-\mu^-,
            \qquad
            \mu'=\mu'^+-\mu'^-.
        \]
        Note that the nonnegative measures
        \(\mu^+,\mu^-,\mu'^+,\mu'^-\)
        all belong to \(\sMexp(\R^2)\).
        By linearity of the pushforward operator, for every \(c\in A\),
        \begin{equation}\label{eqn:cor-inf-proj-signed}
            \pi_c\#\mu^+ - \pi_c\#\mu^-
            \;=\;
            \pi_c\#\mu'^+ - \pi_c\#\mu'^-.
        \end{equation}
        Consider first the case where one of $\mu$ or $\mu'$ is the zero measure.
        Without loss of generality, assume that \(\mu'=0\). Then
        \[
            \pi_c\#\mu^+ = \pi_c\#\mu^-
        \]
        for all \(c\in A\).

        Assume for contradiction that \(\mu^+\neq0\). Since \(\mu^+\) is a nonnegative measure and
        \[
            \mu^+(\R^2)
            =
            (\pi_c\#\mu^+)(\R)
            =
            (\pi_c\#\mu^-)(\R)
            =
            \mu^-(\R^2)
        \]
        for every \(c\in A\), it follows that \(\mu^-(\R^2)=\mu^+(\R^2)>0\), and hence \(\mu^-\neq0\).

        Define
        \[
            P=\frac{\mu^+}{\mu^+(\R^2)},
            \qquad
            Q=\frac{\mu^-}{\mu^-(\R^2)}.
        \]
        Then \(P\) and \(Q\) are probability measures satisfying
        \[
            \pi_c\#P=\pi_c\#Q
        \]
        for all \(c\in A\).
        Since \(\mu\in\sMexp(\R^2)\), the measures \(\mu^+\) and \(\mu^-\) also belong to \(\sMexp(\R^2)\). Consequently, \(P\) and \(Q\) satisfy the Carleman condition. Proposition~\ref{prop:infinite-projections} therefore implies that \(P=Q\), and hence \(\mu^+=\mu^-\). Thus \(\mu=0\), a contradiction. This establishes the result whenever one of \(\mu\) or \(\mu'\) is the zero measure.

        Returning to \eqref{eqn:cor-inf-proj-signed}, assume that both \(\mu\) and \(\mu'\) are nonzero. For every \(c\in A\),
        \[
            \pi_c\#(\mu^+ + \mu'^-)
            =
            \pi_c\#(\mu^- + \mu'^+).
        \]
        Define the finite nonnegative measures
        \[
            \alpha=\mu^+ + \mu'^-,
            \qquad
            \beta=\mu^- + \mu'^+,
        \]
        which belong to \(\sMexp(\R^2)\).
        Then
        \[
            \pi_c\#\alpha
            =
            \pi_c\#\beta
        \]
        for all \(c\in A\).
        It therefore follows that
        \[
            \alpha(\R^2)
            =
            (\pi_c\#\alpha)(\R)
            =
            (\pi_c\#\beta)(\R)
            =
            \beta(\R^2)
        \]
        for every \(c\in A\). Since both \(\mu\) and \(\mu'\) are nonzero, the measures \(\alpha\) and \(\beta\) are also nonzero. Define
        \[
            \bar\alpha
            =
            \frac{\alpha}{\alpha(\R^2)},
            \qquad
            \bar\beta
            =
            \frac{\beta}{\beta(\R^2)}.
        \]
        Then \(\bar\alpha\) and \(\bar\beta\) are probability measures satisfying
        \[
            \pi_c\#\bar\alpha
            =
            \pi_c\#\bar\beta
        \]
        for all \(c\in A\).

        Since \(\bar\alpha,\bar\beta\in\sMexp(\R^2)\), both satisfy the Carleman condition. Proposition~\ref{prop:infinite-projections} therefore implies that
        \[
            \bar\alpha=\bar\beta.
        \]
        Since \(\alpha(\R^2)=\beta(\R^2)\), it follows that \(\alpha=\beta\).

        Finally,
        \[
            \mu-\mu'
            =
            (\mu^+-\mu^-)-(\mu'^+-\mu'^-)
            =
            (\mu^+ + \mu'^-) - (\mu^- + \mu'^+)
            =
            \alpha-\beta
            =
            0,
        \]
        and hence \(\mu=\mu'\).
    \end{proof}

    The proof of Theorem~\ref{thm:main} will additionally require a measure-theoretic analogue of the familiar identity from probability theory that the moment-generating function of a linear combination of independent random variables factorizes. Establishing this identity for finite signed measures requires a substitution formula for pushforward measures. The following standard result is first recalled for finite nonnegative measures.
    \begin{lemma}[{\cite[Lemma~1.22]{fomp}}]\label{lem:substitution}
        Consider a finite nonnegative measure space $(\Omega,\sF,\mu)$, a measurable space $(S,\sS)$, and measurable mappings
        $f:\Omega \to S$ and $g: S \to \R$. Then
        \[
            \int g(f(x))\,d\mu(x)
            \;=\;
            \int g(y)\, d(f\#\mu)(y),
        \]
        whenever either side exists (possibly as an extended real value); that is, if one side exists, then so does the other, and the two are equal.
    \end{lemma}
    The following lemma extends Lemma~\ref{lem:substitution} to finite signed measures.
    \begin{lemma}\label{lem:substitution-signed}
        Consider a finite signed measure space $(\Omega,\sF,\mu)$, a measurable space $(S,\sS)$, and measurable mappings
        $f:\Omega \to S$ and $g:S\to\R$. Then
        \[
            \int g(f(x))\,d\mu(x)
            \;=\;
            \int g(y)\, d(f\#\mu)(y),
        \]
        whenever the left-hand side is finite.
    \end{lemma}
    \begin{proof}[Proof of Lemma~\ref{lem:substitution-signed}]
        Write the Jordan decomposition $\mu=\mu^+-\mu^-$, where $\mu^\pm$ are finite nonnegative measures.
        Since pushforward is linear on finite signed measures, it follows that
        \[
            f\#\mu \;=\; f\#\mu^+ \;-\; f\#\mu^-.
        \]

        By the definition of integration with respect to a signed measure,
        $\int g(f(x))\,d\mu^+(x)$ and $\int g(f(x))\,d\mu^-(x)$ are both finite. Applying Lemma~\ref{lem:substitution} to the nonnegative measures $\mu^\pm$,
        \[
            \int g(f(x))\,d\mu^\pm(x)
            \;=\;
            \int g(y)\,d(f\#\mu^\pm)(y).
        \]
        Subtracting the two identities yields
        \[
            \int g(f(x))\,d\mu(x)
            =
            \int g(y)\,d(f\#\mu)(y).
        \]
    \end{proof}
    The primary application of Lemma~\ref{lem:substitution-signed} in this work is to linear projections. For \(\nu\in\sM(\R^d)\), \(c\in\R^d\), and \(\pi_c(x)=\langle c,x\rangle\),
    \[
        \int_{\R} g(y)\,d(\pi_c\#\nu)(y)
        =
        \int_{\R^d} g(\langle c,x\rangle)\,d\nu(x),
    \]
    whenever the right-hand side is finite. The pushforward operator is also linear on finite signed measures:
    \[
        f\#(a\mu+b\nu)
        =
        a\,f\#\mu+b\,f\#\nu.
    \]

    The next lemma combines these observations with the product structure of low-rank measures to obtain a factorization formula for Laplace transforms.
    \begin{lemma}[Laplace transform of a factorizable measure]
        \label{lem:laplace-product}
        Let $\mu,\nu\in\sMexp(\R)$ and let $a,b\in\R$.
        For any $t$ such that $\cL(\mu)(at)$ and $\cL(\nu)(bt)$ are finite,
        \[
            \cL\!\bigl(\pi_{[a,b]}\#(\mu\times\nu)\bigr)(t)
            =
            \cL(\mu)(at)\,\cL(\nu)(bt).
        \]
    \end{lemma}

    \begin{proof}[Proof of Lemma \ref{lem:laplace-product}]
        Fix \(t\in\R\) such that \(\cL(\mu)(at)\) and \(\cL(\nu)(bt)\) are finite. Write the Jordan decompositions
        \[
            \mu=\mu^+-\mu^-,
            \qquad
            \nu=\nu^+-\nu^-.
        \]
        Let \(f(x)=e^{atx}\) and \(g(y)=e^{bty}\). Then
        \begin{align}
            \cL(\mu)(at)\,\cL(\nu)(bt)
    &=
    \left(\int f\,d\mu^+ - \int f\,d\mu^-\right)
    \left(\int g\,d\nu^+ - \int g\,d\nu^-\right)\\
    &=
    \left(\int f\,d\mu^+\right)\left(\int g\,d\nu^+\right)
    +
    \left(\int f\,d\mu^-\right)\left(\int g\,d\nu^-\right)\\
    &\quad
    -
    \left(\int f\,d\mu^+\right)\left(\int g\,d\nu^-\right)
    -
    \left(\int f\,d\mu^-\right)\left(\int g\,d\nu^+\right),
    \label{eqn:laplace-distributed}
        \end{align}
        where all terms are finite.
        Hence, by the product-measure factorization identity
        \cite[Chapter~2, Exercise~51(b)]{folland99}, \eqref{eqn:laplace-distributed} becomes
        \begin{align*}
            \cL(\mu)(at)\,\cL(\nu)(bt)
    &=
    \int_{\R^2} f(x)g(y)\,d(\mu^+\times\nu^+)(x,y)
    +
    \int_{\R^2} f(x)g(y)\,d(\mu^-\times\nu^-)(x,y)\\
    &\quad
    -
    \int_{\R^2} f(x)g(y)\,d(\mu^+\times\nu^-)(x,y)
    -
    \int_{\R^2} f(x)g(y)\,d(\mu^-\times\nu^+)(x,y).
        \end{align*}
        By the definition of integration with respect to a signed measure and the Jordan decomposition
        \[
            \mu\times\nu
            =
            \bigl(\mu^+\times\nu^+ + \mu^-\times\nu^-\bigr)
            -
            \bigl(\mu^+\times\nu^- + \mu^-\times\nu^+\bigr),
        \]
        it follows that
        \[
            \cL(\mu)(at)\,\cL(\nu)(bt)
            =
            \int_{\R^2} f(x)g(y)\,d(\mu\times\nu)(x,y).
        \]
        Substituting \(f(x)=e^{atx}\) and \(g(y)=e^{bty}\), this becomes
        \[
            \cL(\mu)(at)\,\cL(\nu)(bt)
            =
            \int_{\R^2} e^{t(ax+by)}\,d(\mu\times\nu)(x,y).
        \]
        The equality above shows that the right-hand side is finite. Therefore Lemma~\ref{lem:substitution-signed}, applied to the projection \(\pi_{[a,b]}\), yields
        \[
            \int_{\R^2} e^{t(ax+by)}\,d(\mu\times\nu)(x,y)
            =
            \cL\!\bigl(\pi_{[a,b]}\#(\mu\times\nu)\bigr)(t).
        \]
        Hence
        \[
            \cL\!\bigl(\pi_{[a,b]}\#(\mu\times\nu)\bigr)(t)
            =
            \cL(\mu)(at)\,\cL(\nu)(bt).
        \]
    \end{proof}

    The following corollary makes explicit the structural property that will be central to the proof of Theorem~\ref{thm:main}.
    \begin{corollary}[Linearity for finite mixtures]
        \label{cor:laplace-mixture}
        Let
        \[
            \lambda=\sum_{i=1}^r w_i\,(\mu_i\times\nu_i),
            \qquad w_i\in\R,
        \]
        where each \(\mu_i,\nu_i\in\sMexp(\R)\). Then
        \(\lambda\in\sMexp(\R^2)\). Moreover, for every \(a,b\in\R\), there exists
        \(\alpha>0\) such that
        \[
            \cL\!\bigl(\pi_{[a,b]}\#\lambda\bigr)(t)
            =
            \sum_{i=1}^r
            w_i\,\cL(\mu_i)(at)\,\cL(\nu_i)(bt)
        \]
        for all \(t\in(-\alpha,\alpha)\).
    \end{corollary}

    \subsection{A Rank Criterion for Laplace-Transform Matrices and Proofs of Theorem~\ref{thm:main} and Corollary~\ref{cor:main-multiview}}
    The next proposition provides a key structural ingredient in the proof of Theorem~\ref{thm:main}. It shows that linear independence of a finite collection of measures can be characterized through a full-rank condition on a structured matrix of Laplace-transform evaluations.

    \begin{proposition}\label{prop:full-rank}
        Let $k \in \mathbb{N}$ and let $\mu_1,\dots,\mu_k \in \sMexp(\R)$.
        For each $i=1,\dots,k$, write $\cL_i(t)= \cL(\mu_i)(t)$.

        Then the measures $\mu_1,\ldots,\mu_k$ are linearly independent if and only if for all nonzero $a,b \in \R$ with $|a|\neq 1$, there exists $\alpha>0$ such that, for all $t \in (-\alpha,0)\cup(0,\alpha)$, the matrix    
        \begin{equation}
            L(t) \triangleq
            \begin{bmatrix}
                \cL_1(a^0b t) & \cL_1(a^1b t) & \cdots & \cL_1(a^{k-1}b t)\\
                \cL_2(a^0b t) & \cL_2(a^1b t) & \cdots & \cL_2(a^{k-1}b t)\\
                \vdots       & \vdots       & \ddots & \vdots\\
                \cL_k(a^0b t) & \cL_k(a^1b t) & \cdots & \cL_k(a^{k-1}b t)
            \end{bmatrix}
            \label{eqn:full-rank-matrix}
        \end{equation}
        has full rank.
    \end{proposition}
    The matrix in \eqref{eqn:full-rank-matrix} has a Vandermonde-type structure in the scaling parameter $a$, with entries given by evaluations of Laplace transforms at geometrically spaced points.

    For a matrix \(M\) and index sets \(I,J\subset\mathbb N\), let \(M_{I\times J}\) denote the submatrix of \(M\) obtained by restricting to the rows indexed by \(I\) and the columns indexed by \(J\).

    \begin{proof}[Proof of Proposition \ref{prop:full-rank}]
        Suppose that \(\mu_1,\ldots,\mu_k\) are linearly dependent. Since the Laplace transform is linear, the corresponding functions \(\cL_1,\ldots,\cL_k\) are also linearly dependent. It follows that, for every \(t\), the rows of \(L(t)\) satisfy the same linear relation and hence \(L(t)\) does not have full rank anywhere all entries exist.

        Conversely, suppose that \(\mu_1,\ldots,\mu_k\) are linearly independent. The proof proceeds by induction on \(k\).

        \smallskip
        \noindent\textbf{Base case (\(k=1\)).}
        If no such $\alpha$ exists then there exists a sequence $t_i\to 0$ such that $L(t_i) = 0$ and $\mu_1 \equiv 0$ due to Proposition \ref{prop:accumulation}, contradicting linear independence. Note that the preceding argument shows more generally that if
        \(\mu\not\equiv0\) and its Laplace transform exists, then there exists \(\beta>0\) such that \( \cL(\mu)(t)\neq0 \) for all \(t\in(-\beta,\beta)\setminus\{0\}\).

        \smallskip
        \noindent\textbf{Induction step.}
        Assume that the proposition holds for some \(k\ge 1\). The result is established for \(k+1\). First consider the case \(|a|<1\) and \(b=1\); the cases \(|a|>1\) and \(b\neq1\) are handled afterward.

        Assume for the sake of contradiction that no such $\alpha$ exists. Then there is a sequence $t_i\neq 0$ with $t_i\to 0$ such that $L(t_i)$ is rank-deficient for all $i$.
        Since $t\mapsto \det L(t)$ is analytic near $0$, it follows that $\det L\equiv 0$ on some open interval $I$ containing 0, by the Identity Theorem. Shrinking \(I\) if necessary, assume that \(I\) is symmetric about zero.

        By the inductive hypothesis there exists $\alpha' > 0$ such that the minor $L(t)_{\{1,\ldots,k\}\times\{1,\ldots,k\}}$ (the upper-left $k\times k$ block of $L(t)$) has full rank for all $t \in (-\alpha', \alpha') \setminus \{0\}$.

        Fix $t^\ast\in \big((-\alpha',\alpha')\setminus\{0\}\big)\cap I$.
        Then $L(t^\ast)$ is rank-deficient while $L(t^\ast)_{\{1,\ldots,k\}\times \{1,\ldots,k\}}$ has full rank.
        Hence there exists a unique $\gamma\in\R^k$ with
        \begin{equation}\label{eqn:row-combo-1k}
            L(t^\ast)_{\{k+1\}\times\{1,\ldots,k\}}=\sum_{i=1}^k \gamma_i\,L(t^\ast)_{\{i\}\times \{1,\ldots,k\}}.
        \end{equation}
        Because $L(t^\ast)$ is rank-deficient while $L(t^\ast)_{\{1,\ldots,k\}\times \{1,\ldots,k+1\}}$ has full rank, the $(k\!+\!1)$-st row lies in the span of the first $k$ rows and hence
        \[
            L(t^\ast)_{\{k+1\}\times \{1,\ldots,k+1\}}=\sum_{i=1}^k \gamma_i\,L(t^\ast)_{\{i\}\times \{1,\ldots,k+1\}}.
        \]
        Within the present contradiction argument, it will be shown by induction on \(j\) that
        \begin{equation}\label{eqn:induction-target}
            L(a^j t^\ast)_{\{k+1\}\times \{1,\ldots,k+1\}}=\sum_{i=1}^k \gamma_i\,L(a^j t^\ast)_{\{i\}\times \{1,\ldots,k+1\}},
        \end{equation}
        for all \(j\in\mathbb N\cup\{0\}\).

        It has already been shown that \eqref{eqn:induction-target} holds for \(j=0\), with \(\gamma\) unique. Assume that \eqref{eqn:induction-target} holds for some \(j\in\mathbb N\cup\{0\}\). It will now be shown that \eqref{eqn:induction-target} also holds for \(j+1\).

        From the definition of $L$,
        \[
            L(a^{j}t^\ast)_{\{1,\ldots,k\}\times\{2,\ldots,k+1\}}
            =
            L(a^{j+1}t^\ast)_{\{1,\ldots,k\}\times\{1,\ldots,k\}}.
        \]
        Since \(|a|<1\),
        \[
            a^{j+1}t^\ast \in \bigl((-\alpha',\alpha')\setminus\{0\}\bigr)\cap I.
        \]
        By the induction hypothesis on \(k\), both \(k\times k\) matrix minors appearing in the equality above have full rank. By the inductive hypothesis on \(j\), \(\gamma\) satisfies the following equation and, by the established full-rank property, is its unique solution,            \[
            L(a^jt^\ast)_{\{k+1\}\times\{2,3,\dots,k+1\}}=\sum_{i=1}^k \gamma_i\,L(a^j t^\ast)_{\{i\}\times\{2,3,\dots,k+1\}}.
        \]
        Now observe that, similarly to before,
        \[
            L(a^{j+1} t^\ast)_{\{1,\ldots,k+1\}\times\{1,\ldots,k\}}=L(a^jt^\ast)_{\{1,\ldots,k+1\}\times\{2,3,\dots,k+1\}},
        \]
        hence the same $\gamma$ uniquely satisfies
        \begin{equation}\label{eqn:induction-unique-gamma}
            L(a^{j+1} t^\ast)_{\{k+1\}\times\{1,\ldots,k\}}=\sum_{i=1}^k \gamma_i\,L(a^{j+1} t^\ast)_{\{i\}\times \{1,\ldots,k\}}.
        \end{equation}
        Because
        \(
        a^{j+1}t^\ast \in \bigl((-\alpha',\alpha')\setminus\{0\}\bigr)\cap I,
        \)
        the matrix \(L(a^{j+1}t^\ast)\) is rank deficient. Moreover, the first \(k\) rows are linearly independent, since
        \[
            L(a^{j+1}t^\ast)_{\{1,\ldots,k\}\times\{1,\ldots,k\}}
        \]
        has full rank. Therefore the \((k+1)\)-st row of \(L(a^{j+1}t^\ast)\) lies in the span of the first \(k\) rows. By the uniqueness established in \eqref{eqn:induction-unique-gamma}, the coefficients in this linear relation must equal \(\gamma\). It follows that \eqref{eqn:induction-unique-gamma} remains valid when the \((k+1)\)-st column is included:
        \begin{equation}
            L(a^{j+1} t^\ast)_{\{k+1\}\times\{1,\ldots,k+1\}}
            =
            \sum_{i=1}^k \gamma_i\,L(a^{j+1} t^\ast)_{\{i\}\times \{1,\ldots,k+1\}}.
        \end{equation}
        This completes the induction on \(j\). Thus \eqref{eqn:induction-target} holds for all $j$ with $\gamma$ unique. Looking at the entries of the matrix function $L$, from \eqref{eqn:induction-target} it follows that $\cL_{k+1}(a^jt^*) =\sum_{i=1}^k \gamma_i \cL_i(a^jt^*)$ for all $j \in \N$. 

        Since \(0<|a|<1\) and \(t^\ast\neq 0\), it follows that \(a^j t^\ast\to 0\). Proposition~\ref{prop:accumulation} therefore implies that \[ \mu_{k+1}=\sum_{i=1}^k \gamma_i\,\mu_i, \] contradicting the linear independence of \(\{\mu_1,\dots,\mu_{k+1}\}\). This completes the induction step and establishes the result for the case \(|a|<1\), \(b=1\).

        For the remaining cases, let \(L_a(t)\) denote the matrix in \eqref{eqn:full-rank-matrix} with \(b=1\) and parameter \(a\). Then \(L_a(t)\) is obtained from \(L_{a^{-1}}(a^{k-1}t)\) by reversing the order of the columns:
        \[
            L_a(t) \text{ is the column-reversal of } L_{a^{-1}}(a^{k-1}t).
        \]

        Now assume $|a|>1$ with $b=1$, so that $|a^{-1}|<1$. Since $|a^{-1}|<1$, the $|a|<1$ case yields $\alpha_0>0$ such that $L_{a^{-1}}(t)$ has full rank for all $t\in(-\alpha_0,\alpha_0)\setminus\{0\}$.
        Define
        \[
            \alpha \triangleq \alpha_0\,|a|^{-(k-1)}.
        \]
        Then for every $t\in(-\alpha,\alpha)\setminus\{0\}$ one has $a^{k-1}t\in(-\alpha_0,\alpha_0)\setminus\{0\}$, and hence $L_{a^{-1}}(a^{k-1}t)$ has full rank. Since $L_a(t)$ is obtained from $L_{a^{-1}}(a^{k-1}t)$ by a horizontal column reversal, which preserves rank, it follows that $L_a(t)$ has full rank for all $t\in(-\alpha,\alpha)\setminus\{0\}$, thus taking care of the $|a|>1$ case. 

        Now let $L_{a,b}(t)$ denote the matrix in \eqref{eqn:full-rank-matrix} with both parameters $a$ and $b$ free. For all admissible $a,b$,
        \[
            L_{a,b}(t)=L_{a,1}(bt).
        \]
        From the case $b=1$, there exists $\alpha_0>0$ such that $L_{a,1}(s)$ has full rank for all $s\in(-\alpha_0,\alpha_0)\setminus\{0\}$. Define $\alpha\triangleq \alpha_0/|b|$. Then every $t\in(-\alpha,\alpha)\setminus\{0\}$ satisfies $bt\in(-\alpha_0,\alpha_0)\setminus\{0\}$. Hence $L_{a,b}(t)=L_{a,1}(bt)$ has full rank.
    \end{proof}

    The proof of Theorem~\ref{thm:main} now follows.

    \begin{proof}[Proof of Theorem \ref{thm:main}]
        The proof proceeds by induction, establishing
        \begin{equation}\label{eqn:main-projections-target}
            \pi_{[b_1,b_2 a^m]^\top}\#\mu 
            \;=\;
            \pi_{[b_1,b_2 a^m]^\top}\#\mu',
        \end{equation}
        for all \(m\in\mathbb N\). Corollary~\ref{cor:infinite-projections-signed} then implies \(\mu=\mu'\), since \(\mu,\mu'\in\sMexp(\R^2)\) by Corollary~\ref{cor:laplace-mixture}.

        \smallskip
        \noindent\textbf{Base cases.} By the hypotheses, \eqref{eqn:main-projections-target} holds for $0 \le m \le k-1$.  

        \smallskip
        \noindent\textbf{Induction step.}  For $t,x,y \in \R$, define 
        \[
            \cL_{[x,y]}(t) \;=\; \cL\bigl(\pi_{[x,y]}\#\mu\bigr)(t) \text{ and } 
            \cL'_{[x,y]}(t) \;=\; \cL\bigl(\pi_{[x,y]}\#\mu'\bigr)(t).
        \]
        Assume that \eqref{eqn:main-projections-target} holds for all $0 \le m' \le m$ with $m \ge k-1$. 
        From Corollary \ref{cor:laplace-mixture} and the definition of $\sMexp(\R)$, there exists $\alpha > 0$ such that every entry in the $(k+1)\times(k+1)$ matrix $L(t)$ is finite for all $t$ with $|t| < \alpha$,
        \begin{equation}\label{eqn:big-L}
            L(t) =
            \begin{bmatrix}
                \cL_{[b_1a^0,b_2a^0]}(t) & \cL_{[b_1 a^0,b_2 a^1]}(t) & \cdots & \cL_{[b_1 a^0, b_2 a^{k-1}]}(t) & \boxed{\cL_{[b_1 a^0,b_2 a^{m+1}]}(t)} \\
                \cL_{[b_1 a^1,b_2]}(t) & \cL_{[b_1 a^1,b_2 a^1]}(t) &  \cdots & \cL_{[b_1 a^1,b_2 a^{k-1}]}(t) & \cL_{[b_1a^1,b_2 a^{m+1}]}(t) \\
                \vdots & \vdots & \ddots & \vdots &  \vdots \\
                \cL_{[b_1 a^{k},b_2 a^0]}(t) & \cL_{[b_1 a^k,b_2 a^1]}(t) &  \cdots & \cL_{[b_1 a^k,b_2a^{k-1}]}(t) & \cL_{[b_1 a^{k},b_2a^{m+1}]}(t),
            \end{bmatrix}
        \end{equation}
        with $L'(t)$ defined analogously for $\mu'$. 
        By the theorem hypothesis together with the inductive hypothesis, all entries of \(L(t)\) and \(L'(t)\) coincide for \(|t|< \alpha\), except possibly the boxed entry.
        Note that some hypotheses may need to be used in rescaled form: for suitable integers $q,r,s$,  
        \[
            [b_1 a^q,b_2 a^r]^\top = a^s[b_1 a^{q-s},b_2 a^{r-s}]^\top,
        \]
        so the vectors $[b_1 a^q,b_2 a^r]^\top$ and $[b_1 a^{q-s},b_2 a^{r-s}]^\top$ are collinear and rescaling gives
        \[
            \cL_{[b_1 a^q,b_2 a^r]}(t) = \cL'_{[b_1 a^q,b_2 a^r]}(t)
            \;\;\Longleftrightarrow\;\;
            \cL_{[b_1 a^{q-s},b_2 a^{r-s}]}(a^s t)
            = \cL'_{[b_1 a^{q-s},b_2 a^{r-s}]}(a^s t).
        \]
        It will now be shown that the boxed entry of \(L(t)\) agrees with the corresponding entry of \(L'(t)\) for all \(|t|<\alpha\). This establishes \eqref{eqn:main-projections-target}.

        The proof first treats the case in which \(\mu\) has rank exactly \(k\), meaning that the measures \(\mu_{1,1},\ldots,\mu_{k,1}\) are linearly independent, as are \(\mu_{1,2},\ldots,\mu_{k,2}\) (Proposition \ref{prop:rank-characterization}). The lower-rank case is addressed afterward.

        The argument proceeds in two steps: first, it is shown that there exists \(\alpha'>0\) such that \(L(t)\) and \(L'(t)\) both have rank exactly \(k\) for all \(t\in(-\alpha',\alpha')\setminus\{0\}\); second, it is shown that, for such \(t\), the boxed entry is uniquely determined by the remaining entries together with the requirement that the matrix have rank \(k\). Consequently, the boxed entries of \(L(t)\) and \(L'(t)\) must coincide.

        Letting $\cL_{i,j} = \cL(\mu_{i,j})$, define the vector-valued functions
        \[
            U_i(t) \;=\;
            \begin{bmatrix}
                \cL_{1,i}(t) &
                \cL_{2,i}(t) &
                \cdots &
                \cL_{k,i}(t)
            \end{bmatrix}^{\top},
        \]
        so that Corollary~\ref{cor:laplace-mixture} yields
        \[
            \cL_{[a_1,a_2]}(t) = U_1(a_1 t)^{\top} U_2(a_2 t).
        \]
        Define $U'_i(t)$ analogously for $\mu'$.

        Observe that, for $|t|<\alpha$,
        \[
            L(t) =
            \begin{bmatrix}
                U_1(b_1a^0t)^\top\\
                U_1(b_1a^1t)^\top\\
                \vdots\\
                U_1(b_1a^{k}t)^\top
            \end{bmatrix}
            \begin{bmatrix}
                U_2(b_2 a^0 t) & U_2(b_2 a^1t) & \cdots & U_2(b_2 a^{k-1}t) & U_2(b_2 a^{m+1}t)
            \end{bmatrix}.
        \]
        Since each $U_i:\R\to\R^k$, the matrix $L(t)$ has rank at most $k$ for all $|t|<\alpha$.
        The same conclusion holds for $L'(t)$ by the corresponding argument applied to $\mu'$.

        It will now be shown that there exists $\alpha'>0$ such that, for all $t \in (-\alpha',\alpha')\setminus\{0\}$, the $k\times k$ submatrix
        \begin{equation}\label{eqn:show-full-rank}
            L_{\{2,3,4,\ldots,k+1\}\times [k]}(t) =
            \begin{bmatrix}
                U_1(b_1a^1t)^\top\\
                \vdots\\
                U_1(b_1 a^{k}t)^\top
            \end{bmatrix}
            \begin{bmatrix}
                U_2(b_2 a^0 t) & U_2(b_2 a^1t) & \cdots & U_2(b_2 a^{k-1}t)
            \end{bmatrix}
        \end{equation}
        has full rank $k$.
        This submatrix does not involve the boxed entry and therefore agrees with the corresponding submatrix of $L'(t)$.

        To establish this, first observe that the transpose of the right-hand factor in \eqref{eqn:show-full-rank} is given by
        \begin{align}\notag
            \left[
                \begin{matrix}
                    U_2(b_2a^0 t)^\top\\
                    U_2(b_2a^1 t)^\top\\
                    \vdots\\
                    U_2(b_2a^{k-1} t)^\top
                \end{matrix}
            \right] 
    &= 
    \left[
        \begin{matrix}
            \cL_{1,2}(b_2 a^0 t) & \cdots &  \cL_{k,2}(b_2 a^0 t)\\
            \cL_{1,2}(b_2 a^1 t) & \cdots & \cL_{k,2}(b_2 a^1 t)\\
            \vdots &\ddots&\vdots\\
            \cL_{1,2}(b_2 a^{k-1} t) & \cdots & \cL_{k,2}(b_2 a^{k-1} t)
        \end{matrix}
    \right].
    \label{eqn:show-full-rank-multiplicand}
        \end{align}
        By Proposition~\ref{prop:full-rank}, this matrix has full rank on some punctured open interval about $0$.

        An identical argument applies to the left-hand factor in \eqref{eqn:show-full-rank}, upon replacing $t$ by $a t$. Consequently, there exists $\alpha'>0$ such that both factors exist and have full rank for all $t \in (-\alpha',\alpha')\setminus\{0\}$. Hence the product matrix \eqref{eqn:show-full-rank} has full rank $k$ throughout this punctured neighborhood. Hence, for each fixed $t$ there exists a unique solution, $\gamma(t) \in \R^k$, to
        \begin{equation}\label{eqn:one-solution}
            L_{\{2,3,4,\ldots,k+1\}\times [k]}(t)\,\gamma(t)
            = L_{\{2,3,4,\ldots,k+1\}\times \{k+1\}}(t),
        \end{equation}
        that is, solving for the rightmost column of $L(t)$ with the first row removed. The same conclusion holds for $L'(t)$, with the identical $\gamma(t)$, since only entries where $L(t)$ and $L'(t)$ coincide are involved. Thus $\gamma(t)$ is the unique solution to  
        \begin{equation*}
            L'_{\{2,3,4,\ldots,k+1\}\times [k]}(t)\,\gamma(t)
            = L'_{\{2,3,4,\ldots,k+1\}\times \{k+1\}}(t).
        \end{equation*}   
        Since \(L(t)\) has rank at most \(k\), its rightmost column must lie in the span of its first \(k\) columns. The coefficients expressing this column are uniquely determined by the restricted system \eqref{eqn:one-solution}, and hence are \(\gamma(t)\). Therefore
        \[
            L_{[k+1]\times [k]}(t)\,\gamma(t)
            =
            L_{[k+1]\times \{k+1\}}(t).
        \]

        The same argument applied to \(L'(t)\) gives
        \[
            L'_{[k+1]\times [k]}(t)\gamma(t)
            =
            L'_{[k+1]\times\{k+1\}}(t).
        \]
        Since \(L(t)\) and \(L'(t)\) agree on all entries appearing in the left-hand sides, the right-hand sides agree as well. In particular, the boxed entries coincide.

        Thus the theorem holds when \(\mu\) has rank exactly \(k\), where \(k\) is the parameter in the theorem statement. It remains to consider the case $\operatorname{rank}(\mu)=\ell<k$.
        Write
        \[
            \mu=\sum_{i=1}^{\ell}\mu_{i,1}\times\mu_{i,2},
        \]
        with the component families linearly independent.
        It is first shown that $\operatorname{rank}(\mu')<k$ as well.
        Indeed, if $\operatorname{rank}(\mu')=k$, then the already-proved case of the theorem
        applies with $\mu'$ in place of $\mu$, yielding $\mu'=\mu$.
        This contradicts $\operatorname{rank}(\mu)=\ell<k$.
        Hence $\operatorname{rank}(\mu')=\ell'$ for some $\ell'<k$, and one may write
        \[
            \mu'=\sum_{i=1}^{\ell'}\mu'_{i,1}\times\mu'_{i,2},
        \]
        again with linearly independent component families.

        Let $k^\ast \triangleq \max\{\ell,\ell'\}$. Observe that the $2k$ pushforward equalities in \eqref{eqn:silhoulette-vector1}–\eqref{eqn:silhoulette-vector2} contain, as a subset, the $2k^\ast$ equalities obtained by restricting $m$ to $\{1,\ldots,k^\ast\}$ and $\{0,\ldots,k^\ast-1\}$. Thus the hypotheses of the theorem are satisfied with $k$ replaced by $k^\ast$. Applying the already-established theorem with parameter \(k^\ast\) yields \(\mu=\mu'\).
    \end{proof}

    The proof of Corollary~\ref{cor:main-multiview} relies on the fact that, for finite \emph{nonnegative} measures, failure of finiteness of the Laplace transform is preserved under convolution with any nonzero positive measure. This property is formalized in the following lemma.
    \begin{lemma}\label{lem:nonexp-convolution}
        Let $\mu\in\sM^+(\R)\setminus\sMexp^+(\R)$.  
        Then for every nonzero $\nu\in\sM^+(\R)$,
        \[
            \mu * \nu \in \sM^+(\R)\setminus\sMexp^+(\R).
        \]
    \end{lemma}
    \begin{proof}[Proof of Lemma \ref{lem:nonexp-convolution}]
        Since $\mu,\nu\in\sM^+(\R)$ are finite positive measures, their convolution
        $\mu*\nu$ is again a finite positive measure so $\mu*\nu\in\sM^+(\R)$.

        Fix $t\in \R$ and define the truncated functions
        \[
            f_n(z)\triangleq e^{t z}\wedge n,\qquad n\in\N.
        \]
        Then each $f_n$ is bounded, Borel measurable, and nonnegative. By
        \cite[Proposition~8.48(b)]{folland99},
        \begin{equation}\label{eqn:conv-trunc}
            \int f_n\,d(\mu*\nu)=\int_{\R}\int_{\R} f_n(x+y)\,d\mu(x)\,d\nu(y).
        \end{equation}
        Since $f_n\uparrow e^{tz}$ pointwise as $n\to\infty$, the monotone convergence theorem gives
        \begin{equation}\label{eqn:mct-left}
            \cL(\mu\ast \nu)(t) = \int e^{t z}\,d(\mu*\nu)(z)=\lim_{n\to\infty}\int f_n\,d(\mu*\nu).
        \end{equation}
        Moreover, $f_n(x+y)\uparrow e^{t(x+y)}$ pointwise on $\R^2$, so applying monotone convergence to the product measure $\mu\times\nu$ yields
        \begin{equation}\label{eqn:mct-right}
            \lim_{n\to\infty}\int_{\R}\int_{\R} f_n(x+y)\,d\mu(x)\,d\nu(y)
            =\int_{\R}\int_{\R} e^{t(x+y)}\,d\mu(x)\,d\nu(y).
        \end{equation}
        Combining \eqref{eqn:conv-trunc}, \eqref{eqn:mct-left}, and \eqref{eqn:mct-right} gives
        \[
            \cL(\mu\ast \nu)(t)
            =\int e^{t z}\,d(\mu*\nu)(z)
            =\int_{\R}\int_{\R} e^{t(x+y)}\,d\mu(x)\,d\nu(y)
            =\int_{\R}\int_{\R} e^{tx}e^{ty}\,d\mu(x)\,d\nu(y).
        \]
        Since \((x,y)\mapsto e^{tx}e^{ty}\) is nonnegative, Tonelli's Theorem yields
        \[
            \cL(\mu *\nu)(t)
            =
            \int_{\R}\int_{\R} e^{tx}e^{ty}\,d\mu(x)\,d\nu(y)
            =
            \left(\int_{\R} e^{tx}\,d\mu(x)\right)
            \left(\int_{\R} e^{ty}\,d\nu(y)\right)
            =
            \cL(\mu)(t)\,\cL(\nu)(t).
        \]

        Since \(\mu\notin\sMexp^+(\R)\), there exists a sequence
        \((t_i)_{i=1}^{\infty}\) with \(t_i\to0\) and
        \(\cL(\mu)(t_i)=\infty\) for all \(i\). Since \(\nu\) is a
        nonzero nonnegative measure, \(\cL(\nu)(t)\in(0,\infty]\)
        for every \(t\in\R\). Thus
        \[
            \cL(\mu*\nu)(t_i)
            =
            \cL(\mu)(t_i)\,\cL(\nu)(t_i)
            =
            \infty
        \]
        for all \(i\). Therefore \(\cL(\mu*\nu)\) is not finite
        on any neighborhood of \(0\), and
        \(\mu*\nu\notin\sMexp^+(\R)\).
    \end{proof}
    \begin{proof}[Proof of Corollary~\ref{cor:main-multiview}]
        By hypothesis, $\mu$ admits a representation \eqref{eqn:main-target} with $\mu_{i,j}\in\sMexp^+(\R)$ for all $i\in[k]$ and $j\in[2]$. To apply Theorem~\ref{thm:main}, it suffices to show that for every $i$ with $\mu'_{i,1}\times\mu'_{i,2}\neq 0$, one has $\mu'_{i,1},\mu'_{i,2}\in\sMexp^+(\R)$.

        Proceeding by contradiction, suppose there exists \(i_0\in[k]\) such that \(\mu'_{i_0,1}\times\mu'_{i_0,2}\neq0\) and \(\mu'_{i_0,1}\notin\sMexp^+(\R)\). Relabeling if necessary, assume \(i_0=1\). (The case \(\mu'_{i_0,2}\notin\sMexp^+(\R)\) is analogous.) Since \(\mu'_{1,1}\times\mu'_{1,2}\neq0\), both factors are nonzero.

        Let $c=[c_1,c_2]^\top$ be any direction appearing in \eqref{eqn:silhoulette-vector1}--\eqref{eqn:silhoulette-vector2}, so $c_1\neq 0$ and $c_2\neq 0$. By positivity of the measures and linearity of the Laplace transform,
        \begin{align*}
            \cL(\pi_c\#\mu')(t)
                &= \sum_{i=1}^k \cL\!\bigl(\pi_c\#(\mu'_{i,1}\times\mu'_{i,2})\bigr)(t) \\
                &\ge \cL\!\bigl(\pi_c\#(\mu'_{1,1}\times\mu'_{1,2})\bigr)(t).
        \end{align*}
        It is therefore enough to show that
        \[
            \pi_c\#(\mu'_{1,1}\times\mu'_{1,2})
            \notin\sMexp^+(\R).
        \]
        Let \(S_r(x)=rx\). The pushforward of a product measure under
        \(\pi_c(x,y)=c_1x+c_2y\) is
        \[
            \pi_c\#(\mu'_{1,1}\times\mu'_{1,2})
            =
            (S_{c_1}\#\mu'_{1,1})*(S_{c_2}\#\mu'_{1,2}).
        \]
        By Lemma~\ref{lem:substitution},
        \[
            \cL(S_{c_1}\#\mu'_{1,1})(t)
            =
            \cL(\mu'_{1,1})(c_1t),
        \]
        so \(S_{c_1}\#\mu'_{1,1}\notin\sMexp^+(\R)\), since \(c_1\neq0\). Also \(S_{c_2}\#\mu'_{1,2}\) is a nonzero finite positive measure. Lemma~\ref{lem:nonexp-convolution} therefore gives
        \[
            \pi_c\#(\mu'_{1,1}\times\mu'_{1,2})
            \notin\sMexp^+(\R).
        \]
        Hence \(\pi_c\#\mu'\notin\sMexp^+(\R)\). However, \(\pi_c\#\mu=\pi_c\#\mu'\) by hypothesis, while \(\pi_c\#\mu\in\sMexp^+(\R)\), since each component marginal of \(\mu\) lies in \(\sMexp^+(\R)\). This is a contradiction. Hence \(\mu'_{i,j}\in\sMexp^+(\R)\) whenever \(\mu'_{i,1}\times\mu'_{i,2}\neq0\). Any remaining zero terms may be replaced by \(0\times0\), so Theorem~\ref{thm:main} applies and yields \(\mu=\mu'\).
    \end{proof}
    \section{Proofs of Topological and Geometric Results}\label{sec:proofs-core}
    This section establishes the topological and geometric properties of the silhouette operator. The arguments are conceptually straightforward but rely on the weak\(^*\) closedness of classes of low-rank measures. In finite-dimensional matrix spaces, weak\(^*\) closedness is equivalent to ordinary closedness, and rank-bounded classes are well known to be closed. Proving weak\(^*\) closedness for low-rank measures on \(\mathbb{R}^2\), however, is nontrivial and requires tools from Banach space theory, in particular tensor products and the approximation property. All Banach spaces considered in this section are real. For a Banach space \(X\) and its dual \(X^*\), angle brackets denote the canonical dual pairing, with either order permitted:
    \[
        \langle x,x^*\rangle
        =
        \langle x^*,x\rangle
        =
        x^*(x),
        \qquad x\in X,\quad x^*\in X^*.
    \]
    \subsection{Topological Properties of Low-Rank Measure Classes} 
    The topological and geometric properties of the silhouette operator rely on compactness arguments. Establishing these properties requires showing that the class of low-rank measures is weak$^*$ closed. Throughout this subsection, for compact \(K_1,K_2\subset\R\), \(\sM(K_1\times K_2)\) is equipped with the weak\(^*\) topology arising from the canonical identification
    \[
        \sM(K_1\times K_2)\cong C(K_1\times K_2)^*
    \]
    given by the Riesz--Markov--Kakutani representation theorem, where \(C(K_1\times K_2)\) denotes the Banach space of real-valued continuous functions on \(K_1\times K_2\) equipped with the supremum norm. The main result of this subsection is the following.
    \begin{theorem}\label{thm:measure-tensor-w*closed}
        Let $K_1,K_2 \subset \R$ both be compact and let $k \in \N$. Then $\rankmeas_{\le k}(K_1,K_2)$ is weak$^*$ closed.
    \end{theorem}
    Because every weak$^*$ open set is weakly open and norm open, weak$^*$ closed sets are necessarily weakly closed and norm closed. Consequently, \(\rankmeas_{\le k}(K_1,K_2)\) is also weakly and norm closed.

    The proofs in this subsection make use of tensor products of Banach spaces. A comprehensive treatment may be found in \cite{ryan02}; only the material required for the present arguments is summarized here. Let \(X\) and \(Y\) be Banach spaces. The algebraic tensor product, denoted \(X\otimes Y\), consists of finite sums of elementary tensors,
    \begin{equation}\label{eqn:basic-tensor}
        \sum_{i=1}^k \lambda_i x_i\otimes y_i,
    \end{equation}
    where \(x_i\in X\), \(y_i\in Y\), and \(\lambda_i\) are scalars in the underlying field. The algebraic tensor product \(X\otimes Y\) is a vector space, with addition and scalar multiplication defined termwise on representations of the form \eqref{eqn:basic-tensor}. The tensor product is bilinear, meaning that for \(x,x_1,x_2\in X\), \(y,y_1,y_2\in Y\), and \(\lambda\in\R\),
    \[
        (x_1+x_2)\otimes y
        =
        x_1\otimes y+x_2\otimes y,
    \]
    \[
        x\otimes(y_1+y_2)
        =
        x\otimes y_1+x\otimes y_2,
    \]
    and
    \[
        \lambda(x\otimes y)
        =
        (\lambda x)\otimes y
        =
        x\otimes(\lambda y).
    \]
    In particular,
    \[
        0\otimes y=x\otimes0=0.
    \]

    Tensor products of Banach spaces admit many natural norms. Two of the most important are the \emph{projective} and \emph{injective} tensor product norms. The projective tensor product of \(X\) and \(Y\), denoted \(X\otimes_\pi Y\), is the algebraic tensor product equipped with the norm \(\|\cdot\|_\pi\) defined by
    \begin{equation*}
        \left\|\sum_{i=1}^k x_i\otimes y_i\right\|_\pi
        =
        \inf\left\{
            \sum_{i=1}^{k'} \|x'_i\|\|y'_i\|
            :
            k' \in \N,\,
            \sum_{i=1}^{k'} x'_i\otimes y'_i
            =
            \sum_{i=1}^k x_i\otimes y_i
        \right\}.
    \end{equation*}
    The projective tensor norm may be viewed as an infinite-dimensional analogue of the matrix nuclear norm. The completion of \(X\otimes_\pi Y\) is denoted \(X\projtimes Y\). A useful structural property of the completed projective tensor product is that every element \(u\in X\projtimes Y\) admits a representation
    \[
        u=\sum_{i=1}^{\infty} x_i\otimes y_i,
        \qquad
        \sum_{i=1}^{\infty}\|x_i\|\,\|y_i\|<\infty,
    \]
    see \cite[Proposition~2.8]{ryan02}.

    For a normed space \(X\), let \(\ball_X\) denote its closed unit ball; when subscripts become cumbersome, \(\ball(X)\) will be used instead. The injective tensor product, denoted \(X\otimes_\varepsilon Y\), is the algebraic tensor product equipped with the norm \(\|\cdot\|_\varepsilon\) defined by
    \begin{equation*}
        \left\|\sum_{i=1}^k x_i\otimes y_i\right\|_\varepsilon
        =
        \sup\left\{
            \left| \sum_{i=1}^{k} \phi(x_i)\psi(y_i)\right|
            :
            \phi \in \ball_{X^*},\,
            \psi \in \ball_{Y^*}
        \right\}.
    \end{equation*}
    For the injective tensor product, no optimization over representations of the tensor is required. The completion of \(X\otimes_\varepsilon Y\) is denoted \(X\injtimes Y\).

    For Banach spaces \(X\) and \(Y\), there is a canonical linear mapping
    \[
        J: X^*\projtimes Y^* \;\longrightarrow\; (X\injtimes Y)^*
    \]
    characterized on elementary tensors by
    \[
        \left\langle
            J(x^*\otimes y^*),x\otimes y
        \right\rangle
        \triangleq
        x^*(x)y^*(y),
        \qquad
        x\in X,\quad y\in Y,\quad
        x^*\in X^*,\quad y^*\in Y^*,
    \]
    and extended by linearity and continuity.

    One final functional-analytic concept is required before turning to the proof of Theorem~\ref{thm:measure-tensor-w*closed}. A Banach space \(X\) is said to have the \emph{approximation property} if, for every norm-compact subset \(K\subset X\), the identity operator on \(K\) can be uniformly approximated by finite-rank operators; that is, for every \(\epsilon>0\), there exists a finite-rank bounded linear operator \(F:X\to X\) such that
    \[
        \sup_{x\in K}\|F(x)-x\|\le \epsilon.
    \]
    The approximation property is a deep structural property of Banach spaces. It holds for many classical spaces, while the space of bounded linear operators on an infinite-dimensional Hilbert space provides a well-known counterexample \cite{szankowsk81}.

    A tensor \(u\in X\projtimes Y\) is said to have rank at most \(k\) if there exist \(x_1,\ldots,x_k\in X\) and \(y_1,\ldots,y_k\in Y\) such that
    \[
        u=\sum_{i=1}^k x_i\otimes y_i.
    \]
    The following theorem generalizes Proposition~1.2 of \cite{petitjean25}, which establishes weak closedness of the set of rank-one tensors.
    \begin{theorem}\label{thm:rank-closed}
        Let \(X\) and \(Y\) be Banach spaces, let \(k\in\N\), and suppose that at least one of \(X^*\) or \(Y^*\) satisfies the approximation property. Then, under the canonical mapping
        \[
            X^*\projtimes Y^*
            \longrightarrow
            (X\injtimes Y)^*,
        \]
        the set of tensors of rank at most \(k\) in \(X^*\projtimes Y^*\) is weak\(^*\) closed.
    \end{theorem}
    The proof of Theorem~\ref{thm:rank-closed} draws heavily on the methods developed in \cite{petitjean25}.

    \begin{proof}[Proof of Theorem \ref{thm:rank-closed}]
        Let
        \[
            Q_{\le k}
            =
            \left\{
                u\in X^*\projtimes Y^*
                :
                u \text{ has rank at most } k
            \right\},
        \]
        and let
        \begin{equation}
            \label{eqn:canonical-map-J}
            J:X^*\projtimes Y^* \longrightarrow (X\injtimes Y)^*
        \end{equation}
        denote the canonical mapping. By \cite[Corollary~4.8(b)]{ryan02}, the canonical mapping
        \[
            \Theta:X^*\projtimes Y^*
            \longrightarrow
            \mathcal{B}_N(X\times Y),
            \qquad
            \Theta(u)(x,y)
            =
            \langle J(u),x\otimes y\rangle,
        \]
        is an isometric isomorphism under the approximation property assumption, where \(\mathcal{B}_N(X\times Y)\) denotes the space of nuclear bilinear forms on \(X\times Y\), equipped with the nuclear norm \cite[Section~2.6, pp.~39--40]{ryan02}. To show that \(J\) is injective, let \(u\in X^*\projtimes Y^*\) satisfy \(J(u)=0\). Then, for every \(x\in X\) and \(y\in Y\),
        \[
            \Theta(u)(x,y)
            =
            \langle J(u),x\otimes y\rangle
            =
            0.
        \]
        Thus \(\Theta(u)\) is the zero bilinear form. Since \(\Theta\) is an isometric isomorphism, it is injective, and therefore \(u=0\). Hence \(\ker J=\{0\}\), so the linear map \(J\) is injective.

        For \(H\in(X\injtimes Y)^*\), define the bounded linear operator
        \[
            T_H:X\longrightarrow Y^*,
            \qquad
            T_H(x)(y)
            \triangleq
            \langle H,x\otimes y\rangle.
        \]
        The operator \(T_H\) is clearly linear. To see that it is bounded, observe that
        \[
            |T_H(x)(y)|
            \le
            \|H\|\,\|x\otimes y\|_\varepsilon
            =
            \|H\|\,\|x\|\,\|y\|.
        \]
        It will first be shown that
        \begin{equation}\label{eqn:operator-tensor-rank}
            H\in J(Q_{\le k})
            \quad\Longleftrightarrow\quad
            \dim T_H(X)\le k.
        \end{equation}
        If
        \[
            H
            =
            J\left(\sum_{i=1}^r x_i^*\otimes y_i^*\right),
            \qquad r\le k,
        \]
        then
        \[
            T_H(x)
            =
            \sum_{i=1}^r x_i^*(x)y_i^*,
        \]
        and hence \(T_H(X)\subseteq\spn\{y_1^*,\ldots,y_r^*\}\).
        Thus \(\dim T_H(X)\le k\).

        Conversely, suppose that \(\dim T_H(X)=r\le k\). If \(r=0\), then \(T_H=0\), and hence
        \[
            \langle H,x\otimes y\rangle
            =
            T_H(x)(y)
            =
            0
        \]
        for every \(x\in X\) and \(y\in Y\). By linearity, \(H\) therefore vanishes on \(X\otimes Y\), the span of the elementary tensors. Since \(X\otimes Y\) is dense in \(X\injtimes Y\) and \(H\) is continuous, it follows that \(H=0\). Thus
        \[
            H=J(0)\in J(Q_{\le k}),
        \]
        which proves the implication when \(r=0\).

        Assume now that \(r>0\). Let
        \[
            F\triangleq T_H(X),
        \]
        and choose a basis \(y_1^*,\ldots,y_r^*\) for \(F\). Let
        \(\lambda_1,\ldots,\lambda_r\in F^*\) be the corresponding coordinate
        functionals, characterized by
        \[
            \lambda_i(y_i^*)=1,
            \qquad
            \lambda_i(y_j^*)=0
            \quad\text{when }i\ne j.
        \]
        Equivalently, the identity operator on \(F\) satisfies
        \[
            I_F(z)
            =
            \sum_{i=1}^r\lambda_i(z)y_i^*,
            \qquad z\in F.
        \]
        Since \(F\) is finite dimensional, each \(\lambda_i\) is bounded. Define
        \[
            x_i^*
            \triangleq
            \lambda_i\circ T_H
            \in X^*,
            \qquad i\in[r].
        \]
        For every \(x\in X\), the preceding representation of \(I_F\) gives
        \[
            T_H(x)
            =
            \sum_{i=1}^r
            \lambda_i\bigl(T_H(x)\bigr)y_i^*
            =
            \sum_{i=1}^r x_i^*(x)y_i^*.
        \]
        Consequently, for every \(x\in X\) and \(y\in Y\),
        \begin{align*}
            \left\langle
                J\left(\sum_{i=1}^r x_i^*\otimes y_i^*\right),
                x\otimes y
            \right\rangle
    &=
    \sum_{i=1}^r x_i^*(x)y_i^*(y)\\
    &=
    T_H(x)(y)\\
    &=
    \langle H,x\otimes y\rangle.
        \end{align*}
        Since \(X\otimes Y\) is dense in \(X\injtimes Y\), it follows that
        \[
            H
            =
            J\left(\sum_{i=1}^r x_i^*\otimes y_i^*\right)
            \in J(Q_{\le k}),
        \]
        proving \eqref{eqn:operator-tensor-rank}.

        For
        \[
            S
            =
            \bigl((x_1,\ldots,x_{k+1}),
            (y_1,\ldots,y_{k+1})\bigr)
            \in X^{k+1}\times Y^{k+1},
        \]
        define
        \[
            D_S:(X\injtimes Y)^*\longrightarrow\R
        \]
        by
        \[
            D_S(H)
            \triangleq
            \det\!\left[
                \langle H,x_i\otimes y_j\rangle
            \right]_{i\in[k+1],\,j\in[k+1]}.
        \]
        Each \(D_S\) is weak\(^*\)-continuous, since each mapping
        \[
            H\longmapsto\langle H,x_i\otimes y_j\rangle
        \]
        is weak\(^*\)-continuous and the determinant is a polynomial in these finitely many values. It will now be shown that \(H\in J(Q_{\le k})\) if and only if
        \(D_S(H)=0\) for every \(S\in X^{k+1}\times Y^{k+1}\).

        Suppose that \(\dim T_H(X)\le k\). By \eqref{eqn:operator-tensor-rank}, there exist \(r\le k\), \(x_1^*,\ldots,x_r^*\in X^*\), and \(y_1^*,\ldots,y_r^*\in Y^*\) such that
        \[
            H
            =
            J\left(
                \sum_{\ell=1}^r x_\ell^*\otimes y_\ell^*
            \right).
        \]
        Consequently, for every \(S\in X^{k+1}\times Y^{k+1}\),
        \begin{align*}
            \left[
                \langle H,x_i\otimes y_j\rangle
            \right]_{i,j\in[k+1]}
                &=
                \left[
                    \sum_{\ell=1}^r
                    x_\ell^*(x_i)y_\ell^*(y_j)
                \right]_{i,j\in[k+1]}\\
                &=
                \sum_{\ell=1}^r
                \begin{bmatrix}
                    x_\ell^*(x_1)\\
                    x_\ell^*(x_2)\\
                    \vdots\\
                    x_\ell^*(x_{k+1})
                \end{bmatrix}
                \begin{bmatrix}
                    y_\ell^*(y_1) &
                    y_\ell^*(y_2) &
                    \cdots &
                    y_\ell^*(y_{k+1})
                \end{bmatrix}.
        \end{align*}
        Each summand has rank at most one, so the matrix has rank at most
        \(r\le k\). Its determinant is therefore zero, and hence \(D_S(H)=0\)
        for every \(S\).

        Conversely, suppose that \(\dim T_H(X)>k\). There then exist
        \(x_1,\ldots,x_{k+1}\in X\) such that
        \[
            T_H(x_1),\ldots,T_H(x_{k+1})
        \]
        are linearly independent elements of \(Y^*\). The linear map
        \[
            R:Y\longrightarrow\R^{k+1},
            \qquad
            R(y)
            =
            \bigl(T_H(x_1)(y),\ldots,T_H(x_{k+1})(y)\bigr),
        \]
        is surjective. Otherwise, a nonzero linear functional annihilating
        its range would give a nontrivial linear relation among
        \(T_H(x_1),\ldots,T_H(x_{k+1})\).
        Let \(e_1,\ldots,e_{k+1}\) denote the standard basis of
        \(\R^{k+1}\). Since \(R\) is surjective, for every \(j\in[k+1]\) there
        exists \(y_j\in Y\) such that
        \[
            R(y_j)=e_j.
        \]
        It follows that
        \[
            \left[
                \langle H,x_i\otimes y_j\rangle
            \right]_{i,j\in[k+1]}
            =
            \left[
                T_H(x_i)(y_j)
            \right]_{i,j\in[k+1]}
            =
            I_{k+1}.
        \]
        The determinant of this matrix is therefore nonzero.

        Combining this characterization with
        \eqref{eqn:operator-tensor-rank} gives
        \[
            J(Q_{\le k})
            =
            \bigcap_{S\in X^{k+1}\times Y^{k+1}}
            D_S^{-1}(\{0\}).
        \]
        The right-hand side is an intersection of weak\(^*\) closed subsets
        of \((X\injtimes Y)^*\). Hence \(J(Q_{\le k})\) is weak\(^*\) closed.
    \end{proof}

    In the dual-space setting of Theorem~\ref{thm:rank-closed}, weak\(^*\) closedness also implies weak and norm closedness. The preceding argument also yields the following weak closedness result.
    \begin{corollary}\label{cor:rank-closed}
        Let \(X\) and \(Y\) be Banach spaces, at least one of which has the approximation property, and let \(k\in\N\). Then the set of tensors of rank at most \(k\) in \(X\projtimes Y\) is weakly closed.
    \end{corollary}
    It is now possible to prove Theorem \ref{thm:measure-tensor-w*closed}.
    \begin{proof}[Proof of Theorem~\ref{thm:measure-tensor-w*closed}]
        Set
        \[
            X\triangleq C(K_1),
            \qquad
            Y\triangleq C(K_2).
        \]
        Since \(K_1\) and \(K_2\) are compact metric spaces, the
        Riesz--Markov--Kakutani representation theorem gives the canonical
        linear isometric isomorphisms
        \[
            X^*\cong\sM(K_1),
            \qquad
            Y^*\cong\sM(K_2),
            \qquad
            \sM(K_1\times K_2)
            \cong
            C(K_1\times K_2)^*.
        \]

        Let
        \[
            \Gamma:
            X\injtimes Y
            \longrightarrow
            C(K_1\times K_2)
        \]
        denote the canonical linear isometric isomorphism determined by
        \[
            \Gamma(f\otimes g)(x,y)
            =
            f(x)g(y)
        \]
        \cite[p.~50]{ryan02}. Under the preceding identification of
        \(\sM(K_1\times K_2)\) with \(C(K_1\times K_2)^*\), define
        \[
            \Phi
            \triangleq
            \Gamma^*:
            \sM(K_1\times K_2)
            \longrightarrow
            (X\injtimes Y)^*.
        \]
        Since \(\Gamma\) is a linear isometric isomorphism, so is its dual
        operator \(\Phi\) \cite[Exercise~2.41(ii)]{fabian01}. By the
        definition of the dual operator,
        \begin{equation}\label{eqn:measure-tensor-dual-pairing}
            \langle\Phi(\mu),h\rangle
            =
            \langle\mu,\Gamma(h)\rangle
        \end{equation}
        for every \(\mu\in\sM(K_1\times K_2)\) and
        \(h\in X\injtimes Y\).

        Let
        \[
            J:
            X^*\projtimes Y^*
            \longrightarrow
            (X\injtimes Y)^*
        \]
        denote the canonical linear mapping. For
        \(\nu_1\in\sM(K_1)\), \(\nu_2\in\sM(K_2)\),
        \(f\in C(K_1)\), and \(g\in C(K_2)\),
        \begin{align*}
            \left\langle
                \Phi(\nu_1\times\nu_2),
                f\otimes g
            \right\rangle
                    &=
                    \left\langle
                        \nu_1\times\nu_2,
                        \Gamma(f\otimes g)
                    \right\rangle\\
                    &=
                    \int_{K_1\times K_2}
                    f(x)g(y)\,d(\nu_1\times\nu_2)(x,y)\\
                    &=
                    \left(\int_{K_1}f\,d\nu_1\right)
                    \left(\int_{K_2}g\,d\nu_2\right)\\
                    &=
                    \left\langle
                        J(\nu_1\otimes\nu_2),
                        f\otimes g
                    \right\rangle.
        \end{align*}
        Thus, for each fixed pair \(\nu_1,\nu_2\), the functionals \(\Phi(\nu_1\times\nu_2)\) and \(J(\nu_1\otimes\nu_2)\) agree on every elementary tensor \(f\otimes g\). By linearity, they therefore agree on \(X\otimes Y\). Since \(X\otimes Y\) is dense in \(X\injtimes Y\), their continuity implies that
        \begin{equation}\label{eqn:product-measure-tensor}
            \Phi(\nu_1\times\nu_2)
            =
            J(\nu_1\otimes\nu_2)
            \qquad
            \text{for all }\nu_1,\nu_2.
        \end{equation}

        Let \(Q_{\le k}\) denote the set of tensors of rank at most \(k\)
        in \(X^*\projtimes Y^*\). Equation
        \eqref{eqn:product-measure-tensor} and the linearity of \(\Phi\)
        and \(J\) show that
        \begin{equation}\label{eqn:measure-rank-tensor-rank}
            \Phi\bigl(\rankmeas_{\le k}(K_1,K_2)\bigr)
            =
            J(Q_{\le k}).
        \end{equation}
        Indeed, a representation
        \[
            \mu
            =
            \sum_{i=1}^r
            \nu_{i,1}\times\nu_{i,2},
            \qquad r\le k,
        \]
        is carried by \(\Phi\) to
        \[
            J\left(
                \sum_{i=1}^r
                \nu_{i,1}\otimes\nu_{i,2}
            \right),
        \]
        and every element of \(J(Q_{\le k})\) arises in this way.

        Since \(X^*=C(K_1)^*\) has the approximation property
        \cite[p.~80]{ryan02}, Theorem~\ref{thm:rank-closed} implies that
        \(J(Q_{\le k})\) is weak\(^*\) closed in
        \((X\injtimes Y)^*\).

        Let \((\mu_\alpha)_{\alpha\in A}\) be a net in
        \(\rankmeas_{\le k}(K_1,K_2)\) converging weak\(^*\) in
        \(\sM(K_1\times K_2)\) to some \(\mu\). 
        By \eqref{eqn:measure-rank-tensor-rank},
        \[
            \Phi(\mu_\alpha)\in J(Q_{\le k})
            \qquad
            \text{for every }\alpha.
        \]
        Moreover, for every \(h\in X\injtimes Y\),
        \[
            \langle\Phi(\mu_\alpha),h\rangle
            =
            \langle\mu_\alpha,\Gamma(h)\rangle
            \longrightarrow
            \langle\mu,\Gamma(h)\rangle
            =
            \langle\Phi(\mu),h\rangle.
        \]
        Hence \( \Phi(\mu_\alpha) \longrightarrow \Phi(\mu)\) weak\(^*\) in
        \((X\injtimes Y)^*\). Since \(J(Q_{\le k})\) is weak\(^*\)
        closed,
        \[
            \Phi(\mu)\in J(Q_{\le k}).
        \]
        Equation~\eqref{eqn:measure-rank-tensor-rank} and the injectivity
        of \(\Phi\) therefore imply that
        \[
            \mu\in\rankmeas_{\le k}(K_1,K_2).
        \]
        Thus \(\rankmeas_{\le k}(K_1,K_2)\) is weak\(^*\) closed.
    \end{proof}
    \subsection{Proofs of Theorem~\ref{thm:weakstar-homeo} and Proposition~\ref{prop:rip-span}}

    The proof of Theorem~\ref{thm:weakstar-homeo} uses the following two standard compactness results.
    \begin{lemma}[{\cite[Theorem~5.8]{kelley55}}]\label{thm:compact-homeo}
        Let \((\sX,\tau_{\sX})\) be a compact topological space and \((\sY,\tau_{\sY})\) a Hausdorff topological space.  
        If \(f : \sX \to \sY\) is continuous and injective, then \(f^{-1} : f(\sX) \to \sX\) is continuous, and hence \(f : \sX \to f(\sX)\) is a homeomorphism.
    \end{lemma}
    \begin{theorem}[Banach--Alaoglu {\cite[Theorem~3.21]{fabian01}}]\label{thm:banach-alaoglu}
        If \(X\) is a Banach space, then the closed unit ball of \(X^*\) is compact in the weak\(^*\) topology.
    \end{theorem}
    The proof of Theorem~\ref{thm:weakstar-homeo} also relies on the continuity of the silhouette operator.
    \begin{lemma}\label{lem:silhouette-continuous}
        Let \(K_1,K_2 \subset \R\) be compact, let \(k\in\N\), and let \(\sil\in\silset_{2k}\). Consider the restriction
        \[
            \sil : \sM(K_1\times K_2) \to \sM(\R)^{2k}.
        \]
        Then \(\sil\) is norm-to-norm, weak-to-weak, and weak\(^*\)-to-weak\(^*\) continuous.
    \end{lemma}

    \begin{proof}[Proof of Lemma \ref{lem:silhouette-continuous}]
        It suffices to prove the statement for a single coordinate of $\sil$, since $\sil\in\silset_{2k}$ is a finite tuple of linear pushforwards and $\sM(\R)^{2k}$ is equipped with the corresponding product topology. Thus let
        \[
            \pi\# : \sM(K_1\times K_2) \to \sM(\R)
        \]
        be one of the pushforward maps appearing in $\sil$, where \(\pi:K_1\times K_2\to \R\) is the restriction of the corresponding linear map on \(\R^2\).

        The pushforward map \(\pi\#\) is linear, and for every
        \(\mu\in\sM(K_1\times K_2)\),
        \[
            \|\pi\#\mu\|_{\mathrm{TV}}
            \le
            \|\mu\|_{\mathrm{TV}}.
        \]
        Therefore \(\pi\#\) is a bounded linear operator between Banach spaces, and is hence norm-to-norm and weak-to-weak continuous.

        Let \((\mu_\alpha)_{\alpha\in A}\) be a net in \(\sM(K_1\times K_2)\) that converges weak\(^*\) to \(\mu\), and let \(f \in C_0(\R)\), where \(C_0(\R)\) denotes the space of continuous functions on \(\R\) that vanish at infinity. To prove weak\(^*\)-to-weak\(^*\) continuity, it is enough to show that
        \[
            \langle \pi\#\mu_\alpha, f\rangle
            \longrightarrow
            \langle \pi\#\mu, f\rangle.
        \]

        For each \(\alpha\), Lemma~\ref{lem:substitution-signed} gives
        \[
            \langle \pi\#\mu_\alpha,f\rangle
            =
            \langle\mu_\alpha,f\circ\pi\rangle.
        \]
        Since \(f\circ\pi\in C(K_1\times K_2)\) and
        \(\mu_\alpha\to\mu\) weak\(^*\),
        \[
            \langle \pi\#\mu_\alpha,f\rangle
            =
            \langle\mu_\alpha,f\circ\pi\rangle
            \longrightarrow
            \langle\mu,f\circ\pi\rangle
            =
            \langle\pi\#\mu,f\rangle,
        \]
        where the final equality again follows from
        Lemma~\ref{lem:substitution-signed}. Hence
        \(\pi\#\mu_\alpha\) converges weak\(^*\) to \(\pi\#\mu\), so
        \(\pi\#\) is weak\(^*\)-to-weak\(^*\) continuous.
    \end{proof}

    Theorem~\ref{thm:weakstar-homeo} now follows by combining the compactness results above with the continuity of the silhouette operator and its injectivity on \(\rankmeas_{\le k}(K_1,K_2)\).
    \begin{proof}[Proof of Theorem~\ref{thm:weakstar-homeo}]
        Let
        \[
            A
            =
            \rankmeas_{\le k}(K_1,K_2)
            \,\cap\,
            r\ball_{\sM(K_1\times K_2)}.
        \]
        By Theorem~\ref{thm:measure-tensor-w*closed}, \(\rankmeas_{\le k}(K_1,K_2)\) is weak\(^*\) closed in \(\sM(K_1\times K_2)\). Moreover, by Theorem~\ref{thm:banach-alaoglu}, \(r\ball_{\sM(K_1\times K_2)}\) is weak\(^*\) compact. Hence \(A\) is weak\(^*\) compact.
        By Theorem~\ref{thm:main}, the silhouette operator \(\sil\) is injective on \(\rankmeas_{\le k}(K_1,K_2)\), and hence on \(A\). By Lemma~\ref{lem:silhouette-continuous}, its restriction to \(A\) is weak\(^*\)-to-weak\(^*\) continuous. Since \(\sM(\R)^{2k}\) is Hausdorff in the weak\(^*\) topology, its subspace \(\sil(A)\) is also Hausdorff. Lemma~\ref{thm:compact-homeo} therefore applies to the continuous injection
        \[
            \sil|_A:A\to \sil(A),
        \]
        and shows that this restriction is a weak$^*$-weak$^*$ homeomorphism.
    \end{proof}

    \begin{proof}[Proof of Proposition~\ref{prop:rip-span}]
        If
        \[
            S\cap\rankmeas_{\le k}(K_1,K_2)=\{0\},
        \]
        then the conclusion is immediate, for example by taking
        \(\gamma=1\) and \(\delta=0\). Assume henceforth that this
        intersection contains a nonzero measure.
        It will first be shown that there exists \(\kappa>0\) such that
        \[
            \|\sil(\mu)\| \;\ge\; \kappa\,\|\mu\|
            \qquad
            \text{for all }\mu\in S\cap \rankmeas_{\le 2k}(K_1,K_2).
        \]
        Assume for the sake of contradiction that no such \(\kappa\) exists. Then there exists a sequence \((\mu_i)_{i=1}^\infty\) of nonzero elements of \(S\cap \rankmeas_{\le 2k}(K_1,K_2)\) such that
        \[
            \frac{\|\sil(\mu_i)\|}{\|\mu_i\|}
            \;\longrightarrow\;0.
        \]
        By normalization, assume that \(\|\mu_i\|=1\) for all \(i\). Let
        \(\sphere\) denote the unit sphere in \(\sM(K_1\times K_2)\). Then
        \((\mu_i)_{i=1}^\infty\) lies in
        \[
            S\cap \rankmeas_{\le 2k}(K_1,K_2)\cap\sphere.
        \]
        Since \(S\) is finite dimensional, \(S\cap\sphere\) is norm compact. Since \(\rankmeas_{\le 2k}(K_1,K_2)\) is norm closed by Theorem~\ref{thm:measure-tensor-w*closed}, the set \(S\cap \rankmeas_{\le 2k}(K_1,K_2)\cap\sphere\) is norm compact. Hence, after passing to a subsequence, \(\mu_i\to\mu\) in norm for some nonzero \(\mu\in \rankmeas_{\le 2k}(K_1,K_2)\). By continuity of \(\sil\), \(\sil(\mu)=0\).

        Because \(\mu\in \rankmeas_{\le 2k}(K_1,K_2)\), \(\mu=\nu-\nu'\) for some distinct \(\nu,\nu'\in \rankmeas_{\le k}(K_1,K_2)\). Then \(\sil(\nu)=\sil(\nu')\), contradicting the injectivity of \(\sil\) on \(\rankmeas_{\le k}(K_1,K_2)\) given by Theorem~\ref{thm:main}. Therefore there exists \(\kappa>0\) such that
        \[
            \|\sil(\mu)\| \;\ge\; \kappa\,\|\mu\|
            \qquad
            \text{for all }\mu\in S\cap \rankmeas_{\le 2k}(K_1,K_2).
        \]

        For \(\mu,\mu'\in S\cap\rankmeas_{\le k}(K_1,K_2)\), the difference \(\mu-\mu'\) lies in \(S\cap\rankmeas_{\le 2k}(K_1,K_2)\).
        Let \(M=\|\sil|_S\|\) denote the operator norm of the restriction of \(\sil\) to \(S\).
        Since \(S\) is finite dimensional, \(M<\infty\), and hence
        \[
            \kappa\|\mu-\mu'\|
            \le
            \|\sil(\mu-\mu')\|
            \le
            M\|\mu-\mu'\|,
        \]
        and $0<\kappa\le M$.
        Define
        \[
            \gamma=\frac{2}{M+ \kappa},
            \qquad
            \delta=\frac{M-\kappa}{M+\kappa}.
        \]
        Hence
        \(\delta\in[0,1)\), and            

        \begin{align*}
            (1-\delta)\|\mu-\mu'\|
                    &=
                    \frac{2\kappa}{M+\kappa}\|\mu-\mu'\| \\
                    &\le
                    \frac{2}{M+\kappa}\|\sil(\mu-\mu')\| \\
                    &=
                    \|\gamma\sil(\mu-\mu')\| \\
                    &\le
                    \frac{2M}{M+\kappa}\|\mu-\mu'\| \\
                    &=
                    (1+\delta)\|\mu-\mu'\|.
        \end{align*}
    \end{proof}
    \section{Proofs for Higher-Dimensional Extensions of the Silhouette Framework} \label{sec:proofs-d>2}
    \begin{proof}[Proof of Theorem \ref{thm:li-ident}]
        The proof is by induction on \(d\), the number of coordinates in the product measures. To include \(d=1\) as the base case, the induction will prove the conclusion under the additional explicit hypothesis that all one-dimensional marginals agree:
        \begin{equation}\label{eqn:li-ident}
            R_{\{\ell\}}\!\left(
                \sum_{i=1}^k
                w_i\,\delta_{(\mu_{i,1},\ldots,\mu_{i,d})}
            \right)
            =
            R_{\{\ell\}}\!\left(
                \sum_{i=1}^k
                w'_i\,\delta_{(\mu'_{i,1},\ldots,\mu'_{i,d})}
            \right)
            \qquad
            \text{for every }\ell\in[d].
        \end{equation}
        For \(d\ge2\), this additional hypothesis follows from the assumptions of the theorem. Indeed, because \(G\) is connected, every \(\ell\in[d]\) belongs to some edge \(e\in E\). Taking the \(\ell\)-th marginal of the assumed equality corresponding to \(e\) gives \eqref{eqn:li-ident}. It therefore suffices to prove this auxiliary statement by induction, beginning with \(d=1\).

        For \(d=1\), the hypothesis
        \[
            \{\mu_{1,1},\ldots,\mu_{k,1}\}=\{\mu'_{1,1},\ldots,\mu'_{k,1}\}
        \]
        together with \eqref{eqn:li-ident} gives
        \begin{align*}
            R_{\{1\}}\!\left(\sum_{i=1}^k w_i\,\delta_{(\mu_{i,1})}\right)
    &=R_{\{1\}}\!\left(\sum_{i=1}^k w'_i\,\delta_{(\mu'_{i,1})}\right)\\
            \Rightarrow\qquad \sum_{i=1}^k w_i\,\mu_{i,1}&=\sum_{i=1}^k w'_i\,\mu'_{i,1}.
        \end{align*}
        Relabeling the primed components if necessary, assume that \(\mu_{i,1}=\mu'_{i,1}\) for all \(i\). Since \(\mu_{1,1},\ldots,\mu_{k,1}\) are linearly independent, the preceding equality implies \(w_i=w'_i\) for all \(i\). This proves the base case.

        Assume that the result has been proved for all numbers of coordinates up to \(d\). Consider the case of \(d+1\) coordinates.
        Without loss of generality, assume that \(G\) is a spanning tree. It is enough to prove the result under this weaker assumption, since equality of the marginals on all edges of \(G\) implies equality on the edges of any spanning tree of \(G\).
        For notational ease, relabel the coordinates so that \(d+1\) is a leaf distinct from \(d'\), and \(d\) is its unique neighbor. This relabeling is possible because a spanning tree has at least two leaves, so even if the original coordinate \(d'\) is a leaf, another leaf can be chosen as \(d+1\). Since removing a leaf from a tree leaves a connected graph on the remaining vertices, the inductive hypothesis applies to the coordinates \([d]\).
        Hence
        \[
            \sum_{i=1}^k w_i\,\delta_{(\mu_{i,1},\ldots,\mu_{i,d})}
            =
            \sum_{i=1}^k w'_i\,\delta_{(\mu'_{i,1},\ldots,\mu'_{i,d})}.
        \]
        The \(d\)-tuples on the left are pairwise distinct because \(\mu_{1,d},\ldots,\mu_{k,d}\) are linearly independent, and each has positive weight. Since the right-hand side contains \(k\) terms with nonnegative weights, the preceding equality implies that, after relabeling the primed components,
        \[
            (\mu_{i,1},\ldots,\mu_{i,d})
            =
            (\mu'_{i,1},\ldots,\mu'_{i,d})
            \qquad\text{and}\qquad
            w_i=w'_i
        \]
        for every \(i\in[k]\).
        Since the pairwise marginals agree on the edge \(\{d,d+1\}\),
        \[
            \sum_{i=1}^k w_i\,\mu_{i,d}\times\mu_{i,d+1}
            =
            \sum_{i=1}^k w'_i\,\mu'_{i,d}\times\mu'_{i,d+1}.
        \]
        Using the relabeling above, this becomes
        \[
            \sum_{i=1}^k w_i\,\mu_{i,d}\times\mu_{i,d+1}
            =
            \sum_{i=1}^k w_i\,\mu_{i,d}\times\mu'_{i,d+1}.
        \]
        By Lemma~\ref{lem:measure-bidual}, choose bounded measurable functions
        \(f_1,\ldots,f_k\) such that
        \[
            \mu_{i,d}(f_j)=\ind(i=j),
            \qquad i,j\in[k].
        \]
        For each \(j\in[k]\), applying the operator \(T_{f_j}\), defined in
        \eqref{eqn:Tf-def}, to both sides gives
        \[
            w_j\mu_{j,d+1}=w_j\mu'_{j,d+1}.
        \]
        Since \(w_j>0\), it follows that \(\mu_{j,d+1}=\mu'_{j,d+1}\) for every
        \(j\in[k]\). Hence, after the same relabeling,
        \[
            (\mu_{i,1},\ldots,\mu_{i,d+1})
            =
            (\mu'_{i,1},\ldots,\mu'_{i,d+1})
            \qquad\text{and}\qquad
            w_i=w'_i
        \]
        for all \(i\), completing the induction.
    \end{proof}

    \begin{proof}[Proof of Proposition~\ref{prop:ji-ident}]
        The hypothesis gives
        \begin{align}
            R\!\left(\sum_{i=1}^k w_i\,\delta_{(\mu_{i,1},\mu_{i,2})}\right)
    &=
    R\!\left(\sum_{i=1}^k w'_i\,\delta_{(\mu'_{i,1},\mu'_{i,2})}\right),
        \end{align}
        that is,
        \begin{equation}\label{eqn:ji-equality}
            \sum_{i=1}^k w_i\,\mu_{i,1}\times\mu_{i,2}
            =
            \sum_{i=1}^k w'_i\,\mu'_{i,1}\times\mu'_{i,2}.
        \end{equation}

        Since joint irreducibility implies linear independence, the families \(\mu_{1,1},\ldots,\mu_{k,1}\) and \(\mu_{1,2},\ldots,\mu_{k,2}\) are linearly independent. Thus the left-hand side of \eqref{eqn:ji-equality} has rank \(k\) (Proposition~\ref{prop:rank-characterization}). Hence the right-hand side also has rank \(k\), which implies that \(w'_i>0\) for all \(i\), and that the families \(\mu'_{1,1},\ldots,\mu'_{k,1}\) and \(\mu'_{1,2},\ldots,\mu'_{k,2}\) are linearly independent.

        By Lemma~\ref{lem:measure-bidual}, there exist bounded measurable functions \(f_{j,\ell},f'_{j,\ell}\in L^\infty(\Omega)\), \(j\in[k]\), \(\ell\in[2]\), such that
        \[
            \int f_{j,\ell}\,d\mu_{i,\ell} = \ind(i=j),
            \qquad
            \int f'_{j,\ell}\,d\mu'_{i,\ell} = \ind(i=j),
        \]
        for all \(i,j\in[k]\) and \(\ell\in[2]\).

        Applying the operator \(T_{f'_{j,1}}\), defined in
        \eqref{eqn:Tf-def}, to both sides of
        \eqref{eqn:ji-equality} yields
        \[
            \sum_{i=1}^k
            w_i\,\left\langle f'_{j,1},\mu_{i,1}\right\rangle\,\mu_{i,2}
            =
            \sum_{i=1}^k
            w'_i\,\left\langle f'_{j,1},\mu'_{i,1}\right\rangle\,\mu'_{i,2}
            =
            w'_j\,\mu'_{j,2}.
        \]
        Hence, for each \(j\in[k]\),
        \[
            \mu'_{j,2}
            =
            \sum_{i=1}^k \beta_{j,i}\,\mu_{i,2},
            \qquad
            \beta_{j,i}
            \triangleq
            \frac{w_i}{w'_j}\,
            \left\langle f'_{j,1},\mu_{i,1}\right\rangle.
        \]
        The same argument applied to the second coordinate gives
        \[
            \mu'_{j,1}
            =
            \sum_{i=1}^k \alpha_{j,i}\,\mu_{i,1},
            \qquad
            \alpha_{j,i}
            \triangleq
            \frac{w_i}{w'_j}\,
            \left\langle f'_{j,2},\mu_{i,2}\right\rangle.
        \]
        Since \(\mu'_{j,1}\) and \(\mu'_{j,2}\) are probability measures and \(\mu_{1,1},\ldots,\mu_{k,1}\) and \(\mu_{1,2},\ldots,\mu_{k,2}\) are jointly irreducible, all coefficients \(\alpha_{j,i}\) and \(\beta_{j,i}\) are nonnegative.

        Suppose, toward a contradiction, that for some fixed $j$ there exist distinct indices
        $i_1 \neq i_2$ such that
        \[
            \alpha_{j,i_1} > 0
            \qquad\text{and}\qquad
            \beta_{j,i_2} > 0.
        \]
        Consider the bounded linear functional on $\sM(\Omega\times\Omega)$ defined by
        \[
            \nu \;\mapsto\; \bigl\langle T_{f_{i_1,1}}(\nu),\, f_{i_2,2}\bigr\rangle
            = \int f_{i_1,1}(x)\,f_{i_2,2}(y)\,d\nu(x,y).
        \]
        Applying this functional to both sides of \eqref{eqn:ji-equality} yields
        \begin{equation}\label{eqn:ji-contradiction}
            \sum_{i=1}^k w_i
            \left\langle f_{i_1,1},\mu_{i,1}\right\rangle
            \left\langle f_{i_2,2},\mu_{i,2}\right\rangle
            =
            \sum_{i=1}^k w'_i
            \left\langle f_{i_1,1},\mu'_{i,1}\right\rangle
            \left\langle f_{i_2,2},\mu'_{i,2}\right\rangle .
        \end{equation}

        By construction of the biorthogonal system, the left-hand side of \eqref{eqn:ji-contradiction} is zero, since \(i_1\neq i_2\). On the other hand, consider the summand with index \(i=j\) on the right-hand
        side of \eqref{eqn:ji-contradiction}. Using the representations
        \(
        \mu'_{j,1}=\sum_i \alpha_{j,i}\mu_{i,1}
        \)
        and
        \(
        \mu'_{j,2}=\sum_i \beta_{j,i}\mu_{i,2}
        \), this summand is
        \[
            w'_j\,
            \left\langle f_{i_1,1},\mu'_{j,1}\right\rangle
            \left\langle f_{i_2,2},\mu'_{j,2}\right\rangle
            =
            w'_j\,\alpha_{j,i_1}\,\beta_{j,i_2}
            > 0.
        \]
        More generally, for every \(i\in[k]\), the \(i\)-th summand on the
        right-hand side of \eqref{eqn:ji-contradiction} is
        \[
            w'_i\,\alpha_{i,i_1}\,\beta_{i,i_2}\ge0.
        \]
        Thus the right-hand side is strictly positive, while the left-hand
        side is zero, a contradiction.

        Taking total masses in the representations of
        \(\mu'_{j,1}\) and \(\mu'_{j,2}\) gives
        \[
            \sum_{i=1}^k\alpha_{j,i}
            =
            \sum_{i=1}^k\beta_{j,i}
            =
            1.
        \]

        It follows that, for each \(j\in[k]\), there exists an index
        \(\sigma(j)\in[k]\) such that
        \[
            \alpha_{j,\sigma(j)}
            =
            \beta_{j,\sigma(j)}
            =
            1,
        \]
        and all other coefficients in the \(j\)-th rows of
        \((\alpha_{j,i})\) and \((\beta_{j,i})\) vanish.
        Hence
        \[
            \mu'_{j,1} = \mu_{\sigma(j),1},
            \qquad
            \mu'_{j,2} = \mu_{\sigma(j),2}.
        \]

        Since the left-hand side of \eqref{eqn:ji-equality} has rank \(k\), the map \(\sigma\) must be surjective, and hence is a permutation of \([k]\). Since the rank-one terms \(\mu_{i,1}\times\mu_{i,2}\), \(i\in[k]\), are linearly independent, it follows that

        \[
            w'_j=w_{\sigma(j)}
            \qquad\text{for all }j\in[k].
        \]
        Since the terms on the two sides of \eqref{eqn:ji-equality} differ only by the
        reordering induced by \(\sigma\), it follows that
        \[
            \sum_{i=1}^k w_i\,\delta_{(\mu_{i,1},\mu_{i,2})}
            =
            \sum_{i=1}^k w'_i\,\delta_{(\mu'_{i,1},\mu'_{i,2})}.
        \]
    \end{proof}

    \begin{proof}[Proof of Theorem \ref{thm:singular-ident}]
        By the theorem hypothesis for the edge \(\{1,2\}\in E\),
        \begin{align*}
            R_{\{1,2\}}\!\left(\sum_{i=1}^k w_i\,\delta_{(\mu_{i,1},\ldots,\mu_{i,d})}\right)
        &=
        R_{\{1,2\}}\!\left(\sum_{i=1}^k w'_i\,\delta_{(\mu'_{i,1},\ldots,\mu'_{i,d})}\right)\\
        \Rightarrow\qquad
        \sum_{i=1}^k w_i\,\mu_{i,1}\times \mu_{i,2}
        &=
        \sum_{i=1}^k w'_i\,\mu'_{i,1}\times \mu'_{i,2}.
        \numberthis\label{eqn:singular-ident-d=2}
        \end{align*}

        Since \(w_i>0\) for every \(i\), and since the families
        \(\mu_{1,1},\ldots,\mu_{k,1}\) and
        \(\mu_{1,2},\ldots,\mu_{k,2}\) are linearly independent,
        Proposition~\ref{prop:rank-characterization} implies that the
        left-hand side of \eqref{eqn:singular-ident-d=2} has rank \(k\).
        Hence the right-hand side also has rank \(k\). In particular,
        \(w'_i>0\) for every \(i\), since otherwise it would be a sum of
        at most \(k-1\) rank-one measures. Proposition~\ref{prop:rank-characterization}
        then implies that the families
        \(\mu'_{1,1},\ldots,\mu'_{k,1}\) and
        \(\mu'_{1,2},\ldots,\mu'_{k,2}\) are linearly independent.

        Applying the same argument to the equality associated with each
        edge \(\{a,b\}\in E\) shows that the primed families corresponding
        to both endpoints \(a\) and \(b\) are linearly independent. Since
        every vertex of \(G\) belongs to at least one edge, it follows that
        \(\mu'_{1,j},\ldots,\mu'_{k,j}\) are linearly independent for every
        \(j\in[d]\).

        Applying the same argument used in the proof of
        Proposition~\ref{prop:ji-ident}, for each \(j\in[k]\) there exist
        real coefficients \(\alpha_{j,i}\) and \(\beta_{j,i}\) such that
        \[
            \mu'_{j,1} = \sum_{i=1}^k \alpha_{j,i}\,\mu_{i,1},
            \qquad
            \mu'_{j,2} = \sum_{i=1}^k \beta_{j,i}\,\mu_{i,2}.
        \]
        These coefficients need not all be nonnegative, since joint
        irreducibility is not assumed. Nevertheless, because
        \(\mu'_{j,1}\) and \(\mu'_{j,2}\) are probability measures, the
        singularity assumptions imply that
        \[
            \alpha_{j,1}\ge0,
            \qquad
            \beta_{j,1}\ge0,
            \qquad j\in[k].
        \]
        A negative coefficient on the singular component would
        force the resulting measure to take a negative value.

        By the singularity assumption, there exist measurable sets
        \(A_1,A_2\subset\Omega\) such that
        \[
            \mu_{1,1}(A_1)=1,\quad \mu_{i,1}(A_1)=0,
            \qquad
            \mu_{1,2}(A_2)=1,\quad \mu_{i,2}(A_2)=0,
        \]
        for all \(i=2,\ldots,k\). Thus
        \[
            \mu_{1,1}(A_1^c)=0,\quad \mu_{i,1}(A_1^c)=1,
            \qquad
            \mu_{1,2}(A_2^c)=0,\quad \mu_{i,2}(A_2^c)=1,
        \]
        for all \(i=2,\ldots,k\).

        Fix \(\ell\in[2]\), and suppose that
        \[
            \nu
            \;\triangleq\;
            \sum_{i=1}^k \gamma_i\,\mu_{i,\ell}
        \]
        is a probability measure.

        Let
        \[
            \nu^\perp \triangleq \sum_{i=2}^k \gamma_i\,\mu_{i,\ell},
        \]
        so \(\nu^\perp\) denotes the part of \(\nu\) that is singular with respect to
        \(\mu_{1,\ell}\).
        Then
        \[
            \nu
            =
            \gamma_1\,\mu_{1,\ell}
            +
            \nu^\perp.
        \]

        Now let \(B\subset A_\ell^c\) be measurable. Then
        \[
            \nu(B)
            =
            \gamma_1\,\mu_{1,\ell}(B)
            +
            \sum_{i=2}^k \gamma_i\,\mu_{i,\ell}(B)
            =
            \sum_{i=2}^k \gamma_i\,\mu_{i,\ell}(B),
        \]
        since \(\mu_{1,\ell}(B)=0\). Since \(\nu\) is a probability measure,
        \(\nu(B)\ge 0\) for all such \(B\). Since \(\nu^\perp\) is carried by
        \(A_\ell^c\), it follows that \(\nu^\perp\) is a nonnegative measure.

        If there exists an index \(i_0\ge2\) with \(\gamma_{i_0}\neq0\), then \(\nu^\perp\neq0\) by linear independence of \(\mu_{2,\ell},\ldots,\mu_{k,\ell}\). Since \(\nu^\perp\) is nonnegative, it follows that \(\nu^\perp(\Omega)>0\). Therefore
        \begin{equation} \label{eqn:singular-ident-positive}
            \nu(A_\ell^c)
            =
            \sum_{i=1}^k \gamma_i\,\mu_{i,\ell}(A_\ell^c)
            =
            \sum_{i=2}^k \gamma_i
            >0.
        \end{equation}
        Thus, if \(\gamma_i\neq0\) for any \(i\ge2\), then \(\nu\) assigns strictly positive mass to \(A_\ell^c\).

        With this established, it will now be shown that no primed component can mix the singular component in one coordinate with the remaining components in the other. To begin note that, for every \(i\in[k]\),
        \[
            (\mu_{i,1}\times\mu_{i,2})(A_1\times A_2^c)=0.
        \]
        For \(i=1\), this follows from \(\mu_{1,2}(A_2^c)=0\); for \(i\ge2\), it follows from \(\mu_{i,1}(A_1)=0\). Therefore, evaluating \eqref{eqn:singular-ident-d=2} on \(A_1\times A_2^c\) gives
        \begin{equation}\label{eqn:singular-zero}
            \sum_{i=1}^k w'_i
            (\mu'_{i,1}\times\mu'_{i,2})(A_1\times A_2^c)
            =
            \sum_{i=1}^k w_i
            (\mu_{i,1}\times\mu_{i,2})(A_1\times A_2^c)
            =
            0.
        \end{equation}

        Fix \(j\in[k]\), and suppose toward a contradiction that
        \[
            \alpha_{j,1} > 0
            \quad\text{and}\quad
            \beta_{j,i_0} \neq 0
            \ \text{for some } i_0\ge 2.
        \]
        Then
        \begin{align*}
            \sum_{i=1}^k w'_i (\mu'_{i,1}\times\mu'_{i,2})(A_1\times A_2^c)
                &\ge w'_j (\mu'_{j,1}\times\mu'_{j,2})(A_1\times A_2^c)\\
                &= w'_j \left(\sum_{i=1}^k \alpha_{j,i}\mu_{i,1}(A_1)\right) \left(\sum_{i=1}^k \beta_{j,i}\mu_{i,2}(A_2^c)\right)\\
                &= w'_j\,\alpha_{j,1}\sum_{i=2}^k \beta_{j,i}\mu_{i,2}(A_2^c)\\
                &= w'_j\,\alpha_{j,1}\sum_{i=2}^k \beta_{j,i}.
        \end{align*}
        By \eqref{eqn:singular-ident-positive}, applied with \(\ell=2\) and \(\gamma_i=\beta_{j,i}\), the assumption \(\beta_{j,i_0}\neq0\) implies \(\sum_{i=2}^k\beta_{j,i}>0\). 
        Since \(w'_j>0\) and \(\alpha_{j,1}>0\), the last line above is
        strictly positive, contradicting \eqref{eqn:singular-zero}. Consequently,
        \[
            \alpha_{j,1}>0
            \quad\Longrightarrow\quad
            \beta_{j,i}=0
            \quad\text{for every }i\ge2.
        \]
        An identical argument with the two coordinates interchanged gives
        \[
            \beta_{j,1}>0
            \quad\Longrightarrow\quad
            \alpha_{j,i}=0
            \quad\text{for every }i\ge2.
        \]
        Taking total masses in the representations of
        \(\mu'_{j,1}\) and \(\mu'_{j,2}\) gives
        \[
            \sum_{i=1}^k\alpha_{j,i}
            =
            \sum_{i=1}^k\beta_{j,i}
            =
            1.
        \]
        Therefore, if \(\alpha_{j,1}>0\), the first implication gives
        \(\beta_{j,1}=1\), and the second then gives
        \(\alpha_{j,1}=1\). If \(\alpha_{j,1}=0\), then
        \(\beta_{j,1}=0\), since otherwise the second implication would
        force every \(\alpha_{j,i}\) to vanish. Thus, for each \(j\in[k]\),
        either
        \[
            \alpha_{j,1}=\beta_{j,1}=1
        \]
        and all remaining coefficients in the \(j\)-th rows vanish, or
        \[
            \alpha_{j,1}=\beta_{j,1}=0.
        \]

        Let
        \[
            I \triangleq \{j\in[k]: \alpha_{j,1}=\beta_{j,1}=1\}.
        \]
        By the dichotomy above, if \(j\notin I\), then
        \(\alpha_{j,1}=\beta_{j,1}=0\). The set \(I\) cannot be empty: otherwise all
        terms on the right-hand side of \eqref{eqn:singular-ident-d=2} would have first
        and second marginals lying in the spans of
        \(\mu_{2,1},\ldots,\mu_{k,1}\) and \(\mu_{2,2},\ldots,\mu_{k,2}\), respectively,
        so the right-hand side would have rank at most \(k-1\), contradicting the
        rank-\(k\) property.

        The set \(I\) also cannot contain more than one index. If \(|I|\ge2\), then the
        corresponding summands all have product component
        \(\mu_{1,1}\times\mu_{1,2}\) and can be combined into a single rank-one term,
        leaving a representation with at most \(k-1\) rank-one terms. This again
        contradicts the rank-\(k\) property. Hence \(I=\{r\}\) for some \(r\in[k]\).

        For this index \(r\),
        \[
            \mu'_{r,1}=\mu_{1,1},
            \qquad
            \mu'_{r,2}=\mu_{1,2}.
        \]
        Evaluating \eqref{eqn:singular-ident-d=2} on \(A_1\times A_2\) gives
        \(w'_r=w_1\). Relabeling the primed components, assume \(r=1\). Thus
        \[
            \mu'_{1,1}=\mu_{1,1},
            \qquad
            \mu'_{1,2}=\mu_{1,2},
            \qquad
            w'_1=w_1.
        \]

        To complete the proof, it remains to show that
        \(\mu'_{1,j}=\mu_{1,j}\) for all \(j>2\). Fix \(j>2\). By hypothesis, either
        \(\{1,j\}\in E\) or \(\{2,j\}\in E\). The two cases are identical, so assume
        \(\{1,j\}\in E\).

        The equality of the \(\{1,j\}\)-marginals gives
        \[
            \sum_{i=1}^k w_i\,\mu_{i,1}\times\mu_{i,j}
            =
            \sum_{i=1}^k w'_i\,\mu'_{i,1}\times\mu'_{i,j}.
        \]
        Let \(B\subset\Omega\) be measurable. Evaluating both sides on \(A_1\times B\)
        gives
        \[
            w_1\,\mu_{1,j}(B)
            =
            w'_1\,\mu'_{1,j}(B),
        \]
        since \(\mu_{i,1}(A_1)=0\) for \(i\ge2\), and since \(I=\{1\}\), as established above, implies \(\mu'_{i,1}(A_1)=0\) for \(i\ge2\). Since \(w'_1=w_1>0\), it follows that \(\mu_{1,j}(B)=\mu'_{1,j}(B)\). As \(B\) was arbitrary, \(\mu_{1,j}=\mu'_{1,j}\). The case \(\{2,j\}\in E\) is identical, using \(A_2\) instead of \(A_1\).         Since \(I=\{1\}\), no other primed component agrees with
        \((\mu_{1,1},\ldots,\mu_{1,d})\) in its first two coordinates.
        Therefore
        \[
            \sum_{i=1}^k
            w'_i\,
            \delta_{(\mu'_{i,1},\ldots,\mu'_{i,d})}
            \bigl(\{(\mu_{1,1},\ldots,\mu_{1,d})\}\bigr)
            =
            w'_1
            =
            w_1,
        \]
        completing the proof.
    \end{proof}
    
    \section{Proofs of the Sharpness Results}\label{sec:proofs-sharpness}
    The proofs of Theorem~\ref{thm:2k-necessary} and Proposition~\ref{prop:2k-geometry-necessary} use the following elementary facts about convolution. The first states that pushforward under a linear map commutes with convolution.

    \begin{lemma}\label{lem:pushforward-convolution}
        Let \(\mu_1,\mu_2\in\sM(\R^2)\), and let
        \(T:\R^2\to\R\) be linear. Then
        \begin{equation}\label{eqn:pushforward-convolution}
            T\#(\mu_1*\mu_2)
            =
            (T\#\mu_1)*(T\#\mu_2).
        \end{equation}
    \end{lemma}
    \begin{proof}[Proof of Lemma \ref{lem:pushforward-convolution}]
        Let \(f:\R\to\R\) be bounded and measurable. Then
        \begin{align*}
            \int_{\R} f\,d\bigl(T\#(\mu_1*\mu_2)\bigr)
                &=
                \int_{\R^2}
                f\bigl(T(x)\bigr)\,d(\mu_1*\mu_2)(x)\\
                &=
                \int_{\R^2}\int_{\R^2}
                f\bigl(T(x_1+x_2)\bigr)\,
                d\mu_1(x_1)\,d\mu_2(x_2)\\
                &=
                \int_{\R^2}\int_{\R^2}
                f\bigl(T(x_1)+T(x_2)\bigr)\,
                d\mu_1(x_1)\,d\mu_2(x_2)\\
                &=
                \int_{\R^2}
                \left(
                    \int_{\R}
                    f\bigl(y_1+T(x_2)\bigr)\,
                    d(T\#\mu_1)(y_1)
                \right)
                d\mu_2(x_2)\\
                &=
                \int_{\R}
                \left(
                    \int_{\R}
                    f(y_1+y_2)\,
                    d(T\#\mu_1)(y_1)
                \right)
                d(T\#\mu_2)(y_2)\\
                &=
                \int_{\R}
                f\,d\bigl((T\#\mu_1)*(T\#\mu_2)\bigr).
        \end{align*}
        The first, fourth, and fifth equalities use
        Lemma~\ref{lem:substitution-signed}; the second and final equalities
        use the convolution integration formula
        \cite[Proposition~8.48]{folland99}; and the third uses the linearity of
        \(T\). Hence
        \[
            T\#(\mu_1*\mu_2)
            =
            (T\#\mu_1)*(T\#\mu_2).
        \]
    \end{proof}
    The second fact ensures that the atomic convolutions constructed later in this section are nonzero.
    \begin{lemma}\label{lem:nonzero-convolution}
        Let \(d,m\in\N\), and for each \(i\in[m]\), let
        \(x_i\in\R^d\setminus\{0\}\) and \(c_i\in\R\). Then
        \[
            \bigl(\delta_0+c_1\delta_{x_1}\bigr)
            *\cdots*
            \bigl(\delta_0+c_m\delta_{x_m}\bigr)
            \neq0.
        \]
    \end{lemma}
    The proof of Lemma \ref{lem:nonzero-convolution} uses the Fourier transform. For a finite signed measure \(\mu\in\sM(\R^d)\), write
    \[
        \fourier[\mu](\xi)
        \triangleq
        \int_{\R^d}e^{-i\langle\xi,x\rangle}\,d\mu(x),
        \qquad \xi\in\R^d,
    \]
    for its Fourier transform.

    \begin{proof}[Proof of Lemma \ref{lem:nonzero-convolution}]
        Set
        \[
            \eta
            =
            \bigl(\delta_0+c_1\delta_{x_1}\bigr)
            *\cdots*
            \bigl(\delta_0+c_m\delta_{x_m}\bigr).
        \]
        For each \(i\in[m]\),
        \[
            \fourier[\delta_0+c_i\delta_{x_i}](\xi)
            =
            1+c_i e^{-i\langle\xi,x_i\rangle}.
        \]
        Since \(x_i\neq0\), the zero set of this function is either empty or a countable union of affine hyperplanes in \(\R^d\), and therefore has Lebesgue measure zero. By the convolution theorem,
        \[
            \fourier[\eta](\xi)
            =
            \prod_{i=1}^m
            \left(
                1+c_i e^{-i\langle\xi,x_i\rangle}
            \right),
        \]
        which is nonzero almost everywhere. The injectivity of the Fourier transform on finite signed measures therefore implies that \(\eta\neq0\).
    \end{proof}

    The third fact describes how convolution by a product measure acts on a sum of product measures.

    \begin{lemma}\label{lem:product-convolution}
        Let \(k\in\N\), and let
        \(\mu_{i,j},\Phi_j\in\sM(\R)\) for \(i\in[k]\) and \(j\in[2]\). Then
        \[
            \left(
                \sum_{i=1}^k
                \mu_{i,1}\times\mu_{i,2}
            \right)
            *
            (\Phi_1\times\Phi_2)
            =
            \sum_{i=1}^k
            (\mu_{i,1}*\Phi_1)
            \times
            (\mu_{i,2}*\Phi_2).
        \]
    \end{lemma}
    \begin{proof}[Proof of Lemma~\ref{lem:product-convolution}]
        By linearity of convolution, it suffices to show that
        \[
            (\mu_1\times\mu_2)*(\Phi_1\times\Phi_2)
            =
            (\mu_1*\Phi_1)\times(\mu_2*\Phi_2).
        \]
        Let \(f:\R^2\to\R\) be bounded and measurable. The convolution
        integration formula and Fubini's theorem give
        \begin{align*}
        &\int_{\R^2}
        f\,d\bigl((\mu_1\times\mu_2)*(\Phi_1\times\Phi_2)\bigr)\\
        &\qquad=
        \int_{\R^2}\int_{\R^2}
        f(x_1+y_1,x_2+y_2)\,
        d(\mu_1\times\mu_2)(x_1,x_2)\,
        d(\Phi_1\times\Phi_2)(y_1,y_2)\\
        &\qquad=
        \int_{\R^2}
        f\,d\bigl((\mu_1*\Phi_1)\times(\mu_2*\Phi_2)\bigr).
        \end{align*}
        Hence
        \[
            (\mu_1\times\mu_2)*(\Phi_1\times\Phi_2)
            =
            (\mu_1*\Phi_1)\times(\mu_2*\Phi_2).
        \]
        Summing this identity over \(i\in[k]\) proves the result.
    \end{proof}
    The proof of Theorem~\ref{thm:2k-necessary} follows from the following more general nullspace construction.
    \begin{proposition}\label{prop:balanced-projection-nullspace}
        Let \(m\in\N\) and let \(v_1,\ldots,v_m\in\R^2\).
        Then there exists a nonzero compactly supported \(\eta\in\sM(\R^2)\) of rank at most \(m+1\) such that
        \[
            \pi_{v_\ell}\#\eta=0,
            \qquad \ell\in[m].
        \]
        Moreover, \(\eta\) admits a decomposition
        \[
            \eta=\mu-\mu',
        \]
        where \(\mu,\mu'\) are compactly supported positive measures of equal nonzero mass and admit representations
        \[
            \mu
            =
            \sum_{i=1}^{r_+}
            \mu_{i,1}\times\mu_{i,2},
            \qquad
            \mu'
            =
            \sum_{i=1}^{r_-}
            \mu'_{i,1}\times\mu'_{i,2},
        \]
        where
        \[
            r_+
            \le
            \left\lceil\frac{m+1}{2}\right\rceil,
            \qquad
            r_-
            \le
            \left\lfloor\frac{m+1}{2}\right\rfloor,
        \]
        and all of the measures
        \(\mu_{i,j}\) and \(\mu'_{i,j}\) are compactly supported and
        nonnegative.
    \end{proposition}

    \begin{proof}[Proof of Proposition \ref{prop:balanced-projection-nullspace}]
        It suffices to prove the result when \(v_i\neq0\) for every \(i\in[m]\). If \(v\neq0\), any \(\eta\) satisfying \(\pi_v\#\eta=0\) must also satisfy \(\eta(\R^2)=0\), equivalently \(\pi_0\#\eta=0\). Replacing a zero \(v_i\) by an arbitrary nonzero vector therefore only strengthens the required conditions.

        Write \(v_i=[a_i,b_i]^\top\). For each \(i\in[m]\), define
        \[
            x_i
            \triangleq
            \begin{cases}
                [1,-a_i/b_i]^\top, & b_i\neq0,\\
                [0,1]^\top, & b_i=0.
            \end{cases}
        \]
        Then \(\langle v_i,x_i\rangle=0\) for every \(i\in[m]\). For each
        \(i\in[m]\), define
        \[
            \Delta_i
            \triangleq
            \delta_{0}-\delta_{x_i}.
        \]
        It follows that
        \[
            \pi_{v_i}\#\Delta_i
            =
            \delta_{0}-\delta_{\langle v_i,x_i\rangle}
            =
            0.
        \]
        Let $\eta = \Delta_1 * \cdots * \Delta_m$.

        By repeated application of Lemma~\ref{lem:pushforward-convolution},
        for every \(i\in[m]\),
        \[
            \pi_{v_i}\#\eta
            =
            (\pi_{v_i}\#\Delta_1)*\cdots*(\pi_{v_i}\#\Delta_m)
            =
            0,
        \]
        because \(\pi_{v_i}\#\Delta_i=0\). Moreover, Lemma~\ref{lem:nonzero-convolution} implies that \(\eta\neq0\).

        Because each \(x_i\) has first coordinate either \(0\) or \(1\), and because \(\delta_x*\delta_y=\delta_{x+y}\), every point mass appearing in the expansion of the convolution has first coordinate in \(\{0,\ldots,m\}\). Hence \(\eta\) is supported on
        \(\{0,\ldots,m\}\times\R\). Consequently, for suitable compactly supported signed measures \(\rho_0,\ldots,\rho_m\in\sM(\R)\),
        \[
            \eta
            =
            \sum_{j=0}^m\delta_j\times\rho_j,
        \]
        and \(\eta\) has rank at most \(m+1\). This establishes the first assertion of the proposition. 

        To construct the measures \(\mu\) and \(\mu'\) in the proposition, consider two cases. First, suppose that the collection contains at most one coordinate-axis direction. After interchanging the two coordinates if necessary, assume without loss of generality that no \(v_i\) has the form \([q,0]^\top\).

        Write \(x_i=[1,c_i]^\top\), so that
        \[
            \Delta_i
            =
            \delta_{0}-\delta_{x_i}.
        \]
        For \(\varepsilon=(\varepsilon_1,\ldots,\varepsilon_m)
        \in\{0,1\}^m\), let
        \[
            |\varepsilon|
            \triangleq
            \sum_{i=1}^m\varepsilon_i.
        \]
        By distributivity of convolution over addition, expanding \(\Delta_1*\cdots*\Delta_m\) produces one term for each \(\varepsilon\in\{0,1\}^m\). In this indexing, \(\varepsilon_i=0\) corresponds to choosing \(\delta_{0}\) from \(\Delta_i\), whereas \(\varepsilon_i=1\) corresponds to choosing \(-\delta_{x_i}\). Accordingly,
        \begin{align*}
            \eta
                &=
                \sum_{\varepsilon\in\{0,1\}^m}
                (-1)^{|\varepsilon|}
                \left(
                    \delta_{\varepsilon_1x_1}
                    *\cdots*
                    \delta_{\varepsilon_mx_m}
                \right)\\
                &=
                \sum_{\varepsilon\in\{0,1\}^m}
                (-1)^{|\varepsilon|}
                \delta_{\sum_{i=1}^m\varepsilon_i x_i}.
        \end{align*}
        Since \(x_i=[1,c_i]^\top\),
        \[
            \sum_{i=1}^m\varepsilon_i x_i
            =
            \left[
                |\varepsilon|,
                \sum_{i=1}^m\varepsilon_i c_i
            \right]^\top.
        \]
        For \(j=0,\ldots,m\), define
        \[
            \eta_j
            \triangleq
            \sum_{\substack{\varepsilon\in\{0,1\}^m\\
            |\varepsilon|=j}}
            \delta_{\sum_{i=1}^m\varepsilon_i c_i}.
        \]
        Each \(\eta_j\) is a compactly supported nonnegative measure on
        \(\R\), and grouping the preceding expansion according to
        \(|\varepsilon|\) gives
        \[
            \eta
            =
            \sum_{j=0}^m
            (-1)^j\delta_j\times\eta_j.
        \]

        Define
        \[
            \mu
            \triangleq
            \sum_{\substack{0\le j\le m\\j\ \mathrm{even}}}
            \delta_j\times\eta_j,
            \qquad
            \mu'
            \triangleq
            \sum_{\substack{0\le j\le m\\j\ \mathrm{odd}}}
            \delta_j\times\eta_j.
        \]
        Then \(\mu\) and \(\mu'\) are compactly supported nonnegative measures,
        and
        \[
            \eta=\mu-\mu'.
        \]
        There are
        \[
            \left\lceil\frac{m+1}{2}\right\rceil
        \]
        even integers and
        \[
            \left\lfloor\frac{m+1}{2}\right\rfloor
        \]
        odd integers in \(\{0,\ldots,m\}\). Thus, the displayed expressions
        write \(\mu\) and \(\mu'\) as sums of at most the required numbers of
        product measures with compactly supported nonnegative factors.

        Finally, pushforward preserves total mass. Hence, for any \(i\in[m]\),
        \[
            0
            =
            \bigl(\pi_{v_i}\#\eta\bigr)(\R)
            =
            \eta(\R^2)
            =
            \mu(\R^2)-\mu'(\R^2).
        \]
        Thus \(\mu\) and \(\mu'\) have equal mass. Since they are nonnegative
        and \(\eta=\mu-\mu'\neq0\), their common mass is nonzero.

        It remains to consider the case in which both coordinate-axis directions occur. Then \(v_i=[q,0]^\top\) for some \(i\), where \(q\neq0\), and the corresponding vector is \(x_i=[0,1]^\top\). Hence at most \(m-1\) of the vectors \(x_1,\ldots,x_m\) have nonzero first coordinate. The preceding convolution argument therefore implies that \(\eta\) is supported on \(\{0,\ldots,m-1\}\times\R\). Consequently, there exist compactly supported signed measures \(\rho_0,\ldots,\rho_{m-1}\in\sM(\R)\) such that
        \[
            \eta
            =
            \sum_{j=0}^{m-1}\delta_j\times\rho_j.
        \]

        Define compactly supported signed measures
        \(\zeta_0,\ldots,\zeta_m\) recursively by
        \[
            \zeta_0=0,
            \qquad
            \zeta_{j+1}
            =
            \zeta_j+(-1)^{j+1}\rho_j,
            \qquad j=0,\ldots,m-1.
        \]
        Let
        \[
            \lambda
            =
            \sum_{r=0}^m|\zeta_r|,
            \qquad
            \gamma_r
            =
            \lambda+\zeta_r,
            \qquad r=0,\ldots,m.
        \]
        Since \(\lambda\ge|\zeta_r|\), each \(\gamma_r\) is a compactly supported nonnegative measure.
        Moreover, for \(j=0,\ldots,m-1\),
        \begin{align*}
            \gamma_j-\gamma_{j+1}
                &=
                (\lambda+\zeta_j)-(\lambda+\zeta_{j+1})\\
                &=
                \zeta_j-\zeta_{j+1}\\
                &=
                (-1)^j\rho_j,
        \end{align*}
        where the final equality follows from the recursion defining the
        \(\zeta_j\). Therefore,
        \[
            \rho_j
            =
            (-1)^j(\gamma_j-\gamma_{j+1}).
        \]

        For \(r=0,\ldots,m\), define
        \[
            \alpha_r
            =
            \begin{cases}
                \delta_0,
            & r=0,\\
            \delta_{r-1}+\delta_r,
            & 1\le r\le m-1,\\
            \delta_{m-1},
            & r=m.
            \end{cases}
        \]
        Substituting the preceding expression for \(\rho_j\) and
        reindexing the terms involving \(\gamma_{j+1}\) gives
        \begin{align*}
            \eta
                &=
                \sum_{j=0}^{m-1}
                (-1)^j\delta_j\times(\gamma_j-\gamma_{j+1})\\
                &=
                \sum_{j=0}^{m-1}
                (-1)^j\delta_j\times\gamma_j
                +
                \sum_{j=0}^{m-1}
                (-1)^{j+1}\delta_j\times\gamma_{j+1}\\
                &=
                \sum_{j=0}^{m-1}
                (-1)^j\delta_j\times\gamma_j
                +
                \sum_{r=1}^{m}
                (-1)^r\delta_{r-1}\times\gamma_r\\
                &=
                \delta_0\times\gamma_0
                +
                \sum_{r=1}^{m-1}
                (-1)^r(\delta_{r-1}+\delta_r)\times\gamma_r
                +
                (-1)^m\delta_{m-1}\times\gamma_m\\
                &=
                \sum_{r=0}^m
                (-1)^r\alpha_r\times\gamma_r.
        \end{align*}
        Therefore, set
        \[
            \mu
            =
            \sum_{\substack{0\le r\le m\\r\ \mathrm{even}}}
            \alpha_r\times\gamma_r,
            \qquad
            \mu'
            =
            \sum_{\substack{0\le r\le m\\r\ \mathrm{odd}}}
            \alpha_r\times\gamma_r.
        \]
        Then \(\eta=\mu-\mu'\), with $\alpha_r$ and $\gamma_r$ compactly supported and nonnegative for all $r$. The numbers of even and odd indices in
        \(\{0,\ldots,m\}\) are respectively
        \(\lceil(m+1)/2\rceil\) and \(\lfloor(m+1)/2\rfloor\), giving the
        required numbers of product terms.

        As in the first case, the identities \(\pi_{v_i}\#\eta=0\) imply that \(\mu\) and \(\mu'\) have equal nonzero mass. This completes the proof.
    \end{proof}

    \begin{proof}[Proof of Theorem~\ref{thm:2k-necessary}]
        Apply Proposition~\ref{prop:balanced-projection-nullspace} with \(m=2k-1\). Normalizing the two positive measures by their common mass, normalizing their factors, absorbing the factor masses into mixture weights, and adding zero-weight components if necessary gives nonnegative weights \(w_i,w_i'\) and compactly supported probability measures \(\rho_{i,j},\rho_{i,j}'\), for \(i\in[k]\) and \(j\in[2]\), such that
        \[
            \sum_{i=1}^k w_i
            =
            \sum_{i=1}^k w_i'
            =
            1,
        \]
        and the distinct probability measures
        \[
            \rho
            =
            \sum_{i=1}^k
            w_i\,\rho_{i,1}\times\rho_{i,2},
            \qquad
            \rho'
            =
            \sum_{i=1}^k
            w_i'\,\rho_{i,1}'\times\rho_{i,2}'
        \]
        satisfy
        \[
            \pi_{v_\ell}\#\rho
            =
            \pi_{v_\ell}\#\rho'
            \qquad
            \text{for every }\ell\in[2k-1].
        \]

        Let \(\Phi_1,\Phi_2\) be probability measures with smooth, compactly
        supported densities, and define
        \[
            \mu_{i,j}
            =
            \rho_{i,j}*\Phi_j,
            \qquad
            \mu_{i,j}'
            =
            \rho_{i,j}'*\Phi_j,
            \qquad
            i\in[k],\ j\in[2].
        \]
        Each \(\mu_{i,j}\) and \(\mu_{i,j}'\) thus has a smooth, compactly supported
        probability density. Set
        \[
            \mu
            =
            \sum_{i=1}^k
            w_i\,\mu_{i,1}\times\mu_{i,2},
            \qquad
            \mu'
            =
            \sum_{i=1}^k
            w_i'\,\mu_{i,1}'\times\mu_{i,2}'.
        \]
        Writing \(\Phi=\Phi_1\times\Phi_2\), Lemma~\ref{lem:product-convolution}
        gives
        \[
            \mu=\rho*\Phi,
            \qquad
            \mu'=\rho'*\Phi.
        \]
        Hence, by Lemma~\ref{lem:pushforward-convolution}, for every
        \(\ell\in[2k-1]\),
        \begin{align*}
            \pi_{v_\ell}\#\mu
        &=
        (\pi_{v_\ell}\#\rho)*(\pi_{v_\ell}\#\Phi)\\
        &=
        (\pi_{v_\ell}\#\rho')*(\pi_{v_\ell}\#\Phi)\\
        &=
        \pi_{v_\ell}\#\mu'.
        \end{align*}

        Finally, \(\rho-\rho'\) is nonzero and compactly supported, so its
        Fourier transform is nonzero and analytic. Since
        \(\fourier[\Phi](0)=1\), the convolution theorem implies that
        \((\rho-\rho')*\Phi\neq0\). Therefore, \(\mu\neq\mu'\).
    \end{proof}

    The proof of Proposition~\ref{prop:2k-geometry-necessary} again uses the normalization and smoothing steps from the proof of Theorem~\ref{thm:2k-necessary}, so those parts are treated briefly. The main task is to exploit the symmetry of the paired directions to construct a nonzero signed measure that is annihilated by all \(2k\) projections and, after normalization, can be written as the difference of two mixture-of-products models, each with at most \(\lfloor k/2\rfloor+1\) components.
    \begin{proof}[Proof of Proposition~\ref{prop:2k-geometry-necessary}]
        For \(r\in[k]\), let
        \[
            x_r^+=[1,-a_r]^\top,
            \qquad
            x_r^-=[1,a_r]^\top.
        \]
        Then
        \[
            \pi_{v_r^+}(x_r^+)=0,
            \qquad
            \pi_{v_r^-}(x_r^-)=0.
        \]
        Define
        \[
            \Delta_r^+
            =
            \delta_{0}-\delta_{x_r^+},
            \qquad
            \Delta_r^-
            =
            \delta_{0}-\delta_{x_r^-},
        \]
        and set
        \[
            \Gamma_r=\Delta_r^+*\Delta_r^-.
        \]
        Since pushforward commutes with convolution,
        \[
            \pi_{v_r^+}\#\Gamma_r=0,
            \qquad
            \pi_{v_r^-}\#\Gamma_r=0.
        \]
        Consequently, the measure
        \[
            \eta
            =
            \Gamma_1*\cdots*\Gamma_k
        \]
        satisfies
        \[
            \pi_{v_r^+}\#\eta
            =
            \pi_{v_r^-}\#\eta
            =
            0
            \qquad
            \text{for every }r\in[k],
        \]
        and is nonzero by Lemma \ref{lem:nonzero-convolution}.

        Let
        \[
            S_a
            =
            \frac12(\delta_{-a}+\delta_a),
            \qquad
            A=\delta_0+\delta_2.
        \]
        Expanding \(\Gamma_r\) gives
        \[
            \Gamma_r
            =
            \delta_{0}
            -
            \delta_{[1,-a_r]^\top}
            -
            \delta_{[1,a_r]^\top}
            +
            \delta_{[2,0]^\top}
            =
            A\times\delta_0
            -
            2\delta_1\times S_{a_r}.
        \]
        For \(s=0,\ldots,k\), define the positive measures
        \[
            \alpha_s
            =
            2^s\delta_s*A^{*(k-s)}
        \]
        and
        \[
            \beta_s
            =
            \sum_{\substack{J\subseteq[k]\\|J|=s}}
            \mathop{\scalebox{1.5}{\raisebox{-0.2ex}{$\ast$}}}\limits_{r\in J}
            S_{a_r},
        \]
        where the large asterisk denotes the convolution of the measures \(S_{a_r}\) over \(r\in J\), with the empty convolution interpreted as \(\delta_0\).
        Expanding the convolution according to the number of factors from
        which the second term is selected gives
        \[
            \eta
            =
            \sum_{s=0}^k
            (-1)^s\alpha_s\times\beta_s.
        \]
        Hence, setting
        \[
            \eta_\textrm{even}
            =
            \sum_{\substack{0\leq s\leq k\\s\text{ even}}}
            \alpha_s\times\beta_s,
            \qquad
            \eta_\textrm{odd}
            =
            \sum_{\substack{0\leq s\leq k\\s\text{ odd}}}
            \alpha_s\times\beta_s,
        \]
        gives
        \[
            \eta=\eta_\mathrm{even}-\eta_\mathrm{odd}.
        \]
        The measures \(\eta_\mathrm{even}\) and \(\eta_\mathrm{odd}\) are positive and compactly
        supported, with ranks at most
        \[
            \left\lfloor\frac{k}{2}\right\rfloor+1
            \qquad\text{and}\qquad
            \left\lceil\frac{k}{2}\right\rceil,
        \]
        respectively. 
        The projection equalities imply that \(\eta(\R^2)=0\), so
        \(\eta_\mathrm{even}\) and \(\eta_\mathrm{odd}\) have the same nonzero mass.
        Set
        \[
            q=\left\lfloor\frac{k}{2}\right\rfloor+1.
        \]
        Normalizing as
        in the proof of Theorem~\ref{thm:2k-necessary} and padding the smaller
        representation with a zero-weight component gives weights
        \(w_i,w_i'\geq0\) and compactly supported probability measures
        \(\rho_{i,j},\rho_{i,j}'\in\sP(\R)\), for \(i\in[q]\) and \(j\in[2]\),
        such that
        \[
            \sum_{i=1}^q w_i
            =
            \sum_{i=1}^q w_i'
            =
            1,
        \]
        and the distinct probability measures
        \[
            \rho
            =
            \sum_{i=1}^q
            w_i\,\rho_{i,1}\times\rho_{i,2},
            \qquad
            \rho'
            =
            \sum_{i=1}^q
            w_i'\,\rho_{i,1}'\times\rho_{i,2}'
        \]
        have equal projections along every \(v_r^\pm\).

        Let \(\Phi_1,\Phi_2\) be probability measures with smooth, compactly
        supported densities, and define
        \[
            \mu_{i,j}
            =
            \rho_{i,j}*\Phi_j,
            \qquad
            \mu_{i,j}'
            =
            \rho_{i,j}'*\Phi_j,
            \qquad
            i\in[q],\ j\in[2].
        \]
        Each of these measures has a smooth, compactly supported density. Set
        \[
            \mu
            =
            \sum_{i=1}^q
            w_i\,\mu_{i,1}\times\mu_{i,2},
            \qquad
            \mu'
            =
            \sum_{i=1}^q
            w_i'\,\mu_{i,1}'\times\mu_{i,2}'.
        \]
        By Lemma~\ref{lem:product-convolution}, these are obtained by convolving
        \(\rho\) and \(\rho'\) with
        \(\Phi_1\times\Phi_2\). Lemma~\ref{lem:pushforward-convolution} therefore
        preserves all the prescribed projection equalities, while the same
        Fourier-analytic argument used in the proof of Theorem~\ref{thm:2k-necessary} gives
        \(\mu\neq\mu'\).
    \end{proof}

    \begin{proof}[Proof of Lemma~\ref{lem:proj-null}]
        Recall that the one-dimensional Schwartz space \(\schwartz(\R)\) consists of smooth functions \(f\in C^\infty(\R)\) satisfying
        \[
            |f|_{N,\alpha}
            \triangleq
            \sup_{x\in\R}(1+|x|)^N|f^{(\alpha)}(x)|
            <
            \infty
            \qquad
            \text{for all integers }N,\alpha\ge0,
        \]
        and that \(\schwartz(\R)\subset L^1(\R)\)
        \cite[p.~237]{folland99}. Let \(\fourier\) denote the Fourier
        transform.

        Write
        \[
            v_i=[a_i,b_i]^\top,
            \qquad i\in[\ell],
        \]
        and choose \(c\ge0\) such that
        \[
            \left|\frac{b_i}{a_i}\right|\le c
            \qquad
            \text{whenever }a_i\neq0.
        \]
        Such a \(c\) exists because the collection of directions is finite.

        The measures will be constructed in the Fourier domain. Let
        \(\phi\in\schwartz(\R)\) be a nonnegative, even \(C^\infty\) bump
        function \cite[Section~13]{tu10} such that
        \[
            \phi(0)=1,
            \qquad
            \supp(\phi)\subset[-1,1].
        \]
        Define
        \[
            \widehat{\mu}(\xi)
            =
            \xi^2\phi(\xi),
            \qquad
            \widehat{\nu}(\xi)
            =
            \phi(\xi-(c+2))
            +
            \phi(\xi+(c+2)).
        \]
        These functions are nonzero, real-valued, even, and belong to
        \(\schwartz(\R)\). Moreover,
        \[
            \supp(\widehat{\mu})\subset[-1,1],
            \qquad
            \supp(\widehat{\nu})
            \subset
            [c+1,c+3]\cup[-c-3,-c-1],
        \]
        and
        \[
            \widehat{\mu}(0)=\widehat{\nu}(0)=0.
        \]

        Let
        \[
            f_\mu=\fourier^{-1}[\widehat{\mu}],
            \qquad
            f_\nu=\fourier^{-1}[\widehat{\nu}].
        \]
        Since the Fourier transform maps \(\schwartz(\R)\) to itself \cite[Corollary~8.23]{folland99}, one has \(f_\mu,f_\nu\in\schwartz(\R)\subset L^1(\R)\). Because \(\widehat{\mu}\) and \(\widehat{\nu}\) are real-valued and even, their inverse Fourier transforms are real-valued. They are also nonzero by the injectivity of the Fourier transform.
        Let $\mu,\nu \in \sM(\R)$ be the signed measures corresponding to the densities $f_\mu,f_\nu$, so $\fourier[\mu] = \widehat{\mu}$ and $\fourier[\nu] = \widehat{\nu}$.
        From this it follows that
        \[
            \mu(\R)=\widehat{\mu}(0)=0,
            \qquad
            \nu(\R)=\widehat{\nu}(0)=0.
        \]

        Fix \(i\in[\ell]\). If \(a_i=0\), then
        \[
            \pi_{v_i}\#(\mu\times\nu)
            =
            \mu(\R)\,T_{b_i}\#\nu
            =
            0,
        \]
        where \(T_b(x)=bx\).

        Now suppose that \(a_i\neq0\), and set
        \[
            s_i=\frac{b_i}{a_i}.
        \]
        Since
        \[
            \pi_{[1,s_i]}\#(\mu\times\nu)
            =
            \mu*(T_{s_i}\#\nu),
        \]
        its Fourier transform satisfies
        \[
            \fourier\!\left[
                \pi_{[1,s_i]}\#(\mu\times\nu)
            \right](\xi)
            =
            \widehat{\mu}(\xi)\widehat{\nu}(s_i\xi).
        \]
        If \(\xi\in\supp(\widehat{\mu})\), then \(|\xi|\le1\), and hence
        \[
            |s_i\xi|
            \le
            |s_i|
            \le
            c.
        \]
        Since \(\widehat{\nu}(\eta)=0\) whenever \(|\eta|\le c\), one has \(\widehat{\nu}(s_i\xi)=0\) whenever \(\xi\in\supp(\widehat{\mu})\). Thus,
        \[
            \widehat{\mu}(\xi)\widehat{\nu}(s_i\xi)=0
            \qquad
            \text{for every }\xi\in\R.
        \]
        The injectivity of the Fourier transform therefore gives
        \[
            \pi_{[1,s_i]}\#(\mu\times\nu)=0.
        \]
        Since
        \[
            \pi_{v_i}
            =
            T_{a_i}\circ\pi_{[1,s_i]},
        \]
        it follows that
        \[
            \pi_{v_i}\#(\mu\times\nu)
            =
            T_{a_i}\#
            \left(
                \pi_{[1,s_i]}\#(\mu\times\nu)
            \right)
            =
            0.
        \]
        Thus \(\pi_{v_i}\#(\mu\times\nu)=0\) for every \(i\in[\ell]\).
    \end{proof}
    \begin{proof}[Proof of Corollary~\ref{cor:proj-null}]
        By the construction in the proof of Lemma~\ref{lem:proj-null}, there exist nonzero signed measures \(\alpha,\beta\in\sM(\R)\), each admitting a density in \(\schwartz(\R)\), such that
        \[
            \alpha(\R)=\beta(\R)=0
        \]
        and
        \[
            \pi_{v_i}\#(\alpha\times\beta)=0
            \qquad
            \text{for every }i\in[\ell].
        \]
        Write
        \[
            \alpha=\alpha^+-\alpha^-,
            \qquad
            \beta=\beta^+-\beta^-.
        \]
        Expanding \(\alpha\times\beta\) and using linearity gives
        \[
            \pi_{v_i}\#
            \bigl(
                \alpha^+\times\beta^+
                +
                \alpha^-\times\beta^-
            \bigr)
            =
            \pi_{v_i}\#
            \bigl(
                \alpha^+\times\beta^-
                +
                \alpha^-\times\beta^+
            \bigr)
        \]
        for every \(i\in[\ell]\).

        Since \(\alpha\) and \(\beta\) are nonzero and have total mass zero,
        \[
            \alpha^+(\R)=\alpha^-(\R)>0,
            \qquad
            \beta^+(\R)=\beta^-(\R)>0.
        \]
        Normalizing the four factors therefore turns each side into a
        rank-\(\le2\) mixture-of-products probability measure with both
        weights equal to \(1/2\). The resulting models are distinct because
        their difference is a nonzero multiple of \(\alpha\times\beta\).

        Finally, if \(f_\alpha\) and \(f_\beta\) denote the Schwartz-class densities of \(\alpha\) and \(\beta\), then the densities of their Jordan components are bounded pointwise by \(|f_\alpha|\) and \(|f_\beta|\). They therefore decay faster than any polynomial, and normalization preserves this property.
    \end{proof}

    The proof of Theorem~\ref{thm:boundedness-necessary} uses the following form of the uniform boundedness principle.

    \begin{theorem}[Banach--Steinhaus theorem {\cite[Theorem~3.15]{fabian01}}]\label{thm:uniform-boundedness-principle}
        Let \(X\) be a Banach space and let \(M\subset X^*\). If
        \[
            \sup_{m\in M}|m(x)|<\infty
            \qquad\text{for every }x\in X,
        \]
        then \(M\) is bounded in norm.
    \end{theorem}
    \begin{proof}[Proof of Theorem~\ref{thm:boundedness-necessary}]
        Let \(V\) denote the map in the theorem. Write \(v_i=[a_i,b_i]^\top\) for \(i\in[\ell]\), and equip \(\sM(\R)^\ell\) with the norm
        \[
            \|(\nu_1,\ldots,\nu_\ell)\|
            =
            \max_{i\in[\ell]}\|\nu_i\|.
        \]
        If \(V\) is not injective on
        \(\rankmeas_{\le1}([0,1],[0,1])\), then it is not a homeomorphic
        embedding. Hence, assume that \(V\) is injective on this class.

        Choose nonnegative functions
        \(p^+,p^-\in C_c^\infty((0,1))\) with disjoint supports and
        \[
            \int_{\R}p^+(x)\,dx
            =
            \int_{\R}p^-(x)\,dx
            =
            1.
        \]
        Let \(\mu^+\) and \(\mu^-\) be the corresponding probability measures
        and define
        \[
            \alpha=\mu^+-\mu^-,
            \qquad
            f=p^+-p^-.
        \]
        Then
        \[
            \alpha(\R)=0,
            \qquad
            \|\alpha\|=2.
        \]
        Since \(f\) is smooth and compactly supported, it has a finite
        Lipschitz constant, denoted by \(L\).

        For each \(\delta\in(0,1)\), choose nonnegative functions
        \(r_\delta^+,r_\delta^-\in C_c^\infty(\R)\) satisfying
        \[
            \int_{\R}r_\delta^+(y)\,dy
            =
            \int_{\R}r_\delta^-(y)\,dy
            =
            1
        \]
        and
        \[
            \supp(r_\delta^+)\subset(0,\delta/3),
            \qquad
            \supp(r_\delta^-)\subset(2\delta/3,\delta).
        \]
        Let \(\rho_\delta^+\) and \(\rho_\delta^-\) be the corresponding
        probability measures and set
        \[
            \beta_\delta=\rho_\delta^+-\rho_\delta^-,
            \qquad
            z_\delta=\frac14\,\alpha\times\beta_\delta.
        \]
        The measure \(z_\delta\) belongs to
        \(\rankmeas_{\le1}([0,1],[0,1])\) and has the \(C^\infty\) density
        \[
            (x,y)\longmapsto
            \frac14 f(x)\bigl(r_\delta^+(y)-r_\delta^-(y)\bigr).
        \]
        Moreover, the four product measures obtained by expanding
        \(\alpha\times\beta_\delta\) are mutually singular. Therefore,
        \[
            \|z_\delta\|
            =
            \frac14\|\alpha\times\beta_\delta\|
            =
            1.
        \]

        It remains to bound the projections of \(z_\delta\). For \(c\in\R\), let \(T_c(x)=cx\). By the convolution formula for projections of product measures,
        \[
            \pi_{[a_i,b_i]}\#(\alpha\times\beta_\delta)
            =
            T_{a_i}\#\alpha*T_{b_i}\#\beta_\delta.
        \]
        If \(a_i=0\), then
        \[
            T_{a_i}\#\alpha
            =
            \alpha(\R)\delta_0
            =
            0.
        \]
        Similarly, if \(b_i=0\), then
        \[
            T_{b_i}\#\beta_\delta
            =
            \beta_\delta(\R)\delta_0
            =
            0.
        \]
        Thus
        \[
            \pi_{v_i}\#z_\delta=0
        \]
        whenever \(a_i=0\) or \(b_i=0\).

        Now suppose that \(a_i\neq0\) and \(b_i\neq0\). The measure
        \(T_{a_i}\#\alpha\) has density
        \[
            f_{a_i}(x)
            =
            \frac{1}{|a_i|}f(x/a_i),
        \]
        whose Lipschitz constant is \(L/|a_i|^2\). Let \(h_{\delta,i}\) denote the density of
        \(\pi_{v_i}\#(\alpha\times\beta_\delta)\).

        To determine the density of the projection, consider first
        \(\mu^+\times\rho_\delta^+\). If \(X\sim\mu^+\) and
        \(Y\sim\rho_\delta^+\) are independent, then
        \[
            \pi_{v_i}\#(\mu^+\times\rho_\delta^+)
            =
            \law(a_iX+b_iY).
        \]
        Conditioning on \(Y=y\), the random variable \(a_iX+b_iY\) has density
        \[
            x\longmapsto
            \frac{1}{|a_i|}
            p^+\!\left(\frac{x-b_i y}{a_i}\right).
        \]
        Averaging this conditional density over the distribution of \(Y\)
        gives
        \[
            x\longmapsto
            \int_{\R}
            \frac{1}{|a_i|}
            p^+\!\left(\frac{x-b_i y}{a_i}\right)
            r_\delta^+(y)\,dy.
        \]
        Applying the same argument to the other three terms in
        \[
            \alpha\times\beta_\delta
            =
            \mu^+\times\rho_\delta^+
            -
            \mu^+\times\rho_\delta^-
            -
            \mu^-\times\rho_\delta^+
            +
            \mu^-\times\rho_\delta^-,
        \]
        and using linearity, shows that
        \(\pi_{v_i}\#(\alpha\times\beta_\delta)\) has density
        \[
            h_{\delta,i}(x)
            =
            \int_{\R}
            f_{a_i}(x-b_i y)
            \bigl(r_\delta^+(y)-r_\delta^-(y)\bigr)\,dy.
        \]
        Using the fact that \(r_\delta^+\) and \(r_\delta^-\) integrate to one, Fubini's theorem gives
        \begin{align*}
            h_{\delta,i}(x)
    &=
    \int_{\R}
    f_{a_i}(x-b_i y)r_\delta^+(y)\,dy
    -
    \int_{\R}
    f_{a_i}(x-b_i z)r_\delta^-(z)\,dz\\
    &=
    \left(
        \int_{\R}
        f_{a_i}(x-b_i y)r_\delta^+(y)\,dy
    \right)
    \left(
        \int_{\R}r_\delta^-(z)\,dz
    \right)\\
    &\quad-
    \left(
        \int_{\R}r_\delta^+(y)\,dy
    \right)
    \left(
        \int_{\R}
        f_{a_i}(x-b_i z)r_\delta^-(z)\,dz
    \right)\\
    &=
    \iint_{\R^2}
    f_{a_i}(x-b_i y)
    r_\delta^+(y)r_\delta^-(z)\,dy\,dz\\
    &\quad-
    \iint_{\R^2}
    f_{a_i}(x-b_i z)
    r_\delta^+(y)r_\delta^-(z)\,dy\,dz\\
    &=
    \iint_{\R^2}
    \bigl(
        f_{a_i}(x-b_i y)
        -
        f_{a_i}(x-b_i z)
    \bigr)
    r_\delta^+(y)r_\delta^-(z)\,dy\,dz.
        \end{align*}

        Since \(f_{a_i}\) has Lipschitz constant \(L/|a_i|^2\) and \(y,z\in[0,\delta]\) on the supports of the integrand,
        \[
            |h_{\delta,i}(x)|
            \le
            \frac{L|b_i|}{|a_i|^2}\,\delta.
        \]
        The support of \(h_{\delta,i}\) is contained in an interval of length
        at most \(|a_i|+|b_i|\), independently of \(\delta\). Consequently,
        there exists \(C_i<\infty\), independent of \(\delta\), such that
        \[
            \|\pi_{v_i}\#z_\delta\|
            =
            \frac14\|h_{\delta,i}\|_1
            \le
            C_i\delta.
        \]
        Since the collection of directions is finite, there exists
        \(C<\infty\) such that
        \[
            \|V(z_\delta)\|
            \le
            C\delta
            \qquad\text{for every }\delta\in(0,1).
        \]

        For \(n\ge2\), define
        \[
            x_n=nz_{1/n^2}.
        \]
        Then \(x_n\in\rankmeas_{\le1}([0,1],[0,1])\) and
        \[
            \|x_n\|=n,
            \qquad
            \|V(x_n)\|
            \le
            \frac{C}{n}.
        \]
        Hence \(V(x_n)\) converges to zero in norm and therefore also
        weak\(^*\).

        Let
        \[
            \widetilde V
            =
            \left(
                V\big|_{\rankmeas_{\le1}([0,1],[0,1])}
            \right)^{-1}.
        \]
        If \(\widetilde V\) were weak\(^*\)--weak\(^*\) continuous at zero,
        then
        \[
            x_n
            =
            \widetilde V(V(x_n))
            \towstar
            \widetilde V(0)
            =
            0.
        \]
        In particular, for every \(f\in C([0,1]^2)\), the sequence \(\bigl(\langle x_n,f\rangle\bigr)_{n=2}^\infty\) would be bounded. Theorem~\ref{thm:uniform-boundedness-principle} would then imply
        \[
            \sup_{n\ge2}\|x_n\|<\infty,
        \]
        contradicting \(\|x_n\|=n\). Thus \(\widetilde V\) is not
        weak\(^*\)--weak\(^*\) continuous, and \(V\) is not a weak\(^*\)--weak\(^*\)
        homeomorphic embedding.

        Finally, every \(x_n\) constructed above admits a \(C^\infty\)
        density. The same argument therefore applies when the domain is
        restricted to rank-\(\le1\) signed measures admitting \(C^\infty\)
        densities.
    \end{proof}
    \section*{Acknowledgments}
The author would like to thank his wife, Julia Vandermeulen, for her support throughout this project.

    \clearpage

    \appendix

    \section{Equivalence of RIP, RSC, and Injectivity} 
    \label{appx:rip-rsc-inj-equiv}
    This appendix shows that, for linear maps on matrix spaces, injectivity on rank-\(\le k\) matrices is equivalent to restricted strong convexity on rank-\(\le 2k\) matrices and, up to rescaling, to restricted isometry on rank-\(\le 2k\) matrices. The argument works with any choice of norms on \(\R^{m\times n}\) and \(\R^\ell\). For definiteness, equip these spaces with the Frobenius and Euclidean norms, respectively, both denoted by \(\|\cdot\|\). Linear operators between these spaces are equipped with the corresponding operator norm. For $q \in \N$, let \(R_{\le q} \subseteq \R^{m \times n}\) denote the set of matrices of rank at most \(q\), and let \(T : \R^{m \times n} \to \R^\ell\) be linear.

    Let $k \in \N$. To begin, assume that \(T\) is injective on \(R_{\le k}\). Any nonzero \(x\in R_{\le 2k}\) can be written as \(x=y-z\) with \(y,z \in R_{\le k}\) and \(y\neq z\). Since \(T\) is injective on \(R_{\le k}\), it follows that \(T(y)\neq T(z)\), and hence \(T(x)\neq0\).

    Since \(R_{\le 2k}\) is closed, the set
    \[
        S_{2k}
        \triangleq
        R_{\le 2k}\cap \{x:\|x\|=1\}
    \]
    is compact. Consider
    \[
        \kappa
        \triangleq
        \min_{x\in S_{2k}} \|T(x)\|^2.
    \]
    The minimum is attained by compactness and continuity, and \(\kappa>0\) because \(T(x)\neq0\) for every \(x\in S_{2k}\). Therefore, for any nonzero \(x\in R_{\le 2k}\),
    \[
        \left\|T\!\left(\frac{x}{\|x\|}\right)\right\|^2
        \ge \kappa,
    \]
    and hence
    \begin{equation}\label{eqn:rsc-def}
        \|T(x)\|^2 \ge \kappa\|x\|^2,
        \qquad x\in R_{\le 2k},
    \end{equation}
    which corresponds to restricted strong convexity of the quadratic loss \(x \mapsto \|T(x)\|^2\) on \(R_{\le 2k}\); see \cite{negahban11}.
    In \cite{negahban11}, this condition is stated using the Frobenius norm on \(\R^{m\times n}\) and the Euclidean norm on \(\R^\ell\). Thus injectivity of \(T\) on \(R_{\le k}\) implies restricted strong convexity on \(R_{\le 2k}\).

    Now assume restricted strong convexity in the form \eqref{eqn:rsc-def} for some $\kappa >0$. For \(x \in R_{\le 2k}\) one has
    \[
        \|T(x)\|^2 \;\ge\; \kappa \|x\|^2.
    \]
    Also, since \(T\) is linear on a finite-dimensional normed space, its operator
    norm is finite, and
    \[
        \|T(x)\|^2
        \le
        \|T\|^2\,\|x\|^2.
    \]
    Define the rescaled operator
    \[
        \widetilde T
        =
        \sqrt{\frac{2}{\kappa+\|T\|^2}}\,T,
        \qquad
        \delta
        =
        \frac{\|T\|^2-\kappa}{\|T\|^2+\kappa}
        \in[0,1).
    \]
    Then, for every \(x\in R_{\le 2k}\),
    \begin{equation}\label{eqn:rip-def}
        (1-\delta)\|x\|^2
        \le
        \|\widetilde T(x)\|^2
        \le
        (1+\delta)\|x\|^2.
    \end{equation}
    Thus the restricted strong convexity lower bound on \(R_{\le 2k}\) implies, after rescaling, the restricted isometry property on rank-\(\le 2k\) matrices.

    Now assume that \(T\) satisfies the restricted isometry property on rank-\(\le 2k\) matrices with some \(\delta\in[0,1)\).
    For \(x,y\in R_{\le k}\), the difference \(x-y\) has rank at most \(2k\), so
    \[
        (1-\delta)\|x-y\|^2
        \le
        \|T(x-y)\|^2
        \le
        (1+\delta)\|x-y\|^2.
    \]
    If \(x,y \in R_{\le k}\) with \(x\neq y\), then \(\|x-y\|>0\). Hence \(\|T(x-y)\|>0\), so \(T(x)\neq T(y)\). Thus the restricted isometry property on  rank-\(\le 2k\) matrices implies injectivity of \(T\) on \(R_{\le k}\).

    Combining the preceding implications shows that injectivity of \(T\) on \(R_{\le k}\), restricted strong convexity on \(R_{\le 2k}\), and, up to rescaling, the restricted isometry property on rank-\(\le 2k\) matrices are equivalent.

    \section{Experimental Details} \label{appx:experimental}
    \subsection{Implementation and Runtime Details}
In the experiments of Section~\ref{sec:exp-real-world}, 250 HEPMASS samples were used for training and 250 for validation, while 1000 kin8nm samples were used for training and 300 for validation. For HEPMASS, a fixed test set of 20,000 samples was used across all 20 runs; for kin8nm, the remaining samples formed the test set in each split. These splits were chosen to balance computational cost and stability of model selection, providing sufficient data for validation-based tuning while keeping the overall runtime manageable.

    For SME, partitioning was performed on samples whose coordinates had been rounded to the nearest multiple of \(0.125\) for HEPMASS and \(0.25\) for kin8nm. This allows the use of weighted empirical measures rather than full empirical measures, significantly reducing the cost of Wasserstein distance computations. The rounding was used only for partitioning; the KDEs were fit to the original, unrounded samples using the resulting component labels. In preliminary testing, this rounding step did not appear to noticeably affect density-estimation performance.

    This rounding and the incremental update strategy described in Section~\ref{sec:application-algorithm}, together with parallelization across 9 CPU cores (Apple M2), substantially reduced computation time relative to a naive implementation of the SME objective. For HEPMASS, after the initial objective evaluation has cached the block-level contributions to the projected measures, evaluating the objective for a single candidate label change—that is, the loss resulting from moving one sample to a different component—takes approximately 6--7 ms. This evaluation includes computing and averaging the Wasserstein distances across all projections (42 graph edges with 4 projections per edge). With \(k=4\) components, the three alternative labels for one HEPMASS sample are evaluated separately and take approximately 21 ms in total, consistent with the 6--7 ms cost per label. For kin8nm, evaluating one alternative label for one sample takes approximately 3--4 ms.

    For HEPMASS, each objective evaluation comprises 168 independent computations, each consisting of a projection and its associated Wasserstein-distance evaluation. The experiments reported here used 9 CPU cores, leaving room for finer-grained parallelization and further speedups.

    \subsubsection{MAF Implementation Details}
    For both datasets, MAF was implemented using \texttt{zuko.flows.MAF} with five autoregressive transformations and random feature permutations enabled. The hidden-layer widths were \((64,64)\) for HEPMASS and \((32,32)\) for kin8nm. Models were trained by minimizing the mean negative log-likelihood using Adam with learning rate \(10^{-3}\) and weight decay \(10^{-6}\). HEPMASS used full-batch updates, while kin8nm used mini-batches of size \(50\).

    Training continued until the validation negative log-likelihood failed to improve for \(10\) consecutive epochs for HEPMASS or \(30\) consecutive epochs for kin8nm. The parameter state attaining the lowest validation negative log-likelihood was restored before evaluating the test log-likelihood. All remaining Zuko and Adam settings were left at their default values.

    \subsection{Additional Experimental Results}

    This subsection presents additional experimental results examining the behavior of the SME objective during optimization and its relationship to downstream density-estimation performance. The goal is to better understand the objective and the optimization process rather than provide an additional benchmark comparison.

    \begin{figure}[!t]
        \centering
        \subfloat[\small Single run with 20,000 training samples and \(k=80\).
        \label{fig:sme-diagnostic}]
        {
            \includegraphics[width=0.48\linewidth]
            {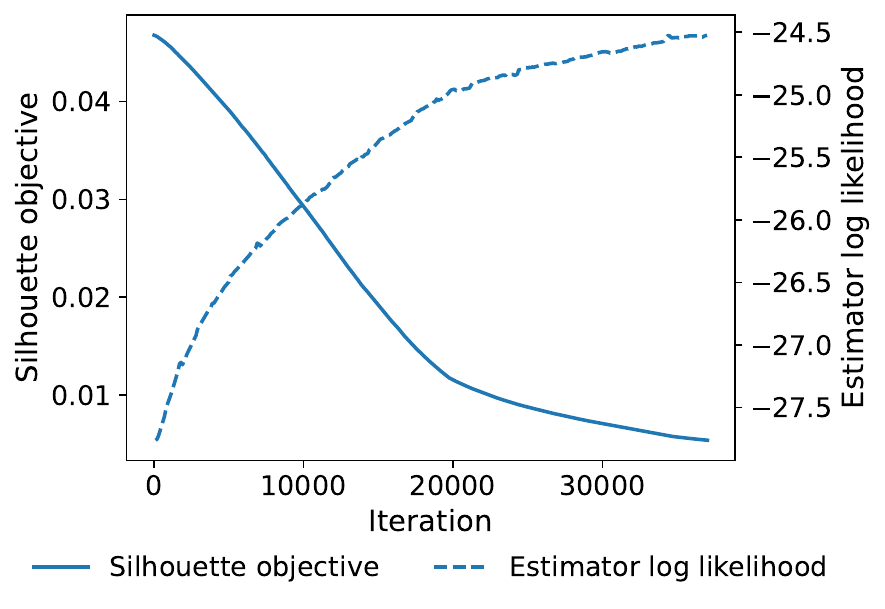}
        }
        \hfill
        \subfloat[\small Mean and standard deviation over 20 runs (\(500\) training samples, \(k=5\)).
        \label{fig:sme-diagnostic-mean-std}]
        {
            \includegraphics[width=0.48\linewidth]
            {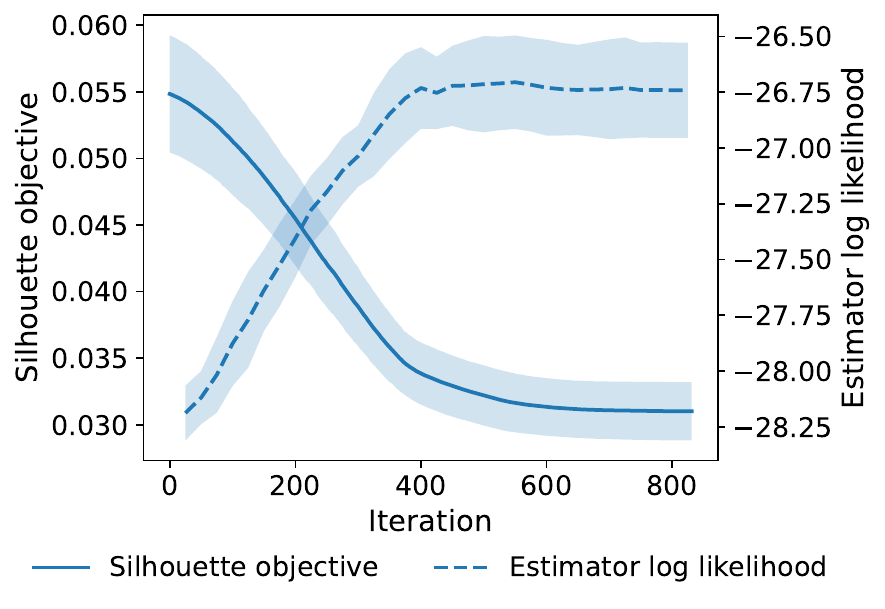}
        }
        \caption{Silhouette objective and estimator log-likelihood during SME optimization on HEPMASS.}
    \end{figure}

    The first experiment used HEPMASS with 20,000 training samples and \(k=80\) components, with the SME objective optimized for approximately 37,000 steps. Every 100 optimization steps, the corresponding mixture-of-Naive-Bayes KDE estimator was evaluated on a held-out test set of size 10,000. The graph structure, projection parameters, KDE fitting, and rounding procedure were the same as in the HEPMASS experiments of Section~\ref{sec:exp-real-world}. Bandwidths were selected by \(5\)-fold cross-validation over a logarithmic grid of \(40\) values ranging from \(10^{-3}\) to \(1\).

    The full experiment required approximately 90 hours, roughly one half of which was spent periodically fitting the mixture-of-Naive-Bayes KDE estimators and evaluating them on the test set. Evaluating all 79 alternative labels for a single sample took approximately 4.5 seconds. Figure~\ref{fig:sme-diagnostic} shows the silhouette objective and test log-likelihood throughout optimization.

    As the silhouette objective decreases during optimization, test log-likelihood improves nearly monotonically. This provides empirical evidence that minimizing the silhouette objective tends to produce partitions that are increasingly useful for constructing density estimators of the form \eqref{eqn:lr-kde}.

    Figure~\ref{fig:sme-diagnostic-mean-std} shows the mean and standard deviation of the corresponding trajectories over 20 independent HEPMASS runs with 500 training samples and \(k=5\) components. Each run continued until no single sample reassignment decreased the objective. Because the runs converged after different numbers of steps, each trajectory was extended after convergence by holding its last recorded value constant.

    As in the larger experiment, estimator log-likelihood improves as the silhouette objective decreases. The improvement then plateaus, with a slight decline late in optimization even as the silhouette objective continues to decrease. The cause of this behavior is unclear.

    \bibliographystyle{unsrtnat}
    \bibliography{rvdm}

\begin{thebibliography}{90}
\providecommand{\natexlab}[1]{#1}
\providecommand{\url}[1]{\texttt{#1}}
\expandafter\ifx\csname urlstyle\endcsname\relax
  \providecommand{\doi}[1]{doi: #1}\else
  \providecommand{\doi}{doi: \begingroup \urlstyle{rm}\Url}\fi

\bibitem[Candes and Tao(2005)]{candes05}
E.~J. Candes and T.~Tao.
\newblock Decoding by linear programming.
\newblock \emph{IEEE Trans. Inf. Theor.}, 51\penalty0 (12):\penalty0 4203–4215, December 2005.
\newblock ISSN 0018-9448.
\newblock \doi{10.1109/TIT.2005.858979}.
\newblock URL \url{https://doi.org/10.1109/TIT.2005.858979}.

\bibitem[Recht et~al.(2010)Recht, Fazel, and Parrilo]{recht10}
Benjamin Recht, Maryam Fazel, and Pablo~A. Parrilo.
\newblock Guaranteed minimum-rank solutions of linear matrix equations via nuclear norm minimization.
\newblock \emph{SIAM Rev.}, 52\penalty0 (3):\penalty0 471--501, August 2010.
\newblock ISSN 0036-1445.
\newblock \doi{10.1137/070697835}.
\newblock URL \url{http://dx.doi.org/10.1137/070697835}.

\bibitem[Negahban and Wainwright(2011)]{negahban11}
Sahand Negahban and Martin~J. Wainwright.
\newblock Estimation of (near) low-rank matrices with noise and high-dimensional scaling.
\newblock \emph{Ann. Statist.}, 39\penalty0 (2):\penalty0 1069--1097, 04 2011.
\newblock \doi{10.1214/10-AOS850}.
\newblock URL \url{https://doi.org/10.1214/10-AOS850}.

\bibitem[Bonnotte(2013)]{bonnotte13}
Nicolas Bonnotte.
\newblock \emph{{Unidimensional and Evolution Methods for Optimal Transportation}}.
\newblock Theses, {Universit{\'e} Paris Sud - Paris XI ; Scuola normale superiore (Pise, Italie)}, December 2013.
\newblock URL \url{https://theses.hal.science/tel-00946781}.

\bibitem[Bonneel et~al.(2015)Bonneel, Rabin, Peyr{\'e}, and Pfister]{bonneel15}
Nicolas Bonneel, Julien Rabin, Gabriel Peyr{\'e}, and Hanspeter Pfister.
\newblock Sliced and {Radon} {Wasserstein} barycenters of measures.
\newblock \emph{Journal of Mathematical Imaging and Vision}, 51\penalty0 (1):\penalty0 22--45, Jan 2015.
\newblock ISSN 1573-7683.
\newblock \doi{10.1007/s10851-014-0506-3}.
\newblock URL \url{https://doi.org/10.1007/s10851-014-0506-3}.

\bibitem[Nadjahi et~al.(2020)Nadjahi, Durmus, Chizat, Kolouri, Shahrampour, and Simsekli]{nadjahi20}
Kimia Nadjahi, Alain Durmus, L\'{e}na\"{i}c Chizat, Soheil Kolouri, Shahin Shahrampour, and Umut Simsekli.
\newblock Statistical and topological properties of sliced probability divergences.
\newblock In H.~Larochelle, M.~Ranzato, R.~Hadsell, M.F. Balcan, and H.~Lin, editors, \emph{Advances in Neural Information Processing Systems}, volume~33, pages 20802--20812. Curran Associates, Inc., 2020.
\newblock URL \url{https://proceedings.neurips.cc/paper_files/paper/2020/file/eefc9e10ebdc4a2333b42b2dbb8f27b6-Paper.pdf}.

\bibitem[Nguyen et~al.(2021)Nguyen, Ho, Pham, and Bui]{nguyen21distributional}
Khai Nguyen, Nhat Ho, Tung Pham, and Hung Bui.
\newblock Distributional sliced-{Wasserstein} and applications to generative modeling.
\newblock In \emph{International Conference on Learning Representations}, 2021.
\newblock URL \url{https://openreview.net/forum?id=QYjO70ACDK}.

\bibitem[Scott et~al.(2013)Scott, Blanchard, and Handy]{scott13}
Clayton Scott, Gilles Blanchard, and Gregory Handy.
\newblock Classification with asymmetric label noise: Consistency and maximal denoising.
\newblock In Shai Shalev-Shwartz and Ingo Steinwart, editors, \emph{Proceedings of the 26th Annual Conference on Learning Theory}, volume~30 of \emph{Proceedings of Machine Learning Research}, pages 489--511, Princeton, NJ, USA, 12--14 Jun 2013. PMLR.
\newblock URL \url{https://proceedings.mlr.press/v30/Scott13.html}.

\bibitem[Blanchard and Scott(2014)]{blanchard14}
Gilles Blanchard and Clayton Scott.
\newblock Decontamination of mutually contaminated models.
\newblock In \emph{Proceedings of the Seventeenth International Conference on Artificial Intelligence and Statistics, {AISTATS} 2014, Reykjavik, Iceland, April 22-25, 2014}, pages 1--9, 2014.
\newblock URL \url{http://jmlr.org/proceedings/papers/v33/blanchard14.html}.

\bibitem[Arora et~al.(2012{\natexlab{a}})Arora, Ge, Kannan, and Moitra]{arora12}
Sanjeev Arora, Rong Ge, Ravindran Kannan, and Ankur Moitra.
\newblock Computing a nonnegative matrix factorization -- provably.
\newblock In \emph{Proceedings of the Forty-fourth Annual ACM Symposium on Theory of Computing}, STOC '12, pages 145--162, New York, NY, USA, 2012{\natexlab{a}}. ACM.
\newblock ISBN 978-1-4503-1245-5.
\newblock \doi{10.1145/2213977.2213994}.
\newblock URL \url{http://doi.acm.org/10.1145/2213977.2213994}.

\bibitem[Hall and Zhou(2003)]{hall03}
Peter Hall and Xiao-Hua Zhou.
\newblock {Nonparametric estimation of component distributions in a multivariate mixture}.
\newblock \emph{The Annals of Statistics}, 31\penalty0 (1):\penalty0 201 -- 224, 2003.
\newblock \doi{10.1214/aos/1046294462}.
\newblock URL \url{https://doi.org/10.1214/aos/1046294462}.

\bibitem[Hall et~al.(2005)Hall, Neeman, Pakyari, and Elmore]{hall05}
Peter Hall, Amnon Neeman, Reza Pakyari, and Ryan Elmore.
\newblock {Nonparametric inference in multivariate mixtures}.
\newblock \emph{Biometrika}, 92\penalty0 (3):\penalty0 667--678, 09 2005.
\newblock ISSN 0006-3444.
\newblock \doi{10.1093/biomet/92.3.667}.
\newblock URL \url{https://doi.org/10.1093/biomet/92.3.667}.

\bibitem[Vandermeulen and Ledent(2021)]{vandermeulen21}
Robert~A Vandermeulen and Antoine Ledent.
\newblock Beyond smoothness: Incorporating low-rank analysis into nonparametric density estimation.
\newblock In M.~Ranzato, A.~Beygelzimer, Y.~Dauphin, P.S. Liang, and J.~Wortman Vaughan, editors, \emph{Advances in Neural Information Processing Systems}, volume~34, pages 12180--12193. Curran Associates, Inc., 2021.
\newblock URL \url{https://proceedings.neurips.cc/paper/2021/file/6591d327f6f731e589b0e869adadf940-Paper.pdf}.

\bibitem[Kwon and Mbakop(2021)]{kwon21}
Caleb Kwon and Eric Mbakop.
\newblock {Estimation of the number of components of nonparametric multivariate finite mixture models}.
\newblock \emph{The Annals of Statistics}, 49\penalty0 (4):\penalty0 2178 -- 2205, 2021.
\newblock \doi{10.1214/20-AOS2032}.
\newblock URL \url{https://doi.org/10.1214/20-AOS2032}.

\bibitem[Amiridi et~al.(2022)Amiridi, Kargas, and Sidiropoulos]{amiridi22I}
Magda Amiridi, Nikos Kargas, and Nicholas~D. Sidiropoulos.
\newblock Low-rank characteristic tensor density estimation part {I}: Foundations.
\newblock \emph{IEEE Transactions on Signal Processing}, 70:\penalty0 2654--2668, 2022.
\newblock \doi{10.1109/TSP.2022.3175608}.

\bibitem[Chhor et~al.(2025)Chhor, Klopp, and Tsybakov]{chhor25}
Julien Chhor, Olga Klopp, and Alexandre~B. Tsybakov.
\newblock Generalized multi-view model: Adaptive density estimation under low-rank constraints.
\newblock \emph{Journal of Machine Learning Research}, 26\penalty0 (236):\penalty0 1--52, 2025.
\newblock URL \url{http://jmlr.org/papers/v26/24-1729.html}.

\bibitem[Lyu and Yang(2025)]{lyu25}
Yichen Lyu and Pengkun Yang.
\newblock Identifiability and estimation in high-dimensional nonparametric latent structure models.
\newblock In Nika Haghtalab and Ankur Moitra, editors, \emph{Proceedings of Thirty Eighth Conference on Learning Theory}, volume 291 of \emph{Proceedings of Machine Learning Research}, pages 3879--3880. PMLR, 30 Jun--04 Jul 2025.
\newblock URL \url{https://proceedings.mlr.press/v291/lyu25a.html}.

\bibitem[Benaglia et~al.(2009)Benaglia, Chauveau, and Hunter]{benaglia09}
Tatiana Benaglia, Didier Chauveau, and David~R. Hunter.
\newblock An {EM}-like algorithm for semi- and nonparametric estimation in multivariate mixtures.
\newblock \emph{Journal of Computational and Graphical Statistics}, 18\penalty0 (2):\penalty0 505--526, 2009.
\newblock \doi{10.1198/jcgs.2009.07175}.
\newblock URL \url{https://doi.org/10.1198/jcgs.2009.07175}.

\bibitem[Vandermeulen(2025)]{vandermeulen25pmode}
Robert~A. Vandermeulen.
\newblock Pmode: Theoretically grounded and modular mixture modeling, 2025.
\newblock URL \url{https://arxiv.org/abs/2508.21396}.

\bibitem[Fabian et~al.(2001)Fabian, Habala, H{\'a}jek, Santaluc{\'i}a, Pelant, and Zizler]{fabian01}
Mari{\'a}n Fabian, Petr Habala, Petr H{\'a}jek, Vicente~Montesinos Santaluc{\'i}a, Jan Pelant, and V{\'a}clav Zizler.
\newblock \emph{Functional Analysis and Infinite-Dimensional Geometry}.
\newblock Springer New York, New York, NY, 2001.
\newblock ISBN 978-1-4757-3480-5.
\newblock \doi{10.1007/978-1-4757-3480-5}.
\newblock URL \url{https://doi.org/10.1007/978-1-4757-3480-5}.

\bibitem[Petitjean(2025)]{petitjean25}
Colin Petitjean.
\newblock The set of elementary tensors is weakly closed in projective tensor products.
\newblock \emph{Bulletin of the Australian Mathematical Society}, 111\penalty0 (2):\penalty0 336--343, 2025.
\newblock \doi{10.1017/S0004972724000376}.

\bibitem[Donoho and Stodden(2004)]{donoho04}
David Donoho and Victoria Stodden.
\newblock When does non-negative matrix factorization give a correct decomposition into parts?
\newblock In S.~Thrun, L.~K. Saul, and B.~Sch\"{o}lkopf, editors, \emph{Advances in Neural Information Processing Systems 16}, pages 1141--1148. MIT Press, 2004.

\bibitem[Arora et~al.(2012{\natexlab{b}})Arora, Ge, and Moitra]{arora12b}
Sanjeev Arora, Rong Ge, and Ankur Moitra.
\newblock Learning topic models -- going beyond {SVD}.
\newblock In \emph{Proceedings of the 2012 IEEE 53rd Annual Symposium on Foundations of Computer Science}, FOCS '12, pages 1--10, Washington, DC, USA, 2012{\natexlab{b}}. IEEE Computer Society.
\newblock ISBN 978-0-7695-4874-6.
\newblock \doi{10.1109/FOCS.2012.49}.
\newblock URL \url{http://dx.doi.org/10.1109/FOCS.2012.49}.

\bibitem[Kargas et~al.(2018)Kargas, Sidiropoulos, and Fu]{kargas18}
Nikos Kargas, Nicholas~D. Sidiropoulos, and Xiao Fu.
\newblock Tensors, learning, and ``{Kolmogorov} extension'' for finite-alphabet random vectors.
\newblock \emph{IEEE Transactions on Signal Processing}, 66\penalty0 (18):\penalty0 4854--4868, 2018.
\newblock \doi{10.1109/TSP.2018.2862383}.

\bibitem[Allman et~al.(2009)Allman, Matias, and Rhodes]{allman09}
Elizabeth~S. Allman, Catherine Matias, and John~A. Rhodes.
\newblock Identifiability of parameters in latent structure models with many observed variables.
\newblock \emph{Ann. Statist.}, 37\penalty0 (6A):\penalty0 3099--3132, 12 2009.
\newblock \doi{10.1214/09-AOS689}.
\newblock URL \url{http://dx.doi.org/10.1214/09-AOS689}.

\bibitem[Yakowitz and Spragins(1968)]{yakowitz68}
Sidney~J. Yakowitz and John~D. Spragins.
\newblock On the identifiability of finite mixtures.
\newblock \emph{Ann. Math. Statist.}, 39\penalty0 (1):\penalty0 209--214, 02 1968.
\newblock \doi{10.1214/aoms/1177698520}.
\newblock URL \url{http://dx.doi.org/10.1214/aoms/1177698520}.

\bibitem[Anandkumar et~al.(2014)Anandkumar, Ge, Hsu, Kakade, and Telgarsky]{anandkumar14}
Animashree Anandkumar, Rong Ge, Daniel Hsu, Sham~M. Kakade, and Matus Telgarsky.
\newblock Tensor decompositions for learning latent variable models.
\newblock \emph{Journal of Machine Learning Research}, 15:\penalty0 2773--2832, 2014.
\newblock URL \url{http://jmlr.org/papers/v15/anandkumar14b.html}.

\bibitem[Tang et~al.(2026)Tang, Chhor, Klopp, Tsybakov, and Zhang]{tang26}
Runshi Tang, Julien Chhor, Olga Klopp, Alexandre~B. Tsybakov, and Anru~R. Zhang.
\newblock Optimal estimation of discrete multiview distributions under heteroskedastic multinomial sampling, 2026.
\newblock URL \url{https://arxiv.org/abs/2605.24858}.

\bibitem[Song and Dai(2013)]{song13}
Le~Song and Bo~Dai.
\newblock Robust low rank kernel embeddings of multivariate distributions.
\newblock In C.~J.~C. Burges, L.~Bottou, M.~Welling, Z.~Ghahramani, and K.~Q. Weinberger, editors, \emph{Advances in Neural Information Processing Systems}, volume~26. Curran Associates, Inc., 2013.
\newblock URL \url{https://proceedings.neurips.cc/paper/2013/file/49b8b4f95f02e055801da3b4f58e28b7-Paper.pdf}.

\bibitem[Song et~al.(2014)Song, Anandkumar, Dai, and Xie]{song14}
Le~Song, Animashree Anandkumar, Bo~Dai, and Bo~Xie.
\newblock Nonparametric estimation of multi-view latent variable models.
\newblock In Eric~P. Xing and Tony Jebara, editors, \emph{Proceedings of the 31st International Conference on Machine Learning}, volume~32 of \emph{Proceedings of Machine Learning Research}, pages 640--648, Beijing, China, 22--24 Jun 2014. PMLR.
\newblock URL \url{https://proceedings.mlr.press/v32/songa14.html}.

\bibitem[Bonhomme et~al.(2016)Bonhomme, Jochmans, and Robin]{bonhomme16}
St{\'e}phane Bonhomme, Koen Jochmans, and Jean-Marc Robin.
\newblock {Estimating multivariate latent-structure models}.
\newblock \emph{The Annals of Statistics}, 44\penalty0 (2):\penalty0 540 -- 563, 2016.
\newblock \doi{10.1214/15-AOS1376}.
\newblock URL \url{https://doi.org/10.1214/15-AOS1376}.

\bibitem[Kargas and Sidiropoulos(2019)]{kargas19}
Nikos Kargas and Nicholas~D. Sidiropoulos.
\newblock Learning mixtures of smooth product distributions: Identifiability and algorithm.
\newblock In Kamalika Chaudhuri and Masashi Sugiyama, editors, \emph{Proceedings of the Twenty-Second International Conference on Artificial Intelligence and Statistics}, volume~89 of \emph{Proceedings of Machine Learning Research}, pages 388--396. PMLR, 16--18 Apr 2019.
\newblock URL \url{https://proceedings.mlr.press/v89/kargas19a.html}.

\bibitem[{Tahmasebi} et~al.(2018){Tahmasebi}, {Abolfazl Motahari}, and {Maddah-Ali}]{tahmasebi18}
Behrooz {Tahmasebi}, Seyed {Abolfazl Motahari}, and Mohammad~Ali {Maddah-Ali}.
\newblock {On the Identifiability of Finite Mixtures of Finite Product Measures}.
\newblock \emph{arXiv e-prints}, art. arXiv:1807.05444, July 2018.
\newblock \doi{10.48550/arXiv.1807.05444}.

\bibitem[Stone(1980)]{stone80}
Charles~J. Stone.
\newblock {Optimal Rates of Convergence for Nonparametric Estimators}.
\newblock \emph{The Annals of Statistics}, 8\penalty0 (6):\penalty0 1348 -- 1360, 1980.
\newblock \doi{10.1214/aos/1176345206}.
\newblock URL \url{https://doi.org/10.1214/aos/1176345206}.

\bibitem[Stone(1982)]{stone82}
Charles~J. Stone.
\newblock {Optimal Global Rates of Convergence for Nonparametric Regression}.
\newblock \emph{The Annals of Statistics}, 10\penalty0 (4):\penalty0 1040 -- 1053, 1982.
\newblock \doi{10.1214/aos/1176345969}.
\newblock URL \url{https://doi.org/10.1214/aos/1176345969}.

\bibitem[Tsybakov(2008)]{tsybakov08}
Alexandre~B. Tsybakov.
\newblock \emph{Introduction to Nonparametric Estimation}.
\newblock Springer Publishing Company, Incorporated, 1st edition, 2008.

\bibitem[Vandermeulen(2023)]{vandermeulen23}
Robert~A. Vandermeulen.
\newblock Sample complexity using infinite multiview models, 2023.
\newblock URL \url{https://arxiv.org/abs/2302.04292}.

\bibitem[Rényi(1952)]{renyi52}
A.~Rényi.
\newblock On projections of probability distributions.
\newblock \emph{Acta Mathematica Academiae Scientiarum Hungaricae}, 3\penalty0 (3):\penalty0 131--142, September 1952.
\newblock ISSN 0001-5954, 1588-2632.
\newblock \doi{10.1007/BF02022515}.
\newblock URL \url{http://link.springer.com/10.1007/BF02022515}.

\bibitem[Gilbert(1955)]{gilbert55}
Walter~M. Gilbert.
\newblock Projections of probability distributions.
\newblock \emph{Acta Mathematica Academiae Scientiarum Hungarica}, 6\penalty0 (1):\penalty0 195--198, March 1955.
\newblock ISSN 1588-2632.
\newblock \doi{10.1007/BF02021275}.
\newblock URL \url{https://doi.org/10.1007/BF02021275}.

\bibitem[Cuesta-Albertos et~al.(2007)Cuesta-Albertos, Fraiman, and Ransford]{cuesta-albertos07}
Juan~Antonio Cuesta-Albertos, Ricardo Fraiman, and Thomas Ransford.
\newblock A {Sharp} {Form} of the {Cram\'er--Wold} {Theorem}.
\newblock \emph{Journal of Theoretical Probability}, 20\penalty0 (2):\penalty0 201--209, June 2007.
\newblock ISSN 1572-9230.
\newblock \doi{10.1007/s10959-007-0060-7}.
\newblock URL \url{https://doi.org/10.1007/s10959-007-0060-7}.

\bibitem[Teicher(1963)]{teicher63}
Henry Teicher.
\newblock Identifiability of finite mixtures.
\newblock \emph{Ann. Math. Statist.}, 34\penalty0 (4):\penalty0 1265--1269, 12 1963.
\newblock \doi{10.1214/aoms/1177703862}.
\newblock URL \url{http://dx.doi.org/10.1214/aoms/1177703862}.

\bibitem[Rabani et~al.(2014)Rabani, Schulman, and Swamy]{rabani14}
Yuval Rabani, Leonard~J. Schulman, and Chaitanya Swamy.
\newblock Learning mixtures of arbitrary distributions over large discrete domains.
\newblock In \emph{Proceedings of the 5th Conference on Innovations in Theoretical Computer Science}, ITCS '14, pages 207--224, New York, NY, USA, 2014. ACM.
\newblock ISBN 978-1-4503-2698-8.
\newblock \doi{10.1145/2554797.2554818}.
\newblock URL \url{http://doi.acm.org/10.1145/2554797.2554818}.

\bibitem[Vandermeulen and Scott(2019)]{vandermeulen19}
Robert~A. Vandermeulen and Clayton~D. Scott.
\newblock An operator theoretic approach to nonparametric mixture models.
\newblock \emph{Ann. Statist.}, 47\penalty0 (5):\penalty0 2704--2733, 10 2019.
\newblock \doi{10.1214/18-AOS1762}.
\newblock URL \url{https://doi.org/10.1214/18-AOS1762}.

\bibitem[Ritchie et~al.(2020)Ritchie, Vandermeulen, and Scott]{ritchie21}
Alexander Ritchie, Robert~A Vandermeulen, and Clayton Scott.
\newblock Consistent estimation of identifiable nonparametric mixture models from grouped observations.
\newblock In H.~Larochelle, M.~Ranzato, R.~Hadsell, M.~F. Balcan, and H.~Lin, editors, \emph{Advances in Neural Information Processing Systems}, volume~33, pages 11676--11686. Curran Associates, Inc., 2020.
\newblock URL \url{https://proceedings.neurips.cc/paper/2020/file/866d90e0921ac7b024b47d672445a086-Paper.pdf}.

\bibitem[Wei and Nguyen(2022)]{wei22}
Yun Wei and XuanLong Nguyen.
\newblock {Convergence of {de Finetti's} mixing measure in latent structure models for observed exchangeable sequences}.
\newblock \emph{The Annals of Statistics}, 50\penalty0 (4):\penalty0 1859 -- 1889, 2022.
\newblock \doi{10.1214/21-AOS2120}.
\newblock URL \url{https://doi.org/10.1214/21-AOS2120}.

\bibitem[Vandermeulen and Saitenmacher(2024)]{vandermeulen24grouped}
Robert~A. Vandermeulen and René Saitenmacher.
\newblock Generalized identifiability bounds for mixture models with grouped samples.
\newblock \emph{IEEE Transactions on Information Theory}, 70\penalty0 (4):\penalty0 2746--2758, 2024.
\newblock \doi{10.1109/TIT.2024.3367433}.

\bibitem[Blei et~al.(2003)Blei, Ng, and Jordan]{LDA}
David~M. Blei, Andrew~Y. Ng, and Michael~I. Jordan.
\newblock Latent {Dirichlet} allocation.
\newblock \emph{J. Mach. Learn. Res.}, 3:\penalty0 993--1022, March 2003.
\newblock ISSN 1532-4435.
\newblock URL \url{http://dl.acm.org/citation.cfm?id=944919.944937}.

\bibitem[Do et~al.(2025)Do, Chakraborty, Terhorst, and Nguyen]{do25}
Dat Do, Sunrit Chakraborty, Jonathan Terhorst, and XuanLong Nguyen.
\newblock {Dirichlet} moment tensors and the correspondence between admixture and mixture of product models, 2025.
\newblock URL \url{https://arxiv.org/abs/2509.25441}.

\bibitem[Rauhut et~al.(2017)Rauhut, Schneider, and Željka Stojanac]{rauhut17}
Holger Rauhut, Reinhold Schneider, and Željka Stojanac.
\newblock Low rank tensor recovery via iterative hard thresholding.
\newblock \emph{Linear Algebra and its Applications}, 523:\penalty0 220--262, 2017.
\newblock ISSN 0024-3795.
\newblock \doi{https://doi.org/10.1016/j.laa.2017.02.028}.
\newblock URL \url{https://www.sciencedirect.com/science/article/pii/S0024379517301295}.

\bibitem[{Donoho}(2006)]{donoho06}
D.~L. {Donoho}.
\newblock Compressed sensing.
\newblock \emph{IEEE Transactions on Information Theory}, 52\penalty0 (4):\penalty0 1289--1306, April 2006.
\newblock ISSN 0018-9448.
\newblock \doi{10.1109/TIT.2006.871582}.

\bibitem[Candes and Tao(2008)]{candes08}
Emmanuel~J Candes and Terence Tao.
\newblock Reflections on compressed sensing.
\newblock \emph{IEEE Information Theory Society Newsletter}, 58\penalty0 (4):\penalty0 20--23, 2008.

\bibitem[Baraniuk et~al.(2008)Baraniuk, Davenport, DeVore, and Wakin]{baraniuk08}
Richard Baraniuk, Mark Davenport, Ronald DeVore, and Michael Wakin.
\newblock A simple proof of the restricted isometry property for random matrices.
\newblock \emph{Constructive Approximation}, 28\penalty0 (3):\penalty0 253--263, Dec 2008.
\newblock ISSN 1432-0940.
\newblock \doi{10.1007/s00365-007-9003-x}.
\newblock URL \url{https://doi.org/10.1007/s00365-007-9003-x}.

\bibitem[Gilbert and Indyk(2010)]{gilbert10}
Anna Gilbert and Piotr Indyk.
\newblock Sparse recovery using sparse matrices.
\newblock \emph{Proceedings of the IEEE}, 98\penalty0 (6):\penalty0 937--947, 2010.
\newblock \doi{10.1109/JPROC.2010.2045092}.

\bibitem[Bandeira et~al.(2013)Bandeira, Dobriban, Mixon, and Sawin]{banderia13}
Afonso~S. Bandeira, Edgar Dobriban, Dustin~G. Mixon, and William~F. Sawin.
\newblock Certifying the restricted isometry property is hard.
\newblock \emph{IEEE Transactions on Information Theory}, 59\penalty0 (6):\penalty0 3448--3450, 2013.
\newblock \doi{10.1109/TIT.2013.2248414}.

\bibitem[Liu et~al.(2011)Liu, Xu, Gu, Gupta, Lafferty, and Wasserman]{liu11}
Han Liu, Min Xu, Haijie Gu, Anupam Gupta, John Lafferty, and Larry Wasserman.
\newblock Forest density estimation.
\newblock \emph{J. Mach. Learn. Res.}, 12:\penalty0 907--951, July 2011.
\newblock ISSN 1532-4435.
\newblock URL \url{http://dl.acm.org/citation.cfm?id=1953048.2021032}.

\bibitem[Györfi et~al.(2023)Györfi, Kontorovich, and Weiss]{gyorfi23}
László Györfi, Aryeh Kontorovich, and Roi Weiss.
\newblock Tree density estimation.
\newblock \emph{IEEE Transactions on Information Theory}, 69\penalty0 (2):\penalty0 1168--1176, 2023.
\newblock \doi{10.1109/TIT.2022.3214888}.

\bibitem[Vandermeulen et~al.(2024)Vandermeulen, Tai, and Aragam]{vandermeulen24mrf}
Robert~A. Vandermeulen, Wai~Ming Tai, and Bryon Aragam.
\newblock Breaking the curse of dimensionality in structured density estimation.
\newblock In A.~Globerson, L.~Mackey, D.~Belgrave, A.~Fan, U.~Paquet, J.~Tomczak, and C.~Zhang, editors, \emph{Advances in Neural Information Processing Systems}, volume~37, pages 65852--65883. Curran Associates, Inc., 2024.
\newblock \doi{10.52202/079017-2105}.

\bibitem[Vandermeulen et~al.(2025)Vandermeulen, Tai, and Aragam]{vandermeulen25mrf}
Robert~A. Vandermeulen, Wai~Ming Tai, and Bryon Aragam.
\newblock Dimension-independent rates for structured neural density estimation.
\newblock In Aarti Singh, Maryam Fazel, Daniel Hsu, Simon Lacoste-Julien, Felix Berkenkamp, Tegan Maharaj, Kiri Wagstaff, and Jerry Zhu, editors, \emph{Proceedings of the 42nd International Conference on Machine Learning}, volume 267 of \emph{Proceedings of Machine Learning Research}, pages 60857--60879. PMLR, 13--19 Jul 2025.
\newblock URL \url{https://proceedings.mlr.press/v267/vandermeulen25a.html}.

\bibitem[Fan et~al.(2025)Fan, Gu, and Li]{fan25}
Jianqing Fan, Yihong Gu, and Ximing Li.
\newblock Optimal estimation of a factorizable density using diffusion models with relu neural networks, 2025.
\newblock URL \url{https://arxiv.org/abs/2510.03994}.

\bibitem[Kwon et~al.(2026)Kwon, Kim, Ohn, and Chae]{kwon26}
Hyeok~Kyu Kwon, Dongha Kim, Ilsang Ohn, and Minwoo Chae.
\newblock Nonparametric estimation of a factorizable density using diffusion models.
\newblock \emph{Journal of Machine Learning Research}, 27\penalty0 (22):\penalty0 1--125, 2026.
\newblock URL \url{http://jmlr.org/papers/v27/25-0121.html}.

\bibitem[Villani(2009)]{villani09}
C{\'e}dric Villani.
\newblock \emph{Optimal Transport: Old and New}, volume 338 of \emph{Grundlehren der mathematischen Wissenschaften}.
\newblock Springer-Verlag, Berlin, 2009.
\newblock ISBN 978-3-540-71049-3.

\bibitem[Xi and Niles-Weed(2022)]{xi22}
Jiaqi Xi and Jonathan Niles-Weed.
\newblock Distributional convergence of the sliced {Wasserstein} process.
\newblock In Alice~H. Oh, Alekh Agarwal, Danielle Belgrave, and Kyunghyun Cho, editors, \emph{Advances in Neural Information Processing Systems}, 2022.
\newblock URL \url{https://openreview.net/forum?id=m8YYs8nJF3T}.

\bibitem[Niles-Weed and Berthet(2022)]{weed22}
Jonathan Niles-Weed and Quentin Berthet.
\newblock Minimax estimation of smooth densities in wasserstein distance.
\newblock \emph{The Annals of Statistics}, 50\penalty0 (3):\penalty0 1519--1540, 2022.
\newblock ISSN 00905364, 21688966.
\newblock URL \url{https://www.jstor.org/stable/27172252}.

\bibitem[Deshpande et~al.(2018)Deshpande, Zhang, and Schwing]{deshpande18}
Ishan Deshpande, Ziyu Zhang, and Alexander Schwing.
\newblock Generative modeling using the sliced {Wasserstein} distance.
\newblock In \emph{2018 IEEE/CVF Conference on Computer Vision and Pattern Recognition}, pages 3483--3491, 2018.
\newblock \doi{10.1109/CVPR.2018.00367}.

\bibitem[Wang et~al.(2021)Wang, Gao, and Xie]{wang21twosample}
Jie Wang, Rui Gao, and Yao Xie.
\newblock Two-sample test using projected {Wasserstein} distance.
\newblock In \emph{2021 IEEE International Symposium on Information Theory (ISIT)}, pages 3320--3325. IEEE Press, 2021.
\newblock \doi{10.1109/ISIT45174.2021.9518186}.
\newblock URL \url{https://doi.org/10.1109/ISIT45174.2021.9518186}.

\bibitem[Meunier et~al.(2022)Meunier, Pontil, and Ciliberto]{meunier22}
Dimitri Meunier, Massimiliano Pontil, and Carlo Ciliberto.
\newblock Distribution regression with sliced {W}asserstein kernels.
\newblock In Kamalika Chaudhuri, Stefanie Jegelka, Le~Song, Csaba Szepesvari, Gang Niu, and Sivan Sabato, editors, \emph{Proceedings of the 39th International Conference on Machine Learning}, volume 162 of \emph{Proceedings of Machine Learning Research}, pages 15501--15523. PMLR, 17--23 Jul 2022.
\newblock URL \url{https://proceedings.mlr.press/v162/meunier22b.html}.

\bibitem[Nguyen et~al.(2023)Nguyen, Nguyen, and Ho]{nguyen23selfattention}
Khai Nguyen, Dang Nguyen, and Nhat Ho.
\newblock Self-attention amortized distributional projection optimization for sliced {W}asserstein point-cloud reconstruction.
\newblock In Andreas Krause, Emma Brunskill, Kyunghyun Cho, Barbara Engelhardt, Sivan Sabato, and Jonathan Scarlett, editors, \emph{Proceedings of the 40th International Conference on Machine Learning}, volume 202 of \emph{Proceedings of Machine Learning Research}, pages 26008--26030. PMLR, 23--29 Jul 2023.
\newblock URL \url{https://proceedings.mlr.press/v202/nguyen23e.html}.

\bibitem[Bing et~al.(2024)Bing, Bunea, and Niles-Weed]{bing24estimationinferencewassersteindistance}
Xin Bing, Florentina Bunea, and Jonathan Niles-Weed.
\newblock Estimation and inference for the {Wasserstein} distance between mixing measures in topic models, 2024.
\newblock URL \url{https://arxiv.org/abs/2206.12768}.

\bibitem[Kolouri et~al.(2018)Kolouri, Rohde, and Hoffmann]{kolouri18}
Soheil Kolouri, Gustavo~K. Rohde, and Heiko Hoffmann.
\newblock Sliced {Wasserstein} distance for learning {Gaussian} mixture models.
\newblock In \emph{2018 IEEE/CVF Conference on Computer Vision and Pattern Recognition}, pages 3427--3436, 2018.
\newblock \doi{10.1109/CVPR.2018.00361}.

\bibitem[Fraiman et~al.(2025)Fraiman, Moreno, and Ransford]{fraiman25}
Ricardo Fraiman, Leonardo Moreno, and Thomas Ransford.
\newblock {A {Cram\'er--Wold} theorem for mixtures}.
\newblock \emph{Electronic Communications in Probability}, 30:\penalty0 1 -- 10, 2025.
\newblock \doi{10.1214/25-ECP711}.
\newblock URL \url{https://doi.org/10.1214/25-ECP711}.

\bibitem[Fraiman et~al.(2023)Fraiman, Moreno, and Ransford]{fraiman23}
Ricardo Fraiman, Leonardo Moreno, and Thomas Ransford.
\newblock A {Cram\'er--Wold} theorem for elliptical distributions.
\newblock \emph{J. Multivar. Anal.}, 196\penalty0 (C), July 2023.
\newblock ISSN 0047-259X.
\newblock \doi{10.1016/j.jmva.2023.105176}.
\newblock URL \url{https://doi.org/10.1016/j.jmva.2023.105176}.

\bibitem[Helgason(2010)]{helgason10}
Sigurdur Helgason.
\newblock \emph{Integral Geometry and {Radon} Transforms}.
\newblock Springer, New York, NY, 2010.
\newblock ISBN 978-1-4419-6055-9.
\newblock \doi{10.1007/978-1-4419-6055-9}.

\bibitem[Kak and Slaney(2001)]{kak01}
Avinash~C. Kak and Malcolm Slaney.
\newblock \emph{Principles of Computerized Tomographic Imaging}.
\newblock Society for Industrial and Applied Mathematics, 2001.
\newblock \doi{10.1137/1.9780898719277}.
\newblock URL \url{https://epubs.siam.org/doi/abs/10.1137/1.9780898719277}.

\bibitem[Campisi and Egiazarian(2007)]{campisi07}
Patrizio Campisi and Karen Egiazarian, editors.
\newblock \emph{Blind Image Deconvolution: Theory and Applications}.
\newblock CRC Press, 2007.
\newblock \doi{10.1201/9781420007299}.

\bibitem[Ahmed et~al.(2014)Ahmed, Recht, and Romberg]{ahmed14}
Ali Ahmed, Benjamin Recht, and Justin Romberg.
\newblock Blind deconvolution using convex programming.
\newblock \emph{IEEE Transactions on Information Theory}, 60\penalty0 (3):\penalty0 1711--1732, 2014.
\newblock \doi{10.1109/TIT.2013.2294644}.

\bibitem[Fan(1991)]{fan91}
Jianqing Fan.
\newblock {On the Optimal Rates of Convergence for Nonparametric Deconvolution Problems}.
\newblock \emph{The Annals of Statistics}, 19\penalty0 (3):\penalty0 1257 -- 1272, 1991.
\newblock \doi{10.1214/aos/1176348248}.
\newblock URL \url{https://doi.org/10.1214/aos/1176348248}.

\bibitem[Biau and Devroye(2015)]{biau15}
G{\'e}rard Biau and Luc Devroye.
\newblock \emph{Lectures on the Nearest Neighbor Method}.
\newblock Springer International Publishing, Cham, 2015.
\newblock ISBN 978-3-319-25388-6.
\newblock \doi{10.1007/978-3-319-25388-6}.

\bibitem[Whiteson(2016)]{hepmass}
Daniel Whiteson.
\newblock {HEPMASS}.
\newblock UCI Machine Learning Repository, 2016.

\bibitem[Ghahramani(1996)]{kin}
Zoubin Ghahramani.
\newblock Kin family of datasets.
\newblock {DELVE}, University of Toronto, 1996.
\newblock URL \url{https://www.cs.toronto.edu/~delve/data/kin/desc.html}.

\bibitem[Papamakarios et~al.(2017)Papamakarios, Pavlakou, and Murray]{papamakarios17}
George Papamakarios, Theo Pavlakou, and Iain Murray.
\newblock Masked autoregressive flow for density estimation.
\newblock In I.~Guyon, U.~Von Luxburg, S.~Bengio, H.~Wallach, R.~Fergus, S.~Vishwanathan, and R.~Garnett, editors, \emph{Advances in Neural Information Processing Systems}, volume~30. Curran Associates, Inc., 2017.
\newblock URL \url{https://proceedings.neurips.cc/paper_files/paper/2017/file/6c1da886822c67822bcf3679d04369fa-Paper.pdf}.

\bibitem[Pedregosa et~al.(2011)Pedregosa, Varoquaux, Gramfort, Michel, Thirion, Grisel, Blondel, Prettenhofer, Weiss, Dubourg, Vanderplas, Passos, Cournapeau, Brucher, Perrot, and Duchesnay]{sklearn}
F.~Pedregosa, G.~Varoquaux, A.~Gramfort, V.~Michel, B.~Thirion, O.~Grisel, M.~Blondel, P.~Prettenhofer, R.~Weiss, V.~Dubourg, J.~Vanderplas, A.~Passos, D.~Cournapeau, M.~Brucher, M.~Perrot, and E.~Duchesnay.
\newblock Scikit-learn: Machine learning in {P}ython.
\newblock \emph{Journal of Machine Learning Research}, 12:\penalty0 2825--2830, 2011.

\bibitem[Rozet et~al.(2022)]{zuko}
François Rozet et~al.
\newblock {Zuko}: Normalizing flows in {PyTorch}, 2022.
\newblock URL \url{https://pypi.org/project/zuko}.

\bibitem[Folland(1999)]{folland99}
Gerald~B. Folland.
\newblock \emph{Real analysis: modern techniques and their applications}.
\newblock Pure and applied mathematics. Wiley, 1999.
\newblock ISBN 9780471317166.
\newblock URL \url{http://books.google.com/books?id=uPkYAQAAIAAJ}.

\bibitem[Aliprantis and Border(2006)]{aliprantis06}
Charalambos~D. Aliprantis and Kim~C. Border.
\newblock \emph{Infinite Dimensional Analysis: A Hitchhiker's Guide}.
\newblock Springer, 2006.
\newblock ISBN 978-3-540-29586-0.
\newblock \doi{10.1007/3-540-29587-9}.

\bibitem[DasGupta(2010)]{dasgupta2010}
Anirban DasGupta.
\newblock \emph{Fundamentals of Probability: A First Course}.
\newblock Springer New York, New York, NY, 2010.
\newblock ISBN 978-1-4419-5780-1.
\newblock \doi{10.1007/978-1-4419-5780-1}.
\newblock URL \url{https://doi.org/10.1007/978-1-4419-5780-1}.

\bibitem[Kallenberg(2002)]{fomp}
O.~Kallenberg.
\newblock \emph{Foundations of Modern Probability}.
\newblock Probability and Its Applications. Springer New York, 2002.
\newblock ISBN 9780387953137.
\newblock URL \url{https://books.google.com/books?id=L6fhXh13OyMC}.

\bibitem[Ryan(2002)]{ryan02}
Raymond~A. Ryan.
\newblock \emph{Introduction to Tensor Products of {Banach} Spaces}.
\newblock Springer London, London, 2002.
\newblock ISBN 978-1-4471-3903-4.
\newblock \doi{10.1007/978-1-4471-3903-4}.
\newblock URL \url{https://doi.org/10.1007/978-1-4471-3903-4}.

\bibitem[Szankowski(1981)]{szankowsk81}
Andrzej Szankowski.
\newblock {B(H) does not have the approximation property}.
\newblock \emph{Acta Mathematica}, 147:\penalty0 89 -- 108, 1981.
\newblock \doi{10.1007/BF02392870}.
\newblock URL \url{https://doi.org/10.1007/BF02392870}.

\bibitem[Kelley(1955)]{kelley55}
John~L. Kelley.
\newblock \emph{General topology}.
\newblock D. Van Nostrand Company, Inc., Toronto-New York-London, 1955.

\bibitem[Tu(2010)]{tu10}
L.W. Tu.
\newblock \emph{An Introduction to Manifolds}.
\newblock Universitext. Springer New York, 2010.
\newblock ISBN 9781441973993.
\newblock URL \url{https://books.google.de/books?id=br1KngEACAAJ}.

\end{thebibliography}

    \vfill

    \end{document}